\documentclass[reqno,12pt]{amsart}

\usepackage{defs_KLM_Markov,setspace}
\allowdisplaybreaks
\usepackage{subcaption}

\begin{document}

\title[Multivariate Hawkes processes]{Markovian multivariate Hawkes population processes: Efficient evaluation of moments}

\author{Raviar S. Karim, Roger J.~A. Laeven, and Michel Mandjes}

\begin{abstract}

We provide probabilistic and computational results on Markovian multivariate Hawkes processes and induced population processes.
By applying the Markov property, we characterize in closed form a joint transform, bijective to the probability distribution, of the population process 
and its underlying intensity process.
We demonstrate a method that exploits the transform to obtain analytic expressions for transient and stationary multivariate moments of any order, as well as auto- and cross-covariances.
We reveal a nested sequence of block matrices that yields the moments in explicit form 
and brings important computational advantages.
We also establish the asymptotic behavior of the intensity of the multivariate Hawkes process in its nearly unstable regime, under a specific parameterization.
In extensive numerical experiments, we analyze the computational complexity, accuracy and efficiency of the established results.

\smallskip

\noindent
\textsc{Keywords.} Hawkes processes $\circ$ mutual excitation $\circ$ Markov processes $\circ$ transform analysis $\circ$ moment computations

\smallskip

\noindent
\textsc{Acknowledgements and Affiliations.}
We are very grateful to Jan Magnus for helpful comments.
RK and RL are with the Dept.\ of Quantitative Economics, University of Amsterdam.
RL is also with E\textsc{urandom}, Eindhoven University of Technology, and with C\textsc{ent}ER, Tilburg University;
his research was funded in part by the Netherlands Organization for Scientific Research under grants NWO VIDI 2009 and NWO VICI 2019/20.
MM is with the Korteweg-de Vries Institute for Mathematics, University of Amsterdam. 
MM is also with E\textsc{urandom}, Eindhoven University of Technology,
and with the Amsterdam Business School, University of Amsterdam;
his research was funded in part the Netherlands Organization for Scientific Research under the Gravitation project N\textsc{etworks}, grant 024.002.003.

\noindent
Email: \tt{r.s.karim@uva.nl}, \tt{r.j.a.laeven@uva.nl}, \tt{m.r.h.mandjes@uva.nl}.
\end{abstract}

\maketitle

\newpage

{\small
\tableofcontents}

\newpage

\section{Introduction}

Since the onset of globalization, the mechanisms according to which events spread have become increasingly complex.
Events involving disease contagion, financial panic, news that goes viral, are all subject to forms of propagation that occur through time and space, be it across human populations, equity markets, or media outlets.
A multitude of mathematical models have been proposed to describe the corresponding underlying dynamics.

Multivariate point processes constitute one such class of models that describe the random nature of the arrival, and subsequent spread, of events, in the time as well as space domain.
In particular, the subclass of multivariate Hawkes processes (\cite{H71,H71-2}) provides a rich structure that is capable of capturing contagious dynamics.
These Hawkes processes allow for flexible dependencies of events, due to the inherent feedback mechanisms known as \textit{self}- and \textit{cross-excitation}.

Recently, Hawkes processes have been increasingly used as the arrival process of infinite-server queues (\cite{DP18,KSBM18}). 
Such infinite-server queues can be viewed as population processes, where individuals are born (i.e., arrive) according to a Hawkes process and die (i.e., leave) after a random time. 
A related, relevant motivation to study a Hawkes-fed infinite-server queue is to account for infected and recovered subjects in a population.
Hawkes processes have been widely applied in epidemiology, in conjunction with SIR models (\cite{R18}) and dynamic contagion models, e.g.,\ in the context of the COVID-19 pandemic (\cite{C20}).
Throughout this paper, we use the terminology of population processes and infinite-server queues interchangeably. 

An important strand of research considers a \textit{multivariate} framework to allow for cross-exciting dynamics between subpopulations (for instance, residing at different geographic locations). 
Also in various other domains this type of model offers a natural framework, for example when considering the number of simultaneous online visitors of a specific website or to describe social interaction. 
In the operations management literature, (multivariate) Hawkes processes have gained significant interest over the past decade, e.g.\ in the context of modeling customer support contact centers (\cite{C23,YT22}), in the context of jump propagation in spot and options markets (\cite{DL17,ALP14}), and in relation to online advertising (\cite{XDW14}).

When analyzing infinite-server queues with multivariate Hawkes input, the main challenge lies in unraveling, and tractably representing, its inherently complex probabilistic structure.
In this paper our main objective is to devise computational techniques to efficiently and accurately evaluate moments.
We consider the Markovian case, i.e., the case in which the excitation functions are exponentially decaying; see (\cite{H71,O75}).
The Markov property is used to obtain a characterization of the transform of the joint process, that is, the Hawkes intensity process and the infinite-server queue, in terms of a system of differential equations.

\vspace{2mm}

\noindent \textit{Contributions.}
This paper makes five contributions.
We start by deriving in closed form a joint transform --- bijective to the joint probability distribution  --- of the multivariate population process and its underlying intensity process,
exploiting their Markovian nature.
We allow for general distributions of the intensity jump sizes and allow the processes to be evaluated at potentially multiple future time instances, thus characterizing all cross-sectional and temporal probabilistic features.
Second, we employ this joint transform to obtain explicit, recursive expressions for both the transient and stationary multivariate cross-moments.
Third, we show that the higher-order transient and stationary moments can be obtained in closed form from a nested sequence of block matrices, having important computational advantages.
Fourth, for a specifically chosen set of parameters, we establish the limiting behavior of the underlying intensity process of the multivariate population process in the practically relevant, nearly unstable situation, where the stability condition is close to being violated; cf.\ the \textit{heavy-traffic regime} in queueing theory.
We thus substantially extend earlier results pertaining to the case that the underlying Hawkes process is single-dimensional (\cite{DZ11, DP18,KSBM18,DP19}); cf.\ also \cite{CHY20}.

The final contribution concerns the use of our analytic findings when devising efficient computational techniques for evaluating moments.
To this end, we analyze in numerical experiments the methods developed in this paper and assess their accuracy and efficiency.
The methods we develop show superior performance in terms of speed and accuracy when compared to the computational alternative of applying finite difference schemes to the joint transform.
Our computational approach is fast and accurate: it provides near-instant response that is exact up to machine precision. 
When using the nested block matrices, it has the attractive feature that the computation speed does not depend on the value of the considered time horizon. 
Moreover, when compared to the simulation alternative, a substantial number of Monte Carlo simulation runs is needed to sufficiently precisely approximate the object of interest, so that the performance gain is even orders of magnitude larger.
The computer codes that implement the methods developed in this paper are available at \url{https://github.com/RaviarKarim/HawkesMarkov}.

In the computation of the transient and stationary moments, we focus on two settings, which are in a sense each others dual, namely the bivariate setting with moments of arbitrary order and the multivariate model with moments up to second order.
The generic recursive structure for computing the transient moments resembles the one by which their stationary counterparts can be evaluated, with one crucial difference:
in the transient setting we obtain recursive systems of non-homogeneous linear differential equations, whereas in the stationary setting we obtain recursive systems of linear algebraic equations.

An important application that our results make possible, is moment-based estimation of multivariate Hawkes processes.
Such an estimation approach requires evaluating a collection of moments, auto-covariances and cross-covariances a (very) large number of times, for the parameter vector proposed at each new iteration of the optimization routine. 
Existing approaches to evaluating moments and other distributional characteristics are computationally prohibitively expensive for this purpose.
One might employ the explicit approximate moments, obtained from the infinitesimal Markov generator using operator methods and Taylor expansions applied to short time intervals, as in \cite{ACL15}.
This approach, however, requires data sampled at least at the daily frequency. 
Our approach is based on exact expressions of the moments, auto-covariances and cross-covariances over time intervals of arbitrary length, and their numerical evaluation remains fast and accurate.
Another advantage of our approach is that it facilitates comparative statics, i.e., it provides an efficient, tractable link between the objects of interest and the parameters of the process.

\vspace{2mm}

\noindent \textit{Organization.}
This paper is organized as follows.
In Section~\ref{sec:ModelBranching}, we introduce the Markovian multivariate Hawkes process and the induced population process.
In Section~\ref{sec:TransformAnalysis}, we derive a characterization of the joint transform using the Markov property of the process.
Subsequently, we use this result to obtain relations for the transient and stationary moments.
In Section~\ref{sec: recursive}, we show the underlying recursive structure in the computation of moments of certain order and dimension.
In Section~\ref{sec: nested block matrices}, we focus on the bivariate setting, revealing a nested block structure of matrices that characterize the moments, enabling fast computation.
In Section~\ref{sec: nearly unstable}, we analyze the nearly unstable case for the intensity of the process.
Section~\ref{sec: numerics} provides our numerical experiments.
Conclusions are in Section~\ref{sec: conclusion}. 
Proofs, technical derivations and details, and additional numerical experiments are relegated to six Appendices.

\section{Multivariate Hawkes Populations}\label{sec:ModelBranching}
In this section we define the multivariate Hawkes process, by means of the associated conditional intensity function, as well as its induced population process.
Hawkes processes, first introduced in a series of papers \cite{H71,H71-2,HO74}, are a class of point processes that exhibit self-exciting behavior, in the sense that the current value of the associated intensity function depends on the history of the point process.
In this paper, we focus on multivariate Hawkes processes of the Markovian type.

We consider a $d$-dimensional point process $\bm{N}(\cdot) \equiv (\bm{N}(t))_{t\in\rr_+} = ((N_1(t),\dots,N_d(t)))^\top_{t\in\rr_+}$, which records the number of points in each component $N_i(t)$, with $i\in [d] := \{1,\dots,d\}$, up to and including time $t$.
It is well known that a point process is characterized by its conditional intensity function $\bm{\lambda}(t) = (\lambda_1(t),\dots,\lambda_d(t))^\top$,
see \cite[Chapter 7]{DVJ07}.

\begin{definition} \label{def: multivariate Markov Hawkes process}
A \textrm{Markovian multivariate Hawkes process} is a point process $\bm{N}(\cdot)$, with $\bm{N}(0)=\bm{0}$, whose components $N_i(\cdot)$ satisfy
\begin{align}
\label{eq: def components N_i}
\begin{split}
    \pp(N_i(t+\Delta) - N_i(t) = 0 \, | \, \cF_t ) &= 1 - \lambda_i(t)\Delta + o(\Delta),\\
    \pp(N_i(t+\Delta) - N_i(t) = 1 \, | \, \cF_t ) &= \lambda_i(t)\Delta + o(\Delta), \\
    \pp(N_i(t+\Delta) - N_i(t) > 1 \, | \, \cF_t ) &= o(\Delta),
\end{split}
\end{align}
as $\Delta \downarrow 0$, with $\cF_t = \sigma(\bm{N}(s) \, : \, s \leqslant t)$ denoting the natural filtration generated by $\bm{N}(\cdot)$.
Here, each component $\lambda_i(t)$ of the intensity function satisfies the mean-reverting dynamics
\begin{align}
\label{eq: intensity dynamics definition}
    \ddiff\lambda_i(t) = \alpha_{i}(\overline{\lambda}_i -\lambda_i(t))\,\ddiff t + \sum_{j=1}^d B_{ij}(t) \,\ddiff N_j(t),
\end{align}
where $\lambda_i(0) = \overline{\lambda}_i \geqslant 0$, $\alpha_i \geqslant 0$, and for each $i,j\in[d]$, $(B_{ij}(t))_t$ is a sequence of independent random variables, distributed as the generic non-negative random variable $B_{ij}$.
\end{definition}

The intuitive explanation behind Definition \ref{def: multivariate Markov Hawkes process} goes as follows.
The constants $\overline{\lambda}_i$ are referred to as the \textit{base rates}.
When a point is generated in component $j\in[d]$, $N_j(t)$ increases by one and makes intensity $\lambda_i(t)$, for all $i\in[d]$, jump by a value $B_{ij}(t)$ that is distributed as the random variable $B_{ij}$.
This jump in the intensity $\lambda_i(t)$ caused by $B_{ij}(t)$ increases the probability of new points being generated in component $i$, thereby more likely to increase $N_i(t)$.
After the jump has occurred, the intensity $\lambda_i(t)$ decays exponentially with rate $\alpha_{i}$, in the direction of the base rate $\overline{\lambda}_i$.
These jumps in intensities and the subsequent decay is what makes points cluster across time and space. 
That is, for any $j\in[d]$, when points in $N_j(t)$ cause $\lambda_j(t)$ to jump, there is a pure temporal effect and we speak of \textit{self-excitation}, while an effect on $\lambda_i(t)$, with $i\neq j$, has an additional spatial effect and we speak of \textit{cross-excitation}.

It is well-known that with exponential decay the joint process $(\bm{N}(\cdot),\bm{\lambda}(\cdot))$ is a Markov process, see e.g.\ \cite{LTP15, L09, O75}.
Furthermore, by It\^{o}'s Lemma applied to $f(t,\lambda_i(t)) = e^{\alpha_i t}\lambda_i(t)$, Eqn.~\eqref{eq: intensity dynamics definition} can alternatively be expressed as
\begin{align}
\label{eq: intensity integral definition}
    \lambda_i(t) = \overline{\lambda}_i + \sum_{j=1}^d \int_0^t B_{ij}(s) \,e^{-\alpha_{i}(t-s)}\,\mathrm{d}N_j(s),
\end{align}
where we used that $\bm{N}(0) = \bm{0}$, implying $\lambda_i(0) = \overline{\lambda}_i$.
The exponential function $g_i(t) := e^{-\alpha_i t}$ in Eqn.~\eqref{eq: intensity integral definition} is known in the literature as the \textit{decay function}.
We emphasize that only exponentially decaying  $g_i(\cdot)$ render the joint process $(\bm{N}(\cdot),\bm{\lambda}(\cdot))$ a Markov process.

To ensure stability of the multivariate Hawkes process, \cite{H71-2,BM96} show that we must impose a stability condition.
In what follows, we assume that this stability condition applies.

\begin{assumption} \label{ass: stability condition}
With $\rho(\cdot)$ denoting the spectral radius of a matrix, assume that $\rho(\bm{H}) < 1$, where the matrix $\bm{H} = (h_{ij})_{i,j\in[d]}$ has elements
\begin{align} \label{eq: general stability condition}
    h_{ij} \equiv \ee[B_{ij}] \int_0^{\infty} e^{-\alpha_{i}t}\,\mathrm{d}t = \ee[B_{ij}]/ \alpha_{i}.
\end{align}
\end{assumption}


In the model discussed in this paper, the Markovian multivariate Hawkes process will serve as the arrival process of a multivariate population process. Our overarching goal is to compute various quantities pertaining to the (joint) distribution of the arrival process and the population process. 
A similar setup has been considered in \cite{DP18, KSBM18}, where the univariate Hawkes process serves as the arrival process for an infinite-server queue in which arriving customers reside for independent and identically distributed (i.i.d.) amounts of time (often assumed exponentially distributed).
The following definition generalizes that framework to the multivariate setting. 

\begin{definition}
Let $\bm{N}(\cdot)$ be a Markovian multivariate  Hawkes process as given in Definition~\ref{def: multivariate Markov Hawkes process}.
Define the associated \textrm{Hawkes population process} $\bm{Q}(\cdot)$, with $\bm{Q}(0)=\bm{0}$, by setting, for any component $i$ and $t\geqslant 0$,
\begin{align}
\label{eq: def hawkes population process}
    Q_i(t) :=\int_0^t \bm{1}_{\{E_i(s) > t-s\}}\, \ddiff N_i(s),
\end{align}
where $(E_i(s))_s$ is a sequence of independent random variables, exponentially distributed with parameter $\mu_i$, also independent of the multivariate arrival process ${\bs N}(t)$.
\end{definition}

The above definition entails that each arrival in component $i\in[d]$ remains in the system for an exponentially distributed amount of time.
In demography, one can think of $Q_i(\cdot)$ as the number of individuals in a subpopulation $i\in[d]$, where each $E_i(s)$ models the lifetime of an individual in this subpopulation.
In epidemiology, $Q_i(\cdot)$ would represent the number of individuals infected by a disease at location $i\in[d]$, where $E_i(s)$ would model the duration from infection to recovery (or death).
In queuing terminology, this type of system can be interpreted as a special case of an infinite-server queue.
We note that if $\mu_i \equiv 0$, then no points ever depart from component $i$, and hence ${Q}_i(\cdot) \equiv N_i(\cdot)$.
However, in some of the expressions we will encounter in this paper one must take proper care of taking the limit $\mu_i \downarrow 0$, so as to avoid dividing by zero; note that if $\mu_i=0$, then $Q_i(\cdot)$  eventually grows unbounded.

The application of our methodology to the class of multivariate Hawkes population processes, as performed in the present paper, may be viewed as a ``proof of principle''. 
We would like to stress that our methodology can, in principle, be applied to any general multivariate Markov process.

\section{Transform and Joint Moments}\label{sec:TransformAnalysis}

The objective of this section is to analyze the joint transform of $(\bm{Q}(t),\bm{\lambda}(t))$ for any fixed $t\in\rr_+$, and to use this transform to obtain an algorithm to determine the corresponding (joint) moments.
To this end, define, for given initial values $\bm{Q}(t_0) = \bm{Q}_0 \in \nn_0^d$ and $\bm{\lambda}(t_0)= \bm{\lambda}_0\in\rr_+^d$ for some $0\leqslant t_0<t$, the conditional joint transform
\begin{align}\label{eq: def joint transform Q L condition time t0}
    \zeta_{t_0}(t,\bm{s},\bm{z}) 
    = \ee_{t_0}\big[ \prod_{i=1}^d z_i^{Q_i(t)} e^{-s_i \lambda_i(t)}\big]
    :=\ee\big[ \prod_{i=1}^d z_i^{Q_i(t)} e^{-s_i \lambda_i(t)} \, |\, \bm{Q}(t_0) = \bm{Q}_0, \bm{\lambda}(t_0) = \bm{\lambda}_0 \big],
\end{align}
where $t\geqslant 0$, $\bm{s} \in \rr_+^d$ and $\bm{z} \in [-1,1]^d$.
In the specific case that $t_0=0$, with the assumed initial conditions $\bm{Q}(0) = \bm{0}$ and $\bm{\lambda}(0) = \bm{\overline{\lambda}}$, we define $ \zeta(t,\bm{s},\bm{z})
    \equiv \zeta_0(t,\bm{s},\bm{z})$ as
\begin{align} \label{eq: def joint transform Q L condition time 0}
    \zeta(t,\bm{s},\bm{z})
    :&=\ee\big[ \prod_{i=1}^d z_i^{Q_i(t)} e^{-s_i \lambda_i(t)}\big]
    = \ee\big[ \prod_{i=1}^d z_i^{Q_i(t)} e^{-s_i \lambda_i(t)}\, |\, \bm{Q}(0) = \bm{0}, \bm{\lambda}(0) = \bm{\overline{\lambda}}\big],
\end{align}
where the expectation operator $\ee[\cdot]$ is understood as the conditional $\ee_0[\cdot]$.

\subsection{Transform characterization} \label{sec: characterization}

The following theorem, proven in Appendix~\ref{appendix: proofs}, identifies the joint transform $\zeta_{t_0}(t,\bm{s},\bm{z})$. Define  $\beta_j(\bm{s}) := \ee[e^{-\bm{s}^\top \bm{B}_j}] = \ee[\exp(-\sum_{i=1}^d s_iB_{ij})]$.

\begin{theorem} \label{thm: zeta characterization}
Fix $t\in\rr_+$, and assume $\bm{Q}(t_0) = \bm{Q}_0 \in\nn_0^{d}$ and $\bm{\lambda}(t_0) = \bm{\lambda}_0 \in\rr_+^d$ for some $0 \leqslant t_0 <t$. Then, for any $\bm{z}\in[-1,1]^d$, $\bm{s}\in\rr_+^d$,
\begin{align} \label{eq: joint transform conditioned at t_0}
    \zeta_{t_0}(t,\bm{s},\bm{z})
    &= \prod_{j=1}^d \hat{z}_j(t_0)^{Q_{j,0}} \exp \Big( -\tilde{s}_j(t)\lambda_{j,0} -\overline{\lambda}_j \alpha_j \int_{t_0}^t \tilde{s}_j(u)\ddiff u \Big),
\end{align}
where, for $t_0 \leqslant u \leqslant t$ and $j\in[d]$, the functions $\hat{z}_j(\cdot)$ and $\tilde{s}_j(\cdot)$ satisfy
\begin{align} \label{eq: thm zeta_t0 joint transform eqn for hatZ and tildeS}
\begin{split}
    &\hat{z}_j(u) = 1 + (z_j-1)e^{-\mu_j(t-u)}, \\
    &\frac{\mathrm{d} \tilde{s}_j(u)}{\mathrm{d} u} + \alpha_j \tilde{s}_j(u) + \big(1 + (z_j - 1)e^{-\mu_j(u-t_0)}\big) \beta_j(\tilde{\bm{s}}(u)) - 1 = 0,
\end{split}
\end{align}
with  boundary condition $\tilde{s}_j(t_0) = s_j$.
\end{theorem}

\begin{corollary} \label{cor: zeta_0 characterization}
Fix $t\in\rr_+$, and let $\bm{Q}(0) = \bm{0}$, $\bm{\lambda}(0) = \bm{\overline{\lambda}}$.
Then, for any $\bm{z}\in[-1,1]^d$, $\bm{s}\in\rr_+^d$,
\begin{align}
\label{eq: zeta characterization in terms of s(u)}
    \zeta(t,\bm{s},\bm{z}) = \prod_{j=1}^d \exp\Big( - \overline{\lambda}_j \tilde{s}_j(t)- \overline{\lambda}_j\alpha_j \int_0^t \tilde{s}_j(v)\mathrm{d} v \Big).
\end{align}
Here, for $j\in[d]$ and $v\in[0,t]$, the functions $\tilde{s}_j(\cdot)$ satisfy, with boundary condition $\tilde{s}_j(0) = s_j$,
\begin{align}
\label{eq: thm statement zeta ODE for tilde_s}
    \frac{\mathrm{d} \tilde{s}_j(v)}{\mathrm{d} v} + \alpha_j\tilde{s}_j(v) + (1 + (z_j - 1)e^{-\mu_j v})\beta_j(\tilde{\bm{s}}(v)) - 1 = 0.
\end{align}
\end{corollary}

Clearly, using the expression presented in Theorem~\ref{thm: zeta characterization}, in principle any conditional joint moment can be obtained.
Indeed, for any $n_{\lambda_i}, n_{Q_i} \in \nn_0$, we have
\begin{align}\label{eq: computation of moments by derivatives of JT}
    \frac{\ddiff^{n_{\lambda_1}}}{\ddiff s_1^{n_{\lambda_1}}} \cdots \frac{\ddiff^{n_{\lambda_d}}}{\ddiff s_d^{n_{\lambda_d}}} 
    \frac{\ddiff^{n_{Q_1}}}{\ddiff z_1^{n_{Q_1}}} \cdots
    \frac{\ddiff^{n_{Q_d}}}{\ddiff z_d^{n_{Q_d}}}
    \zeta_{t_0}(t,\bm{s},\bm{z}) \Bigg|_{\substack{{\bs s} = {\bs 0} \\ {\bs z} = {\bs 1}}}
    = \ee_{t_0}\Big[ \prod_{i=1}^d (-1)^{n_{\lambda_i}}\lambda_i(t)^{n_{\lambda_i}} Q_i(t)^{[n_{Q_i}]}\Big],
\end{align}
i.e., an object composed from (standard) moments of $\lambda_i(t)$ and \textit{reduced} moments of $Q_i(t)$.
Note that in \eqref{eq: computation of moments by derivatives of JT} we have used the standard Pochhammer notation: for integers $m$ and $n$ we denote $m^{[n]} := m(m-1)\cdots(m-n +1 )$, by convention setting $m^{[0]} :=1$.

In the above, we focused on identifying transforms pertaining to a single point in time. 
Using similar methods, however, it is possible to derive the transform of $(\bm{Q}(t),\bm{\lambda}(t))$ and $(\bm{Q}(t+\tau),\bm{\lambda}(t+\tau))$ jointly, for some $\tau>0$.
More precisely, the following theorem (proven in Appendix~\ref{appendix: proofs}) identifies, for $\bm{y},\bm{z}\in[-1,1]^d$ and $\bm{r},\bm{s}\in\rr_+^d$, the object
\begin{align} \label{eq: def joint transform Q,lambda at t and t+tau}
    \zeta_\tau(t,\bm{r},\bm{y},\bm{s},\bm{z}) := \ee\Big[ \prod_{i=1}^d y_i^{Q_i(t)}e^{-r_i\lambda_i(t)} z_i^{Q_i(t+\tau)}e^{-s_i\lambda_i(t+\tau)}\Big],
\end{align}
where as before, $\ee[\cdot]$ is understood as $\ee_0[\cdot]$, with $\bm{Q}(0)=\bm{0}$ and $\bm{\lambda}(0)=\bm{\overline{\lambda}}$. 
In addition, ${\bs y}\odot{\bs z}$ is the component-wise product of the vectors ${\bs y}$ and ${\bs z}$. 

\begin{theorem} \label{thm: joint transform characterization t, t+tau}
Fix $t,\tau\in{\mathbb R}_+$, and let $\bm{Q}(0)= \bm{0}$, $\bm{\lambda}(0)=\bm{\overline{\lambda}}$. Then, for any  $\bm{y},\bm{z}\in[-1,1]^d$, $\bm{r},\bm{s} \in \rr_+^d$,
\begin{align}
\begin{split}
    \zeta_\tau(t,\bm{r},\bm{y},\bm{s},\bm{z}) &= \zeta(t, \bm{y}\odot\hat{\bm{z}}(t), \bm{r} + \tilde{\bm{s}}(t+\tau))
    \prod_{j=1}^d \exp\Big(-\overline{\lambda}_j \alpha_j \int_t^{t+\tau} \tilde{s}_j(u)\ddiff u\Big)\\
    &=\prod_{j=1}^d
    \exp\Big(-\overline{\lambda}_j\tilde{r}_j(t) - \overline{\lambda}_j\alpha_j \int_0^t \tilde{r}_j(v)\ddiff v -\overline{\lambda}_j \alpha_j \int_t^{t+\tau} \tilde{s}_j(u)\ddiff u\Big).
\end{split}
\end{align}
Here $\zeta(\cdot)$ is given by Eqn.~\eqref{eq: zeta characterization in terms of s(u)}, and, for $j\in[d]$, the functions $\hat{z}_j(\cdot)$, $\tilde{s}_j(\cdot)$ and $\tilde{r}_j(\cdot)$ satisfy
\begin{align} \label{eq: thm joint transform at t and t+tau equations for hatZ and tildeS}
\begin{split}
    &\hat{z}_j(u) = 1 + (z_j-1)e^{-\mu_j(t+\tau-u)}, \\
    &\frac{\mathrm{d} \tilde{s}_j(u)}{\mathrm{d} u} + \alpha_j \tilde{s}_j(u) + \big(1 + (z_j - 1)e^{-\mu_j(u-t)}\big) \beta_j(\tilde{\bm{s}}(u)) - 1 = 0,\\
    &\frac{\ddiff \tilde{r}_j(v)}{\ddiff v} + \alpha_j \tilde{r}_j(v) + \big(1 + (y_j-1)e^{-\mu_jv} + y_j(z_j-1)e^{-\mu_j(v+\tau)}\big)\beta(\tilde{\bm{r}}(v)) - 1 = 0,\\
\end{split}
\end{align}
with boundary condition $\tilde{s}_j(t) = s_j$ and $\tilde{r}_j(0) = r_j+\tilde{s}_j(t+\tau)$, and where $0\leqslant v \leqslant t$ and $t \leqslant u \leqslant t+\tau$.
\end{theorem}

Observe that the result of Theorem~\ref{thm: joint transform characterization t, t+tau} can be extended to include arbitrarily many time points $t<t_1<t_2<\dots<t_k$, $k\in\nn$, by repeated conditioning and applying Eqn.~\eqref{eq: joint transform conditioned at t_0}.
Further, as in Eqn.~\eqref{eq: computation of moments by derivatives of JT}, we can obtain corresponding joint moments by differentiation.
This in particular allows us to compute the auto-correlation and auto-covariance functions of the multivariate Hawkes process and its associated population process.
More precisely, for any $t\geqslant  0$ and $\tau>0$, we can compute the auto-correlation function by
\begin{align*}
    R_{\bm{Q}}(t,\tau) = \ee\big[ \bm{Q}(t)\,\bm{Q}(t+\tau)^\top\big],\qquad
    R_{\bm{\lambda}}(t,\tau) = \ee\big[\bm{\lambda}(t)\, \bm{\lambda}(t+\tau)^\top\big],
\end{align*}
and the auto-covariance function by
\begin{align*}
    C_{\bm{Q}}(t,\tau) &= \ee\big[ \bm{Q}(t)\,\bm{Q}(t+\tau)^\top\big] - \ee\big[ \bm{Q}(t) \big]\ee\big[ \bm{Q}(t+\tau) \big]^\top, \\
    C_{\bm{\lambda}}(t,\tau) &= \ee\big[\bm{\lambda}(t)\, \bm{\lambda}(t+\tau)^\top\big] - \ee\big[ \bm{\lambda}(t) \big] \ee\big[ \bm{\lambda}(t+\tau) \big]^\top.
\end{align*}

\subsection{Joint moments} \label{sec: joint moments d-dim}
We proceed by exploiting the characterization of the joint transform $\zeta(t,\bm{s},\bm{z})$, as given in Theorem~\ref{thm: zeta characterization}, to derive a system of linear differential equations for the joint transient moments pertaining to $({\bs\lambda}(t),{\bs Q}(t))$, as well as a system of linear (algebraic) equations for the corresponding stationary moments pertaining to $({\bs\lambda},{\bs Q})$, where
\begin{align} \label{eq: def stationary Q and L}
\begin{split}
    \bm{\lambda} &= (\lambda_1,\dots,\lambda_d) := \lim_{t\to\infty} (\lambda_1(t),\dots,\lambda_d(t)),\\
     \bm{Q} &= (Q_1,\dots,Q_d) := \lim_{t\to\infty} (Q_1(t),\dots,Q_d(t)).
\end{split}
\end{align}
As mentioned, Assumption \ref{ass: stability condition} entails that $\bm{\lambda}$ exists, and if in addition all $\mu_i$ are positive, then $\bm{Q}$ is well defined as well.

We start by analyzing the joint transient moments via a system of linear differential equations. Let  $n_{\lambda_i}, n_{Q_i}\in\nn$ be such that  $\sum_{i=1}^d n_{\lambda_i} = n_\lambda$ and $\sum_{i=1}^d n_{Q_i} = n_Q$.
Consider, for any $t\in\rr_+$, with $\bm{n_\lambda} = (n_{\lambda_1},\dots,n_{\lambda_d})$, $\bm{n_Q} = (n_{Q_1},\dots,n_{Q_d})$, the object of our interest:
\begin{align} \label{eq: joint moments L_i Q_j}
    \varphi_t(\bm{n_{\lambda}},\bm{n_Q}) :=\ee\Big[ \prod_{i=1}^d \lambda_i(t)^{n_{\lambda_i}} Q_i(t)^{n_{Q_i}}\Big];
\end{align}
we call $\varphi_t(\bm{n_{\lambda}},\bm{n_Q})$ {\it joint transient moments of total order $n_\lambda$ and $n_Q$}. 
To make sure the objects that we consider are well-defined, we throughout assume that, for any $j\in[d]$,
\begin{align}
    \ee\Big[ \prod_{i=1}^d B_{ij}^{n_{\lambda_i}} \Big] < \infty.
\end{align}
By generalizing the approach of \cite{KSBM18} to the multivariate setting, we obtain a vector-valued ODE (ordinary differential equation) to derive the joint transient moments \eqref{eq: joint moments L_i Q_j}. 
In this derivation, a crucial role is played by an intermediate step in the proof of Theorem~\ref{thm: zeta characterization}, as given in Appendix~\ref{appendix: proofs}. 
In this intermediate step, summarized in Eqn.~\eqref{eq: PDE equation in terms of zeta}, the following PDE has been derived:
\begin{align}
\label{eq: PDE zeta rewritten}
    &\frac{\mathrm{d}}{\mathrm{d}t}\ee\Big[e^{-\bm{s}^\top \bm{\lambda}(t)} \prod_{n=1}^d z_n^{Q_n(t)} \Big]
    - \sum_{j=1}^d \big( \alpha_j s_j + z_j \beta_j(\bm{s}) - 1\big) \ee\big[ \lambda_j(t) e^{-\bm{s}^\top \bm{\lambda}(t)} \prod_{n=1}^d z_n^{Q_n(t)} \big]  \\
    &+ \sum_{j=1}^d\mu_j(z_j-1) \ee\big[ Q_j(t) e^{-\bm{s}^\top \bm{\lambda}(t)} \prod_{n=1}^d z_n^{Q_n(t)-\bm{1}_{\{n=j\}}} \big] = - \sum_{j=1}^d \alpha_j s_j \overline{\lambda}_j\ee\big[e^{-\bm{s}^\top \bm{\lambda}(t)} \prod_{n=1}^d z_n^{Q_n(t)} \big]\notag,
\end{align}
where we rewrote each of the terms appearing in Eqn.~\eqref{eq: PDE equation in terms of zeta} using the definition of $\zeta(t,\bm{s},\bm{z})$.
The next step is to repeatedly differentiate this PDE: differentiate $n_{\lambda_1}, \dots, n_{\lambda_d}$ times with respect to $s_1,\dots,s_d$, respectively and substitute $\bs{s}=\bm{0}$, and then differentiate $n_{Q_1},\dots,n_{Q_d}$ times with respect to $z_1,\dots,z_d$ respectively and substitute $\bs{z}=\bs{1}$.
To make our notation concise, we introduce the `reduced version' of $\varphi_t(\bm{n_Q},\bm{n_{\lambda}})$:
\begin{align}\label{eq: psi concise notation joint moments}
    \psi_t(\bm{n_{\lambda}},\bm{n_Q}) := \ee\Big[ \prod_{i=1}^d \lambda_i(t)^{n_{\lambda_i}} Q_i(t)^{[n_{Q_i}]}\Big],
\end{align}
where evidently  $\psi_t(\bm{n_Q}, \bm{n_{\lambda}})$ can be expressed in objects of the type $\varphi_t(\bm{n_Q},\bm{n_{\lambda}})$ as given in Eqn.~\eqref{eq: joint moments L_i Q_j}, and vice versa.
By elementary algebraic operations, we obtain, with $\sum_{=1}^d m_i=m$,
\begin{align} \label{eq: PDE diff wrt s and z}
    &\frac{\mathrm{d}}{\mathrm{d}t}\psi_t(\bm{n_{\lambda}},\bm{n_Q})
    + \sum_{j=1}^d \big(n_{\lambda_j} (\alpha_j -\ee\big[B_{jj}\big]) + n_{Q_j} \mu_j\big) \psi_t( \bm{n_{\lambda}},\bm{n_Q})\notag \\
    &= \sum_{j=1}^d \sum_{\substack{i=1 \\ i\neq j}}^d  n_{\lambda_i}   \ee\big[B_{ij}\big] \psi_t( \bm{n_{\lambda}} - \bm{e}_i + \bm{e}_j,\bm{n_Q}) + \sum_{j=1}^d n_{Q_j} \psi_t(\bm{n_{\lambda}}+\bm{e_j},\bm{n_Q} - \bm{e}_j)\\
    &\quad  + \sum_{j=1}^d \alpha_j \overline{\lambda}_jn_{\lambda_j}  \psi_t(  \bm{n_{\lambda}}-\bm{e}_j,\bm{n_Q}) + \sum_{i=1}^d \sum_{j=1}^d n_{\lambda_i} n_{Q_j} \ee\big[B_{ij}\big]  \psi_t( \bm{n_{\lambda}} - \bm{e}_i + \bm{e}_j,\bm{n_Q}-\bm{e}_j) \notag\\
    & \quad + \sum_{j=1}^d\sum_{m_1=0}^{n_{\lambda_1}} \cdots \sum_{m_d=0}^{n_{\lambda_d}} \bm{1}_{\{m\leqslant n_\lambda -2\}}   \prod_{k=1}^d {n_{\lambda_k}\choose m_k}\Big\{ n_{Q_j} \prod_{i=1}^d\ee\big[B_{ij}^{n_{\lambda_i} - m_i}\big]
    \psi_t( \bm{m}+\bm{e}_j,\bm{n_Q}-\bm{e}_j)\notag \\
    &\quad \quad +\:\prod_{i=1}^d \ee\big[B_{ij}^{n_{\lambda_i} - m_i}\big] \psi_t( \bm{m}+\bm{e}_j,\bm{n_Q})\Big\}\notag;
\end{align}
see Appendix~\ref{appendix: joint moment compuations} for details.
The key observation is that Eqn.~\eqref{eq: PDE diff wrt s and z} provides us with a relation involving joint (reduced) transient moments of $(\bm{\lambda}(t),\bm{Q}(t))$ of total order $n_\lambda$ and $n_Q$, expressed in terms of their counterparts of at most the same total order, so that the system can be solved.
The resulting linear vector-valued ODE thus enables the computation of the joint transient moments of total order $n_\lambda$ and $n_Q$, thus generalizing Eqn.~(3.9) in \cite{KSBM18}.

After having dealt with the joint transient moments, we now focus on the stationary counterpart of $\psi_t( \bm{n_\lambda},\bm{n_Q})$, i.e., the {\it joint reduced stationary moments of total order $n_\lambda$ and~$n_Q$}:
\begin{align} \label{eq: psi concise notation joint moments stationary}
    \psi( \bm{n_\lambda},\bm{n_Q}) := \lim_{t\to\infty} \psi_t( \bm{n_\lambda},\bm{n_Q}) 
    = \ee\Big[ \prod_{i=1}^d \lambda_i^{n_{\lambda_i}} Q_i^{[n_{Q_i}]} \Big].
\end{align}
In this case, we obtain a system of algebraic equations, to be interpreted as the stationary version of Eqn.~\eqref{eq: PDE diff wrt s and z}. Indeed, it involves joint (reduced) stationary moments of total order $n_\lambda$ and $n_Q$, expressed in terms of their counterparts of at most the same total order:
\begin{align} \label{eq: PDE diff wrt s and z in steady-state}
    &\sum_{j=1}^d \big(n_{\lambda_j} (\alpha_j -\ee\big[B_{jj}\big]) + n_{Q_j} \mu_j\big)\,\psi(\bm{n_{\lambda}},\bm{n_Q} )  \notag \\
    &= \sum_{j=1}^d \sum_{\substack{i=1 \\ i\neq j}}^d  n_{\lambda_i}   \ee\big[B_{ij}\big] \psi( \bm{n_{\lambda}} - \bm{e}_i + \bm{e}_j,\bm{n_Q}) + \sum_{j=1}^d n_{Q_j} \psi( \bm{n_{\lambda}}+\bm{e_j},\bm{n_Q} - \bm{e}_j)\\
    &\quad  + \sum_{j=1}^d \alpha_j \overline{\lambda}_jn_{\lambda_j}  \psi( \bm{n_{\lambda}} -\bm{e}_j,\bm{n_Q}) + \sum_{i=1}^d \sum_{j=1}^d n_{\lambda_i} n_{Q_j} \ee\big[B_{ij}\big]  \psi( \bm{n_{\lambda}} - \bm{e}_i + \bm{e}_j,\bm{n_Q}-\bm{e}_j) \notag\\
    & \quad + \sum_{j=1}^d\sum_{m_1=0}^{n_{\lambda_1}} \cdots \sum_{m_d=0}^{n_{\lambda_d}} \bm{1}_{\{m\leqslant n_\lambda -2\}}   \prod_{k=1}^d {n_{\lambda_k}\choose m_k}\Big\{ n_{Q_j} \prod_{i=1}^d\ee\big[B_{ij}^{n_{\lambda_i} - m_i}\big]
    \psi( \bm{m}+\bm{e}_j,\bm{n_Q}-\bm{e}_j)\notag \\
    &\quad \quad  
    + \prod_{i=1}^d \ee\big[B_{ij}^{n_{\lambda_i} - m_i}\big] \psi( \bm{m}+\bm{e}_j,\bm{n_Q})\Big\}\notag.
\end{align}

In Appendix \ref{appendix: trans_stat_moments} we present an illustration concerning moments of order $n\in\{1,2\}$.

\section{Recursive Procedure: Bivariate Setting}\label{sec: recursive}

In this section, we specifically consider the bivariate setting ($d=2$), and focus on the structure behind the joint moments of arbitrary order $n\in\nn$.
The presented method can be extended to higher dimensions $d\in\nn$, at the cost of heavier notation and more intricate objects.
As such, this section serves as a proof of principle on how the underlying recursive structure can be exploited.
We construct, based on the results obtained in the previous sections, a recursive procedure to compute the joint transient moments $\psi_t(\bm{n_\lambda}, \bm{n_Q})$ as well as the joint stationary moments $\psi(\bm{n_\lambda}, \bm{n_Q})$.
To make the analysis as transparent as possible, we express the main objects in vector/matrix-form.
As it turns out, there is a strong similarity between the structure of the algorithm to evaluate the transient moments on one hand, and its counterpart for the stationary moments on the other hand.
In the sequel we let $n$ be the \textit{total order} of the joint moments, i.e., $n = n_Q + n_\lambda  =  n_{Q_1} + n_{Q_2} + n_{\lambda_1} + n_{\lambda_2}.$
We first rewrite the coupled equations~\eqref{eq: PDE diff wrt s and z} (transient case) and \eqref{eq: PDE diff wrt s and z in steady-state} (stationary case) in vector-matrix form.
We then use these to set up a procedure to compute the corresponding moments.

\subsection{Transient moments} \label{sec: recursive transient moments 2-dim}

We construct a recursive procedure to compute transient joint moments $\psi_t(\bm{n_\lambda}, \bm{n_Q})$ by introducing properly defined vector- and matrix-valued objects, such that we can exploit the ODE in Eqn.~\eqref{eq: PDE diff wrt s and z}.
As we have $d=2$, the objective is to compute
\begin{align*}
    \psi_t((n_{Q_1},n_{Q_2}), (n_{\lambda_1},n_{\lambda_2})) = \ee\Big[ \lambda_1(t)^{n_{\lambda_1}}\lambda_2(t)^{n_{\lambda_2}} Q_1(t)^{[n_{Q_1}]}Q_2(t)^{[n_{Q_2}]}\Big].
\end{align*}
The multivariate setting has the intrinsic complication that the number of combinations of possible joint moments increases rapidly in $n$ and $d$; already in this bivariate setting, there are many possible combinations of joint moments of order $n$.
To collect all joint moments $\psi_t((n_{Q_1},n_{Q_2}), (n_{\lambda_1},n_{\lambda_2}))$ in a single vector, we need to specify an ordering of the different moments.
To that end, we introduce the \textit{stacked} vector
\begin{align} \label{eq: definition stacked vector Psi}
    \bm{\Psi}^{(n)}_t := \Big( \bm{\Psi}_t^{(0,n)}, \bm{\Psi}_t^{(1,n-1)}, \dots,  \bm{\Psi}_t^{(n,0)}\Big)^\top,
\end{align}
where, for each $k\in\{0,1,\dots,n\}$, the vector $\bm{\Psi}^{(k,n-k)}_t$ exhaustively contains all combinations of joint moments such that $n_{Q_1} + n_{Q_2} = k$ and $n_{\lambda_1} + n_{\lambda_2} = n-k$.
For instance, for $k=0$ and $k=n$ we respectively have
\begin{align*}
    \bm{\Psi}^{(0,n)}_t
    &= \Big( \psi_t(\bm{0}, (n,0)), \psi_t(\bm{0},(n-1,1)),\dots,\psi_t(\bm{0}, (0,n))\Big)^\top \\
    &= \Big( \ee\big[\lambda_1(t)^n\big], \ee\big[\lambda_1(t)^{n-1}\lambda_2(t)\big],\dots,\ee\big[\lambda_2(t)^n\big] \Big)^\top;\\
    \bm{\Psi}^{(n,0)}_t
    &= \Big( \psi_t((n,0), \bm{0}), \psi_t((n-1,1), \bm{0}),\dots,\psi_t((0,n), \bm{0})\Big)^\top \\
    &=\Big(\ee\big[Q_1(t)^{[n]}\big], \ee\big[Q_1(t)^{[n-1]}Q_2(t)\big],\dots, \ee\big[Q_2(t)^{[n]}\big]\Big)^\top.
\end{align*}
The cases corresponding with  $k\in\{1,\dots,n-1\}$ are notationally considerably more burdensome since one has to include all possible combinations of order $k$ as well as $n-k$. 

For concrete examples of the stacked vector in~\eqref{eq: definition stacked vector Psi} for orders $n=1,2,3$, see Appendix~\ref{appendix: bivariate recursive procedure}.
It is readily verified that for general $d$ the dimension of $\bm{\Psi}^{(n)}_t$ equals
\begin{align} \label{eq: size psi vector general d}
     \mathfrak{D}(d,n):=\sum_{k=1}^n {2d \choose k} {n-1 \choose k-1},
\end{align}
where the $2d$ is due to the fact that we include moments of both $\bm{Q}(t)$ and $\bm{\lambda}(t)$.
In the $d=2$ case considered in this section, we thus have that the size of $\bm{\Psi}^{(n)}_t$ is $\mathfrak{D}(2,n)$.

By Eqn.~\eqref{eq: PDE diff wrt s and z}, the stacked vector $\bm{\Psi}^{(n)}_t$  satisfies a vector-valued ODE:
\begin{align} \label{eq: ODE stacked vector Psi}
    \frac{\ddiff}{\ddiff t}\bm{\Psi}^{(n)}_t 
    = \bm{M} \bm{\Psi}^{(n)}_t +\bm{L}\big(\bm{\Psi}^{(1)}_t,\dots, \bm{\Psi}^{(n-1)}_t\big)^\top,
\end{align}
for certain matrices $\bm{M}$ and $\bm{L}$ of appropriate dimension. Here, the matrix ${\bs M}$ is of dimension $\mathfrak{D}(2,n)\times\mathfrak{D}(2,n)$, and $\bm{L}$ of dimension $\mathfrak{D}(2,n)\times\overline{\mathfrak{D}}(2,n)$,
where
\begin{align}
    \overline{\mathfrak{D}}(2,n):=\sum_{m=1}^{n-1}{\mathfrak{D}}(2,m).
\end{align}
As a next step, we identify blocks of $\bm{M}$ that correspond to subvectors $\bm{\Psi}^{(k,n-k)}_t$ of the stacked vector $\bm{\Psi}^{(n)}_t$.
Upon inspecting \eqref{eq: PDE diff wrt s and z} we observe that, when considering in \eqref{eq: ODE stacked vector Psi} the differential equations that correspond to $\frac {\ddiff}{\ddiff t} \bm{\Psi}^{(k,n-k)}_t$, in the right-hand side only $\bm{\Psi}^{(k,n-k)}_t$ and $\bm{\Psi}^{(k-1,n-k+1)}_t$ appear, besides a linear combination of objects of lower total order (i.e., $\bm{\Psi}^{(1)}_t,\ldots,\bm{\Psi}^{(n-1)}_t$). 
As a consequence, we can write
\begin{align}\label{eq: ODE indiv vectors Psi}
\begin{split}
    \frac{\ddiff}{\ddiff t} \bm{\Psi}^{(k,n-k)}_t 
    &= \bm{M}^{(k,n-k)} \bm{\Psi}^{(k,n-k)}_t + \bm{K}^{(k,n-k)}\bm{\Psi}^{(k-1,n-k+1)}_t+\bm{L}^{(k,n-k)}\big(\bm{\Psi}^{(1)}_t,\dots, \bm{\Psi}^{(n-1)}_t\big)^\top,
\end{split}
\end{align}
for appropriately chosen matrices $\bm{M}^{(k,n-k)}$, $\bm{K}^{(k,n-k)}$, and $\bm{L}^{(k,n-k)}$, where we set $\bm{K}^{(k,n-k)} \equiv 0$ when $k=0$.
Eqn.~\eqref{eq: ODE indiv vectors Psi} thus reveals a recursive procedure to compute $\bm{\Psi}^{(n)}_t$, where in the $n$-th iteration a non-homogeneous linear system of ODEs has to be solved, with $\bm{\Psi}^{(1)}_t,\dots, \bm{\Psi}^{(n-1)}_t$, as derived in the previous steps, appearing in the non-homogeneous part.

We proceed by introducing the notation needed to set up the recursive procedure.
A \textit{tridiagonal} matrix in $\rr^{n\times n}$ is a matrix with elements on the main diagonal, the first diagonal above and below the main diagonal only, for which we use the notation, with ${\bs a}, {\bs c}\in{\mathbb R}^{n-1}$ and ${\bs d}\in{\mathbb R}^n$, 
\begin{align*}
    \tridiag{\bm{a}}{\bm{d}}{\bm{c}}
    :=
    \begin{bmatrix}
    d_1 & c_1 & 0 & \cdots & 0 \\
    a_1 & d_2 & c_2 & \cdots & 0 \\
    \vdots & \ddots & \ddots & \ddots & \vdots \\
    0 & 0 & a_{n-2} & d_{n-1} & c_{n-1} \\
    0 & 0 & 0 & a_{n-1} & d_n
    \end{bmatrix}.
\end{align*}
Given the vectors $\bm{n_Q} = (n_{Q_1},n_{Q_2})$ and  $\bm{n_\lambda}= (n_{\lambda_1},n_{\lambda_2})$, we set
\begin{align*}
    v(\bm{n_Q}, \bm{n_\lambda}) &:= -n_{\lambda_1} \overline{\alpha}_1 - n_{Q_1} \mu_1  - n_{\lambda_2} \overline{\alpha}_2 - n_{Q_2} \mu_2, \\
    \bm{w}(\bm{n_Q},n) &:= \big( v(\bm{n_Q},(n,0)), v(\bm{n_Q}, (n-1,1)), \dots, v(\bm{n_Q},(1,n-1)), v(\bm{n_Q},(0,n))\big)^\top,
\end{align*}
where $\overline{\alpha}_i := \alpha_i -\ee[B_{ii}]$ for $i=1,2$, corresponding to the left-hand side of Eqn.~\eqref{eq: PDE diff wrt s and z}. 
With the vectors
$\bm{1}_{(n)} := (1,2,\dots,n)$ and $\bm{1}^{(n)} := (n,n-1,\dots,1)$,
and with $e_j$ the unit vector with $1$ on the $j-$th component,
we finally define the matrix, for each $k\in\{0,1,\dots,n-1,n\}$,
\begin{align} \label{eq: def tridiag matrix M}
    \bm{M}^{(k,n-k)} := \bigoplus_{m=0}^k \tridiag{\bm{1}_{(n-k)}\ee[B_{21}]}{\bm{w}\big((n-m)\bm{e}_1 +m\bm{e}_2,n-k\big)}{\bm{1}^{(n-k)}\ee[B_{12}]},
\end{align}
where $\bigoplus$ denotes the direct sum for matrices.
For concrete examples of the matrix $\bm{M}^{(k,n-k)}$ and how it appears in the ODE in~\eqref{eq: ODE indiv vectors Psi}, see Appendix~\ref{appendix: explicit bivariate}.
We can now present our algorithm to compute the transient moments.

\begin{algorithm}
\label{alg: recursive procedure ODE blocks}
Fix $n\in{\mathbb N}$, and suppose we know  $\bm{\Psi}^{(1)}_t,\dots, \bm{\Psi}^{(n-1)}_t$.
Then $\bm{\Psi}^{(n)}_t$ can be found from  Eqn.~\eqref{eq: ODE stacked vector Psi}, as follows:

\begin{enumerate}
\item[Step]0: The vector $\bm{\Psi}^{(0,n)}_t$ satisfies the vector-valued ODE
\begin{align}
    \frac{\ddiff}{\ddiff t} \bm{\Psi}^{(0,n)}_t 
    &= \bm{M}^{(0,n)} \bm{\Psi}^{(0,n)}_t  + \bm{L}^{(0,n)}\big(\bm{\Psi}^{(1)}_t,\dots, \bm{\Psi}^{(n-1)}_t\big),
\end{align}
where $\bm{L}^{(0,n)}$ follows from Eqn.~\eqref{eq: PDE diff wrt s and z}, with initial condition $\bm{\Psi}^{(0,n)}_0$ determined by $\bm{Q}(0)$ and $\bm{\lambda}(0)$. 

\item[Step]$k$: For any $k\in\{1,2,\dots,n-1,n\}$, $\bm{\Psi}^{(k,n-k)}_t$ satisfies the vector-valued ODE
\begin{align} \label{eq: ODE transient step m, 2-dim}
    \frac{\ddiff}{\ddiff t} \bm{\Psi}^{(k,n-k)}_t 
    &=  \bm{M}^{(k,n-k)}\bm{\Psi}^{(k,n-k)}_t + \bm{K}^{(k,n-k)}\bm{\Psi}^{(k-1,n-k+1)}_t  \\
    &\quad\quad + \bm{L}^{(k,n-k)}\big(\bm{\Psi}^{(1)}_t,\dots, \bm{\Psi}^{(n-1)}_t\big), \notag
\end{align}
where $\bm{K}^{(k,n-k)}$ and $\bm{L}^{(k,n-k)}$ follow from Eqn.~\eqref{eq: PDE diff wrt s and z}, with initial condition $\bm{\Psi}^{(k,n-k)}_0$.
\end{enumerate}
\end{algorithm}

Clearly, Eqn.~\eqref{eq: PDE diff wrt s and z} uniquely defines the matrices $\bm{K}^{(k,n-k)}$ and $\bm{L}^{(k,n-k)}$ needed in the above algorithm.
However, their explicit definition would require objects that are even more notationally involved. 
In Appendix~\ref{appendix: explicit bivariate} we show that for moments of orders $n=1$ and $n=2$, we can still explicitly write down the matrix for the stacked vector ODE: we combine the blocks of matrices into $4\times4$ and $10\times10$ matrices $\bm{M}$, respectively, and also construct the corresponding matrix $\bm{L}$.
However, for order $n=3,4,\ldots$, we would need very large matrices which are cumbersome to write down explicitly.

Due to the direct sum structure of $\bm{M}^{(k,n-k)}$, it consists of blocks. 
This allows us to decompose the $k$-th step in the algorithm into smaller steps, by considering the parts of the vector $\bm{\Psi}_t^{(k,n-k)}$ associated with the individual blocks of the matrix.
The solution to the ODEs in Algorithm~\ref{alg: recursive procedure ODE blocks} can be given in terms of matrix exponentials, as follows. This result follows by observing that computing $\frac{\ddiff}{\ddiff t} \bm{\Psi}_t^{(k,n-k)}$ in Eqn.~\eqref{eq: general ODE solution transient moments order n} immediately yields \eqref{eq: ODE transient step m, 2-dim} by inspection and Leibniz' integral rule.

\begin{proposition} \label{prop: solution vector ODE with matrix exponential}
For fixed $t\in\rr_+$, $n\in\nn$ and $k=0,1,\dots,n$, the solution for the vector-valued ODE for $\bm{\Psi}_t^{(k,n-k)}$ in Eqn.~\eqref{eq: ODE transient step m, 2-dim} is given by
\begin{align} \label{eq: general ODE solution transient moments order n}
    \bm{\Psi}_t^{(k,n-k)} &= e^{t\bm{M}^{(k,n-k)}} \bm{\Psi}_0^{(k,n-k)} \\
    &\quad + \int_0^t e^{(t-s)\bm{M}^{(k,n-k)}}\Big(\bm{K}^{(k,n-k)}\bm{\Psi}^{(k-1,n-k+1)}_s  + \bm{L}^{(k,n-k)}\big(\bm{\Psi}^{(1)}_s,\dots, \bm{\Psi}^{(n-1)}_s\big)\Big)\,\ddiff s. \notag
\end{align}
\end{proposition}

\subsection{Stationary moments} \label{sec: recursive stationary moments 2-dim}

In this subsection, we focus on, with $n = n_Q +  n_\lambda$,
\begin{align}
    \psi((n_{Q_1},n_{Q_2}),(n_{\lambda_1},n_{\lambda_2})) = 
    \ee\Big[ \lambda_1^{n_{\lambda_1}}\lambda_2^{n_{\lambda_2}} Q_1^{[n_{Q_1}]}Q_2^{[n_{Q_2}]} \Big].
\end{align}
By exploiting Eqn.~\eqref{eq: PDE diff wrt s and z in steady-state}, we develop a recursive procedure similar to the one for the transient moments. The central object of study is 
\begin{align}
    \bm{\Psi}^{(n)} :=\lim_{t\to\infty} \bm{\Psi}^{(n)}_t = \Big( \bm{\Psi}^{(0,n)}, \bm{\Psi}^{(1,n-1)}, \dots, \bm{\Psi}^{(n,0)}\Big)^\top.
\end{align}
For the following recursive procedure, we use the notation introduced in Section \ref{sec: recursive transient moments 2-dim}; recall in particular the matrices defined in Eqn.~\eqref{eq: def tridiag matrix M}. 
Observe the strong similarity between the Algorithms \ref{alg: recursive procedure ODE blocks} and \ref{alg: recursive procedure stationary blocks}, in the sense that the underlying recursive structures fully match.

\begin{algorithm} \label{alg: recursive procedure stationary blocks}
Fix $n\in{\mathbb N}$, and suppose we know  $\bm{\Psi}^{(1)},\dots, \bm{\Psi}^{(n-1)}$.
Then $\bm{\Psi}^{(n)}$ can be computed as follows:

\begin{enumerate}
\item[Step]0: The vector $\bm{\Psi}^{(0,n)}$ satisfies the linear equation
\begin{align}
    0 = \bm{M}^{(0,n)}\,\bm{\Psi}^{(0,n)} + \bm{L}^{(0,n)}\big(\bm{\Psi}^{(1)},\dots, \bm{\Psi}^{(n-1)}\big)^\top,
\end{align}
where $\bm{L}^{(0,n)}$ follows from Eqn.~\eqref{eq: PDE diff wrt s and z in steady-state}.

\item[Step]$k$: For any $k\in\{1,2,\dots,n-1,n\}$, $\bm{\Psi}^{(k,n-k)}$ satisfies the linear equations
\begin{align}\label{eq: linear equation stationary step m, 2-dim}
\begin{split}
    0
    =\:&  \bm{M}^{(k,n-k)}\,\bm{\Psi}^{(k,n-k)} +  \bm{K}^{(k,n-k)}\bm{\Psi}^{(k-1,n-k+1)} +\bm{L}^{(k,n-k)}\big(\bm{\Psi}^{(1)},\dots, \bm{\Psi}^{(n-1)}\big)^\top.
    \end{split}
\end{align}
where $\bm{K}^{(k,n-k)}$ and $\bm{L}^{(k,n-k)}$ follow from Eqn.~\eqref{eq: PDE diff wrt s and z in steady-state}.
\end{enumerate}
\end{algorithm}

The solution to the linear equations in Algorithm~\ref{alg: recursive procedure stationary blocks} is given in the following proposition. Its proof follows immediately from solving Eqn.~\eqref{eq: linear equation stationary step m, 2-dim}. 

\begin{proposition}
For $n\in\nn$ and $k=0,1,\dots,n-1,n$, the solution for the linear equation for $\bm{\Psi}^{(k,n-k)}$ in Eqn.~\eqref{eq: linear equation stationary step m, 2-dim} is given by
\begin{align}
    \bm{\Psi}^{(k,n-k)}&=  -\Big(\bm{M}^{(k,n-k)}\Big)^{-1}\Big\{ \bm{K}^{(k,n-k)}\bm{\Psi}^{(k-1,n-k+1)} \:+\bm{L}^{(k,n-k)}\big(\bm{\Psi}^{(1)},\dots, \bm{\Psi}^{(n-1)}\big)^\top\Big\}, \notag
\end{align}
\end{proposition}

\section{Nested Block Matrices: Bivariate Setting} \label{sec: nested block matrices}

In this section, we investigate the nested structure of the matrices associated with the ODEs of the moments more thoroughly, again in the bivariate setting $d=2$ (but, as before, extension to higher $d$ is in principle possible), to develop an efficient computational method for evaluating the transient and stationary moments.
It turns out that one can find a nested sequence of well-behaved matrices that describe the relations between the moments, which facilitates a substantial reduction of the required computational effort.
The approach followed in this section can be seen as a bivariate version of Section 3.2 in \cite{DP19}, providing the structure of ODEs associated with the transient moments using lower triangular matrices with scalar entries.
The key difference, however, is that in our case they are replaced by block lower triangular matrices, containing matrix entries.

We first revisit the transient moments of $\bm{\lambda}(t)$ for a fixed $t\in\rr_+$ so as to illustrate the nested structure of the matrices.
After that, we consider the joint transient moments of $(\bm{Q}(t), \bm{\lambda}(t))$, which has a similar but more complex nested structure.
To motivate our analysis, consider the ODEs associated with the vectors $\bm{\Psi}_t^{(0,1)}$ and $\bm{\Psi}_t^{(0,2)}$ containing the first and second order moments of $\bm{\lambda}(t)$ (see Eqns.~\eqref{eq: ODE lambda transient order 1, 2-dim} and~\eqref{eq: ODE lambda transient order 2, 2-dim} in Appendix~\ref{appendix: bivariate recursive procedure}).
Observe that in stacked form, they can be represented in a block lower triangular matrix structure:
\begin{align} \label{eq:stack of stacked Psi lambda ODE}
\begin{split}
    \frac{\ddiff}{\ddiff t} \begin{bmatrix} \bm{\Psi}^{(0,1)}_t \\ \bm{\Psi}^{(0,2)}_t\end{bmatrix} &=
    \begin{bmatrix}
    \bm{A}_{1}^{2\times 2} & \bm{0}_{2}^{2\times 3} \\
    \bm{D}_{2}^{3\times 2} & \bm{C}_{2}^{3\times 3}
    \end{bmatrix}
    \begin{bmatrix} \bm{\Psi}^{(0,1)}_t \\\bm{\Psi}^{(0,2)}_t\end{bmatrix}
    +
    \begin{bmatrix}
    \bm{b}^{2\times 1}  \\
    \bm{0}^{3\times 1}
    \end{bmatrix},
\end{split}
\end{align}
with
$\bm{A}_{1}^{2\times 2} = \bm{M}^{(0,1)}$, $\bm{C}_2^{3 \times 3} = \bm{M}^{(0,2)}$, $\bm{D}_{2}^{3\times 2} = \bm{L}^{(0,2)}$, where the superscripts denote the dimensionality of the matrices, and where $\bm{M}^{(0,1)}$ and $\bm{M}^{(0,2)}$ are defined in Eqn.~\eqref{eq: def tridiag matrix M} and $\bm{L}^{(0,2)}$ in Eqn.~\eqref{eq: ODE lambda transient order 2, 2-dim}.
The notation $\bm{0}^{k\times l} \in \rr^{k\times l}$ represents the all-zeros matrix and $\bm{b}^{2\times 1} = (\alpha_1\overline{\lambda}_1, \alpha_2\overline{\lambda}_2)^\top$.
Observe that this stacked form contains the previously defined matrices as blocks.
A similar form occurs for higher orders $n$ of the stacked vector $(\bm{\Psi}^{(0,1)}_t,\ldots,\bm{\Psi}^{(0,n)}_t)^\top$, containing mixed moments of $\bm{\lambda}(t)$ up to order $n$.

Careful inspection of previous results reveals a \textit{nested sequence} of \textit{block lower triangular matrices}.
Let $m_n := n+1$ and define $\mathfrak{m}_n := m_1 + \dots + m_n$.
Then, consider a nested sequence of block lower triangular matrices $\{\bm{A}_{n}^{\mathfrak{m}_{n}\times\mathfrak{m}_{n}}\}_{n\in\nn}$ defined by
\begin{align} 
\bm{A}_{n}^{\mathfrak{m}_{n}\times\mathfrak{m}_{n}}=
\begin{bmatrix}
    \bm{A}_{n-1}^{\mathfrak{m}_{n-1}\times\mathfrak{m}_{n-1}} & \bm{0}_{n}^{\mathfrak{m}_{n-1}\times m_{n}} \\
    \bm{D}_{n}^{m_{n}\times \mathfrak{m}_{n-1}} & \bm{C}_{n}^{m_{n}\times m_{n}}
\end{bmatrix},
\label{eq:blockMatr}
\end{align}
where $\bm{A}_{1}^{2\times 2}\equiv \bm{C}_{1}^{2\times 2} = \bm{M}^{(0,1)}$, $\bm{C}_{n}^{m_{n}\times m_{n}} = \bm{M}^{(0,n)}$ and $\bm{D}_{n}^{m_{n}\times \mathfrak{m}_{n-1}} = \bm{L}^{(0,n)}$.
Clearly, the first term on the right-hand side of~\eqref{eq:stack of stacked Psi lambda ODE} occurs as a special case of~\eqref{eq:blockMatr} when $n=2$.
Recall that we know the structure of the matrices $\bm{M}^{(0,n)}$ as given in Eqn.~\eqref{eq: def tridiag matrix M}.
In Appendix~\ref{appendix: higher order moments 2-dim}, we give some further details on the structure of $\bm{L}^{(0,n)}$ for the case $n=3$.
The sequence of matrices $\{\bm{A}_{n}^{\mathfrak{m}_{n}\times\mathfrak{m}_{n}}\}_{n\in\nn}$, as defined in~\eqref{eq:blockMatr}, has been chosen such that
\begin{align}
\begin{split}
    \frac{\ddiff}{\ddiff t} \begin{bmatrix} \bm{\Psi}^{(0,1)}_t \\ 
    \vdots \\
    \bm{\Psi}^{(0,n)}_t\end{bmatrix} &=
    \bm{A}_{n}^{\mathfrak{m}_{n}\times\mathfrak{m}_{n}}
    \begin{bmatrix} \bm{\Psi}^{(0,1)}_t \\ 
    \vdots \\
    \bm{\Psi}^{(0,n)}_t\end{bmatrix}
    +
    \begin{bmatrix}
    \bm{b}^{2\times 1}  \\
    \bm{0}^{(\mathfrak{m}_{n}-m_{1})\times 1}
    \end{bmatrix},
\end{split}
\label{eq:stacked n-dim lambda}
\end{align}
with initial condition $(\bm{\Psi}^{(0,1)}_0, \dots, \bm{\Psi}^{(0,n)}_0)^\top$.

The following proposition, providing an explicit expression for $(\bm{\Psi}^{(0,1)}_t,\ldots,\bm{\Psi}^{(0,n)}_t)^\top$, follows directly by noting that taking the time derivative of Eqn.~\eqref{eq:stacked n-dim lambda solution} immediately yields Eqn.~\eqref{eq:stacked n-dim lambda}.

\begin{proposition} \label{prop: stacked n-dim lambda solution}
If $\bm{C}_{i}^{m_{i}\times m_{i}}$ is invertible for all $i\in\{1,\ldots,n\}$,
then 
\begin{align} \label{eq:stacked n-dim lambda solution}
\begin{split}
\begin{bmatrix} \bm{\Psi}^{(0,1)}_t \\ 
    \vdots \\
    \bm{\Psi}^{(0,n)}_t\end{bmatrix}
    &= e^{\bm{A}_{n}^{\mathfrak{m}_{n}\times\mathfrak{m}_{n}}t}
    \begin{bmatrix} \bm{\Psi}^{(0,1)}_0 \\ 
    \vdots \\
    \bm{\Psi}^{(0,n)}_0\end{bmatrix}
    -\left(\bm{A}_{n}^{\mathfrak{m}_{n}\times\mathfrak{m}_{n}}\right)^{-1}\left(\bm{I}^{\mathfrak{m}_{n}\times\mathfrak{m}_{n}} - e^{\bm{A}_{n}^{\mathfrak{m}_{n}\times\mathfrak{m}_{n}}t}\right)\begin{bmatrix}
    \bm{b}^{2\times 1}  \\
    \bm{0}^{(\mathfrak{m}_{n}-m_{1})\times 1}
    \end{bmatrix}.
\end{split}
\end{align}    
\end{proposition}

Proposition \ref{prop: stacked n-dim lambda solution} allows for the simultaneous computation of the first $n$ transient moments of $\bm{\lambda}(t) = (\lambda_1(t),\lambda_2(t))^\top$, but it requires the computation of the matrix exponential and the inverse of $\bm{A}_n^{\mathfrak{m}_{n}\times\mathfrak{m}_{n}}$.

This idea can be extended to the joint transient moments of $(\bm{Q}(t), \bm{\lambda}(t))$.
As before, we show the details of the first and second order, and point out how this extends to higher order moments.
For the first order moments, close inspection of the associated ODEs (given in Eqns.~\eqref{eq: ODE lambda transient order 1, 2-dim} and \eqref{eq: ODE Q transient order 1, 2-dim} in Appendix~\ref{appendix: bivariate recursive procedure}) yields that
\begin{align} \label{eq: ODE stack of stacked Psi vector, QL joint first order moments}
\begin{split}
    &\frac{\ddiff}{\ddiff t} \bm{\Psi}_t^{(1)} 
    = \bm{F}_1^{4\times4} \bm{\Psi}_t^{(1)} 
    + \begin{bmatrix} \bm{b}^{2\times1} \\ \bm{0}^{2\times1} \end{bmatrix}
\end{split},\:\:\mbox{with}\:\:
    \bm{F}_1^{4\times4} = \begin{bmatrix}
    \bm{M}^{(0,1)} & \bm{0}^{2\times2} \\
    \bm{I}^{2\times2} & \bm{M}^{(1,0)}
    \end{bmatrix}, \quad
    \bm{b}_1^{2\times1} = \begin{bmatrix}
    \alpha_1\overline{\lambda}_1 \\
    \alpha_2\overline{\lambda}_2
    \end{bmatrix}.
\end{align}
As before, the superscripts denote the dimensionality of the matrices, with $\bm{I}^{k\times l} \in\rr^{k\times l}$ the identity matrix.
Note the block lower triangular shape of $\bm{F}_1^{4\times4}$.

For the second order moments, we can infer from the associated ODEs (given in Eqns.~\eqref{eq: ODE lambda transient order 2, 2-dim}, \eqref{eq: ODE Q + lambda transient order 2, 2-dim} and \eqref{eq: ODE Q transient order 2, 2-dim} in Appendix~\ref{appendix: bivariate recursive procedure}), in combination with Eqn.~\eqref{eq: ODE stack of stacked Psi vector, QL joint first order moments}, that
\begin{align}\label{eq: ODE stack of stacked Psi vector, QL joint second order moments}
    &\frac{\ddiff}{\ddiff t} 
    \begin{bmatrix}
    \bm{\Psi}_t^{(1)}  \\
    \bm{\Psi}_t^{(2)} 
    \end{bmatrix}
    = \bm{F}_2^{14\times14}
        \begin{bmatrix}
    \bm{\Psi}_t^{(1)}  \\
    \bm{\Psi}_t^{(2)} 
    \end{bmatrix}
    + \begin{bmatrix}
    \bm{b}_1^{2\times1} \\
    \bm{0}^{12\times 1}
    \end{bmatrix},\:\:\mbox{with}\:\:\:
    \bm{F}_2^{14\times14}
    = \begin{bmatrix}
    \bm{F}_1^{4\times4} & \bm{0}^{4\times10} \\
    \bm{G}_2^{10\times4} & \bm{H}_2^{10\times10},
    \end{bmatrix};
\end{align}
here $\bm{F}_2^{14\times 14}$ is a lower triangular matrix with the matrices contained in it defined by
\begin{align*}
    \bm{G}_2^{10\times 4}
    &= \begin{bmatrix}
    \bm{L}^{(0,2)} & \bm{0}^{3\times2} \\
    \bm{L}_{\lambda}^{(1,1)} & \bm{L}_{Q}^{(1,1)} \\
    \bm{0}^{3\times2} & \bm{0}^{3\times 2}
    \end{bmatrix},
    \quad \bm{H}_2^{10\times10}
    = \begin{bmatrix}
    \bm{M}^{(0,2)} & \bm{0}^{3\times 4} & \bm{0}^{3\times 3} \\
    \bm{K}^{(1,1)} & \bm{M}^{(1,1)} & \bm{0}^{3 \times 3} \\
    \bm{0}^{3\times3} & \bm{K}^{(2,0)} & \bm{M}^{(2,0)}
    \end{bmatrix}, \\
    \bm{L}_{\lambda}^{(1,1)}
    &=\begin{bmatrix}
    \ee[B_{11}] & 0 \\
    \ee[B_{21}] & 0 \\
    0 & \ee[B_{12}] \\
    0 & \ee[B_{22}]
    \end{bmatrix},
    \quad \bm{L}_Q^{(1,1)} 
    = \begin{bmatrix}
    \alpha_1\overline{\lambda}_1 & 0 \\
    \alpha_2\overline{\lambda}_2 & 0 \\
    0 & \alpha_1\overline{\lambda}_1 \\
    0 & \alpha_2\overline{\lambda}_2 \\
    \end{bmatrix},
\end{align*}
where $\bm{L}^{(0,2)}$ and the matrices $\bm{K}^{(1,1)}$ and $\bm{K}^{(2,0)}$ are known, see Appendix~\ref{appendix: bivariate recursive procedure}.

Continuing in this fashion we can consider vectors of arbitrary length $n\in\nn$.
Recall that we know the dimension $p_n\equiv \mathfrak{D}(2,n)$ of the vector ${\Psi}_t^{(n)}$ from Eqn.~\eqref{eq: size psi vector general d}, which also yields the dimension $\mathfrak{p}_n= p_1+\cdots + p_n$ of the stacked vector 
$(\bm{\Psi}_t^{(1)}, \bm{\Psi}_t^{(2)},\dots,\bm{\Psi}_t^{(n)})^\top$.
Consider the sequence of matrices $\{\bm{F}_n^{\mathfrak{p}_n\times\mathfrak{p}_n}\}_{n\in\nn}$ 
given by
\begin{align*}
    \bm{F}_n^{\mathfrak{p}_n\times\mathfrak{p}_n}
    =\begin{bmatrix}
    \bm{F}_{n-1}^{\mathfrak{p}_{n-1}\times\mathfrak{p}_{n-1}} & \bm{0}^{\mathfrak{p}_{n-1}\times p_n} \\
    \bm{G}_n^{p_n\times\mathfrak{p}_{n-1}} & \bm{H}_n^{p_n\times p_n}
    \end{bmatrix}
\!,\:
\mbox{with}\:\:\:
    \bm{H}_n^{p_n \times p_n}
    &= \begin{bmatrix}
    \bm{M}^{(0,n)} & \bm{0}  & \cdots & \bm{0} \\
    \bm{K}^{(1,n-1)} & \bm{M}^{(1,n-1)} & \cdots & \bm{0} \\
    \vdots & \ddots & \ddots & \vdots \\
    \bm{0} & \cdots & \bm{K}^{(n,0)} & \bm{M}^{(n,0)}
    \end{bmatrix}\!.
\end{align*}
The matrices $\bm{G}_n^{p_n\times\mathfrak{p}_{n-1}}$ are not as elegantly expressed for general $n\in\nn$, but can be explicitly obtained through Eqn.~\eqref{eq: PDE diff wrt s and z}.
We have used these matrices for $n=1,2$ to compute explicit moments in Appendix~\ref{appendix: transient moments 2-dim}, and in Appendix~\ref{appendix: higher order moments 2-dim} we give further details on these matrices for order $n=3$.
The stacked vector $(\bm{\Psi}_t^{(1)}, \bm{\Psi}_t^{(2)},\dots, \bm{\Psi}_t^{(n)})^\top$ satisfies the ODE
\begin{align} \label{eq: ODE stack of stacked Psi vector, QL joint general order moments}
    \frac{\ddiff}{\ddiff t}
    \begin{bmatrix} \bm{\Psi}^{(1)}_t \\ 
    \vdots \\
    \bm{\Psi}^{(n)}_t\end{bmatrix}
    &=
    \bm{F}_{n}^{\mathfrak{p}_{n}\times\mathfrak{p}_{n}}
    \begin{bmatrix} \bm{\Psi}^{(1)}_t \\ 
    \vdots \\
    \bm{\Psi}^{(n)}_t\end{bmatrix}
    +
    \begin{bmatrix}
    \bm{b}^{2\times 1}  \\
    \bm{0}^{(\mathfrak{p}_{n}-p_{1})\times 1}
    \end{bmatrix},
\end{align}
with initial condition given by $(\bm{\Psi}_0^{(1)}, \bm{\Psi}_0^{(2)},\dots, \bm{\Psi}_0^{(n)})^\top$.
The solution is given in the following proposition, noting that taking the time derivative of Eqn.~\eqref{eq: stack of stacked Psi Q lambda solution ODE} immediately yields Eqn.~\eqref{eq: ODE stack of stacked Psi vector, QL joint general order moments}.

\begin{proposition} \label{prop: stack of stacked Psi Q lambda solution}
If $\bm{H}_{i}^{p_{i}\times p_{i}}$ is invertible for all $i\in\{1,\ldots,n\}$,
then 
\begin{align} \label{eq: stack of stacked Psi Q lambda solution ODE}
\begin{split}
\begin{bmatrix} 
    \bm{\Psi}^{(1)}_t \\ 
    \vdots \\
    \bm{\Psi}^{(n)}_t\end{bmatrix}
    &= e^{\bm{F}_{n}^{\mathfrak{p}_{n}\times\mathfrak{p}_{n}}t}
    \begin{bmatrix} \bm{\Psi}^{(1)}_0 \\ 
    \vdots \\
    \bm{\Psi}^{(n)}_0\end{bmatrix}
    -\left(\bm{F}_{n}^{\mathfrak{p}_{n}\times\mathfrak{p}_{n}}\right)^{-1}
    \left(\bm{I}^{\mathfrak{p}_{n}\times\mathfrak{p}_{n}}
    -e^{\bm{F}_{n}^{\mathfrak{p}_{n}\times\mathfrak{p}_{n}}t}\right)
    \begin{bmatrix}
    \bm{b}^{2\times 1}  \\
    \bm{0}^{(\mathfrak{p}_{n}-p_{1})\times 1}
    \end{bmatrix}.
\end{split}
\end{align}    
\end{proposition}

A crucial computational advantage of Proposition~\ref{prop: stack of stacked Psi Q lambda solution} is that the evaluation of joint moments does not require integration of matrix exponentials as in Proposition~\ref{prop: solution vector ODE with matrix exponential}.
By considering the joint moments, the matrix in the associated ODE is more involved, but we reduced the non-homogeneous part of the ODE in Eqn.~\eqref{eq: ODE stack of stacked Psi vector, QL joint general order moments} to a constant, which allows for a closed-form solution.
This in particular means that the run time of computing these transient moments does not increase in $t$, as opposed to the other computational methods.
In the numerical computations of Section~\ref{sec: numerics} we quantify the advantage in terms of computation time, compared to alternative approaches.

\section{The Nearly Unstable Behavior} \label{sec: nearly unstable}

In this section, we analyze the stationary behavior of $\bm{\lambda}$ when the spectral radius of $\bm{H}$ approaches 1, see Eqn.~\eqref{eq: general stability condition}, i.e., when the underlying process approaches criticality.
This regime directly relates to the \textit{heavy traffic} regime in queueing theory.
In the setting considered, the dimension $d\in\nn$ is general.
For the sake of tractability, we impose the following `symmetry assumption'.

\begin{assumption}\label{ass_sym}
For all $i\in[d]$, 
\begin{align} \label{eq: symmetric parameter choice heavy traffic, d-dim}
    \alpha_i = \alpha \geqslant 0, \quad B_{1i} \overset{d}{=} \dots \overset{d}{=} B_{di} \overset{d}{=}B_i, \quad \overline{\lambda}_i = \overline{\lambda} > 0,
\end{align}
where $B_i$ are independent non-negative random variables with $\ee[B_i^2] <\infty$.
\end{assumption}
This choice of parameters induces symmetry, since it implies that each component $\lambda_i$ has the same base rate $\overline{\lambda}$, the same decay rate $\alpha$, and that it is self- or cross-excited by all $B_1,\dots,B_d$.
The object of study is the Laplace transform
\begin{align}
    \cT\{\bm{\lambda}\}(\bm{s}) = \ee\big[ e^{-\bm{s}^\top \bm{\lambda}}\big] 
    = \ee\Big[ \prod_{i=1}^d e^{-s_i\lambda_i}\Big].
\end{align}
The following result, proven in Appendix \ref{appendix: nearly unstable}, yields an explicit solution for $\cT\{\bm{\lambda}\}(\bm{s})$.

\begin{lemma}\label{lemma: near unstable}
Assume Eqn.~\eqref{eq: symmetric parameter choice heavy traffic, d-dim} and let $\beta_i(u) = \ee[e^{-uB_i}]$ for any $u \geqslant 0$. 
Then, with $\overline{s} = s_1 + \dots + s_d$, we have
\begin{align} \label{eq: lemma statement solution LT lambda steady-state}
    \cT\{\bm{\lambda}\}(\bm{s}) = \exp\Big(-\alpha\overline{\lambda} \int_0^{\overline{s}} \frac{u}{\alpha u + \sum_{i=1}^d \beta_i(u) - d} \ddiff u\Big).
\end{align}
\end{lemma}

We now use the above lemma to derive the desired limit result in the nearly unstable case.
Observe that the stability condition of the matrix $\bm{C}$, see Eqn.~\eqref{eq: general stability condition}, having a maximum eigenvalue smaller than 1, is in our setting explicitly given by
\begin{align}
    \theta := \frac{1}{\alpha} \sum_{i=1}^d \ee[B_i] < 1.
\end{align}
Furthermore, let $\sigma := 2\alpha \big(\sum_{i=1}^d \ee[B_i^2]\big)^{-1}$.
The following result is proven in Appendix \ref{appendix: nearly unstable}.

\begin{theorem} \label{thm: heavy traffic lambda thm}
Under Assumption \ref{ass_sym}, 
\begin{align} \label{eq: limit heavy traffic Laplace transform lambda steady state, d-dim}
    \lim_{\theta \uparrow 1} \cT\{\bm{\lambda}\}(\bm{s}(1-\theta)) 
    = \Big( \frac{\sigma}{\sigma + \overline{s}}\Big)^{\sigma\overline{\lambda}}.
\end{align}
\end{theorem}

Theorem \ref{thm: heavy traffic lambda thm} shows that the limiting random vector has a (multivariate) Gamma distribution.
It is immediately verified that, for any $i,j\in[d]$,
\begin{align} \label{eq: covariance steady-state heavy traffic}
\begin{split}
    \lim_{\theta \uparrow 1}\text{Cov}((1-\theta)\lambda_i,(1-\theta)\lambda_j) 
    = \overline{\lambda} / \sigma,
\end{split}
\end{align}
by virtue of Eqn.~\eqref{eq: limit heavy traffic Laplace transform lambda steady state, d-dim}.
This yields the following result, writing $\Gamma(r,\lambda)$ for a Gamma distributed random variable with shape parameter $r>0$ and scale parameter $\lambda>0$. It follows immediately from Theorem \ref{thm: heavy traffic lambda thm} combined with Lévy's continuity theorem.
The covariance expression follows from Eqn.~\eqref{eq: covariance steady-state heavy traffic}.

\begin{corollary} \label{cor: lambda heavy traffic convergence in distribution}
Under Assumption \ref{ass_sym}, for a random vector $\bm{X}$, we have as $\theta \uparrow 1$
\begin{align*}
    (1-\theta)\bm{\lambda} \overset{d}{\to} \bm{X},
\end{align*}
where each marginal $X_i \sim \Gamma(\sigma\overline{\lambda}, \sigma)$ and $\text{\rm Cov}(X_i,X_j) = \overline{\lambda} /\sigma$ for any $i,j\in[d]$.
\end{corollary}

\section{Numerical Experiments} \label{sec: numerics}
The objective of the results presented in the previous section is to be able to numerically evaluate moments.
In this section we compare the resulting output, in terms of efficiency and accuracy, to two alternatives.

\smallskip

\noindent $\circ$~The first alternative is based on \textit{finite differences} (FD). 
    The main idea is that we obtain approximations of moments by performing numerical differentiation of the relevant transform, using the characterizations in Theorem~\ref{thm: zeta characterization} and Corollary~\ref{cor: zeta_0 characterization}.
As we have seen, the moments of interest can be obtained by appropriately differentiating the joint transform with respect to $\bm{s}$ and $\bm{z}$ and then setting $\bm{s} = \bm{0}$ and $\bm{z}=\bm{1}$.
In the FD approach these derivatives are approximated by the corresponding (central) finite differences, parameterized the `width parameter' $h>0$. 
In our experiments, we assess how the precision of the approximations depends on $h$.
 
    \noindent $\circ$~The second alternative technique is based on \textit{Monte Carlo simulation} (MC). To simulate the Hawkes process we use an algorithm based on Ogata's thinning algorithm; see \cite{O81} and \cite[Algorithm 1.21]{L09} for details. The sampling mechanism is based on the cluster representation of \cite[Definition 2]{KLM21}. 
    Clearly, when relying on MC there is an evident tradeoff between precision and computational effort. Indeed, the well-known rule of thumb is that a reduction of the width of the confidence interval by a factor 2, requires the number of runs $m\in\nn$ to be multiplied by a factor 4.

\smallskip

The first subsection focuses on computing moments of order $n=1$ and $n=2$ in the $d$-dimensional setting of Section~\ref{sec: joint moments d-dim}.
For $d=3$ we compute these moments by evaluating the solutions of the ODEs derived in Section~\ref{sec: joint moments d-dim}, which are used as benchmarks to compare the alternative methods to.
In the second subsection, we consider a setting with $d=2$, where we apply the results of the nested block-matrix developed in Section~\ref{sec: nested block matrices}, allowing us to compute moments up to any order $n\in\nn$.
By this method, relying on analytical closed-form expressions, we compute moments of order up to three, which are later used as benchmarks.

The performance of the various approaches encompasses efficiency and precision, 
which we quantify in terms of {\it run time} and {\it error} (relative to the benchmark that we defined above, in the sequel abbreviated by BM), respectively.
Two types of errors are distinguished, namely the Mean Absolute Error (MAE) and Mean Relative Error (MRE):
\begin{align}
    \text{MAE} = \sum_{j=1}^{k} |m_j^\textrm{(BM)} - m_j^\textrm{(FD/MC)}|, \quad 
    \text{MRE} = \sum_{j=1}^{k} \frac{|m_j^\textrm{(BM)} - m_j^\textrm{(FD/MC)}|}{m_j^\textrm{(BM)}},
\end{align}
where $m_j$ denotes our BM value of the $j$-th moment, $k$ is the number of moments computed, and the superscript indicates the computational method used.
The general conclusion of this section is that the experiments systematically reveal that the benchmarks outperform the alternative approaches, in terms of efficiency and accuracy, typically by a large margin.

\subsection{Multivariate}

We consider an example of dimension $d=3$, in which we numerically evaluate the first and second order moments of $(\bm{Q}(t),\bm{\lambda}(t))$. 
We start by taking $t=5$, to later study the impact of the value of $t$.
The marks are exponentially distributed: for any combination $i,j\in\{1,2,3\}$, we set $B_{ij} \sim \text{Exp}(b_{ij})$ for some $b_{ij} > 0$, which are (for simplicity) assumed independent. 
In the experiments we take
\begin{align*}
    \overline{\bm{\lambda}} = \begin{bmatrix}
    0.3 \\
    1 \\
    0.5
    \end{bmatrix},
    \quad   \ee\bm{B} = \begin{bmatrix}
    0.5 & 0.3 & 0.4 \\
    0.7 & 0.5 & 0.5 \\
    0.4 & 0.2 & 0.5
    \end{bmatrix},
    \quad \bm{D}_\alpha = \begin{bmatrix}
    2 & 0 & 0 \\
    0 & 1.5 & 0 \\
    0 & 0 & 2.5
    \end{bmatrix},
    \quad \bm{D}_\mu = \begin{bmatrix}
    1.5 & 0 & 0 \\
    0 & 0.5 & 0 \\
    0 & 0 & 1
    \end{bmatrix}.
\end{align*}
One readily verifies that for these parameters the stability condition of Assumption~\ref{ass: stability condition} is met.

Recall that the stacked vector $\bm{\Sigma}_t^{(1)}$ and the stacked matrix $\bm{\Sigma}_t^{(2)}$ contain all the first order moments and combinations of second order moments, respectively; see Eqns.~\eqref{eq: def transient order 1 d dim stacked vector} and \eqref{eq: def 2nd order stacked matrix d-dim}.
In this subsection, the benchmark BM corresponds to the solution of the vector- and matrix-valued ODEs as given in Eqns.~\eqref{eq: transient ODE 1st order d-dim} and \eqref{eq: transient ODE 2nd order d-dim}, obtained using the \texttt{SciPy} package in \texttt{Python}.
We used the default precision of the \texttt{SciPy} ODE solver; the output presented in the next subsection indicates that this provides sufficiently precise results.

Table~\ref{tab: runtime all trivariate moments}, displaying the resulting run times and errors, quantifies the superior performance of our approach.
In each of the settings, the run time RT is reliably estimated by performing the experiment sufficiently often and taking the average of the corresponding run times.
The table shows that the benchmark ODE method is faster than the FD method, and orders of magnitude faster than the Monte Carlo simulation, where the latter method in addition typically yields estimates with substantial errors (where the number of simulation runs is $m = 10^3$). 
We see that in the FD method smaller values of $h$ lead to lower run times: in this method, we vary the arguments $\bm{s}$ and $\bm{z}$ with $h$ when evaluating the joint transform, which is faster for smaller values of $h$.
Furthermore, observe that the MAE and MRE are not monotone in $h$: for larger $h$ the derivative is poorly approximated by the finite difference, while for smaller $h$  numerical stability issues have a detrimental effect.
There is an optimal width where the error is smallest, which in our instance happens to be around $h=10^{-3}$.

\begin{table}[h!]
{\small
    \centering
    \begin{tabular}{c c c c c c c c c c c c } 
         & BM & & \multicolumn{4}{c}{FD} & & \multicolumn{4}{c}{MC} \\
         \cline{2-2} \cline{4-7} \cline{9-12}
        $n$ & RT & & $h$ & RT & MAE & MRE &  & $m$ & RT  & MAE & MRE \\ \hline
        \rule{0pt}{2.5ex} 1 & 7.17$\,\cdot 10^{-3}$  & & $10^{-2}$ & 5.26$\,\cdot 10^{-2}$  & 1.27$\,\cdot 10^{-3}$ & 3.42$\,\cdot 10^{-4}$ & & $10^2$ & 9 & 8.92$\,\cdot 10^{-1}$ & 4.98$\,\cdot 10^{-1}$  \\
        & $\cdot$ &  & $10^{-3}$ & 3.94$\, \cdot 10^{-2}$ &1.81$\,\cdot 10^{-4}$ &  9.84$\,\cdot10^{-5}$ &  & $10^3$ & 96 & 2.41$\,\cdot 10^{-1}$ & 1.43$\,\cdot 10^{-1}$ \\ 
        & $\cdot$ & & $10^{-4}$ & 3.35$\, \cdot 10^{-2}$ & 6.60$\,\cdot 10^{-4}$ & 3.69$\,\cdot 10^{-4}$ & & $10^4$ & 956 & 6.87$\,\cdot 10^{-2}$ & 3.68$\,\cdot 10^{-2}$ \\ \hline
        \rule{0pt}{2.5ex} 2 & $9.65 \, \cdot 10^{-2}$ & & $10^{-2}$ &  3.29$\,\cdot 10^{-1}$ & 9.96$\,\cdot 10^{-2}$ &  1.17$\,\cdot 10^{-2}$ & & $10^{2}$ & 13 & 2.48$\,\cdot 10^{1}$ & 4.56$\,\cdot 10^{0}$   \\
         & $\cdot$ & & $10^{-3}$ &  2.62$\,\cdot10^{-1}$ & 2.03$\,\cdot 10^{-3}$ & 4.18$\,\cdot 10^{-4}$ & & $10^{3}$ & 115 & 3.89$\,\cdot 10^{0}$ & 7.69$\,\cdot 10^{-1}$   \\
          & $\cdot$ & & $10^{-4}$ &  2.20$\,\cdot10^{-1}$ & 1.13$\,\cdot 10^{-2}$ &  1.18$\,\cdot10^{-3}$ & & $10^{4}$ & 1025 & 9.07$\,\cdot 10^{-1}$ & 2.42$\,\cdot 10^{-1}$ 
    \end{tabular}
    \caption{\label{tab: runtime all trivariate moments}\small \textit{Run times (RT) in seconds and errors (MAE, MRE) for first $(n=1)$ and second $(n=2)$ order moments: performance of the benchmark ODE-based method relative to FD and MC.  
    }}
}
\end{table}

To assess whether the effects observed in the previous experiment hold in general,  we have performed experiments with a set of intrinsically different parameter settings.
In these experiments, we study run times and errors, while we fix the `width parameter' at $h = 10^{-3}$ and number of simulation runs at $m = 10^3$.
Since varying each entry in each of the vectors and matrices would lead to a large set of instances, we decided to focus on altering only the parameters directly pertaining to $\lambda_1(\cdot)$ and $Q_1(\cdot)$, while respecting Assumption~\ref{ass: stability condition}; note that the effect will propagate to other components due to cross-excitation.
Table~\ref{tab: runtime all trivariate moments diff parameters} shows the resulting run times and errors.
The main conclusion is that the experiments reveal that, uniformly across all instances, the benchmark ODE-based method remains the fastest, with a run time that is hardly affected by the parameters chosen.
We note that increasing the value of $\ee[B_{11}]$ or decreasing the value of $\alpha_1$ results in the system approaching the boundary of the stability condition in Assumption~\ref{ass: stability condition}, thus leading to larger relative errors.

\begin{table}[h!]
{\small
    \centering
    \begin{tabular}{r @{\hspace{0.8\tabcolsep}} l c c c c c c c c c} 
     & & BM & \multicolumn{4}{c}{FD} & \multicolumn{4}{c}{MC} \\
    \cline{3-3} \cline{5-7} \cline{9-11}
     \multicolumn{2}{r}{Parameter} & RT & & RT & MAE & MRE & & RT & MAE & MRE \\ \hline
    \rule{0pt}{2.5ex} $\ol{\lambda}_1$ &$= 3$ & 1.13$\,\cdot 10^{-1}$ & & 3.04$\,\cdot 10^{-1}$ & 2.48$\,\cdot 10^{-2}$ & 7.44$\,\cdot 10^{-4}$ & & 286 & 7.19$\,\cdot 10^{0}$ & 3.34 $\,\cdot 10^{-1}$ \\
    $\ol{\lambda}_1$ & $= 5$ & 1.20$\,\cdot 10^{-1}$ & & 3.10$\,\cdot 10^{-1}$ & 7.68$\,\cdot 10^{-2}$ & 1.02$\,\cdot 10^{-3}$ & & 432 & 1.47$\,\cdot 10^{1}$ & 3.98$\,\cdot 10^{-1}$ \\
    $\ol{\lambda}_1$ & $= 10$ & 1.36$\,\cdot 10^{-1}$ & & 3.32 $\,\cdot 10^{-1}$ & 5.37$\,\cdot 10^{-1}$ & $2.09\,\cdot 10^{-3}$ & & 812 & 1.18$\,\cdot 10^{2}$ & 6.19$\,\cdot 10^{-1}$ \\ \hline
    $\ee B_{11}$ & $= 1$ & 1.04$\,\cdot 10^{-1}$ & &  3.05$\,\cdot 10^{-1}$ & 1.02$\,\cdot 10^{-2}$ & 1.14$\,\cdot 10^{-3}$ & & 118 & 5.48$\,\cdot 10^{0}$ & 7.55$\,\cdot10^{-1}$ \\
    $\ee B_{11}$ & $= 1.3$ & 1.06$\,\cdot 10^{-1}$ & & 3.07$\,\cdot 10^{-1}$ & 1.52$\,\cdot 10^{-2}$ & 1.07$\,\cdot 10^{-3}$ & & 131 & 1.12$\,\cdot 10^{1}$ & 9.38$\,\cdot 10^{-1}$ \\
    $\ee B_{11}$ & $= 1.6$ & 1.02$\,\cdot 10^{-1}$ & & 3.08$\,\cdot 10^{-1}$ & $1.19\,\cdot 10^{-1}$ & 2.62$\,\cdot 10^{-3}$ & & 152 & 8.72$\,\cdot 10^{1}$ & 2.38$\,\cdot 10^{0}$ \\ \hline
    $\alpha_1$ & $= 1$ & 1.03$\,\cdot 10^{-1}$ & &  2.97$\,\cdot 10^{-1}$ & 7.91$\,\cdot 10^{-3}$ & 5.11$\,\cdot 10^{-3}$ & & 130 & 1.29$\,\cdot 10^{1}$ & 1.09$\,\cdot 10^{0}$ \\
    $\alpha_1$ & $= 3$ & 1.11$\,\cdot 10^{-1}$ & &  3.27$\,\cdot 10^{-1}$ & 2.89$\,\cdot 10^{-3}$ & 1.92$\,\cdot 10^{-3}$ & & 93 & 5.85$\,\cdot 10^{0}$ & 2.04$\,\cdot 10^{0}$ \\
    $\alpha_1$ & $= 10$ & 1.73$\,\cdot 10^{-1}$ & &  5.45$\,\cdot 10^{-1}$ & 1.14$\,\cdot 10^{-2}$ & 8.89$\,\cdot 10^{-3}$ & & 80 & 2.29$\,\cdot 10^{0}$ & 9.83$\,\cdot 10^{-1}$ \\ \hline
    $\mu_1$ & $= 0.5$ & 1.01$\,\cdot 10^{-1}$ & &  3.01$\,\cdot 10^{-1}$ & 2.49$\,\cdot 10^{-3}$ & 5.05$\,\cdot 10^{-4}$ & & 98 & 3.84$\,\cdot 10^{0}$ & 6.30$\,\cdot 10^{-1}$ \\
    $\mu_1$ & $= 2$ & 1.06$\,\cdot 10^{-1}$ & &  3.04$\,\cdot 10^{-1}$ & 2.23$\,\cdot 10^{-3}$ & 5.82$\,\cdot 10^{-4}$ & & 97 & 4.98$\,\cdot 10^{0}$ & 1.25$\,\cdot 10^{0}$ \\
    $\mu_1$ & $= 5$ & 1.26$\,\cdot 10^{-1}$ & &  3.18$\,\cdot 10^{-1}$ & 2.42$\,\cdot 10^{-3}$ & 1.12$\,\cdot 10^{-3}$ & & 99 & 5.22$\,\cdot 10^{0}$ & 1.45$\,\cdot 10^{0}$
    \end{tabular}
    \caption{\label{tab: runtime all trivariate moments diff parameters}\small \textit{Run times (RT) in seconds and errors (MAE, MRE) for combined first and second order moments in the trivariate setting: effect of parameter changes of the benchmark ODE-based method relative to FD and MC.  
    }}
}
\end{table}

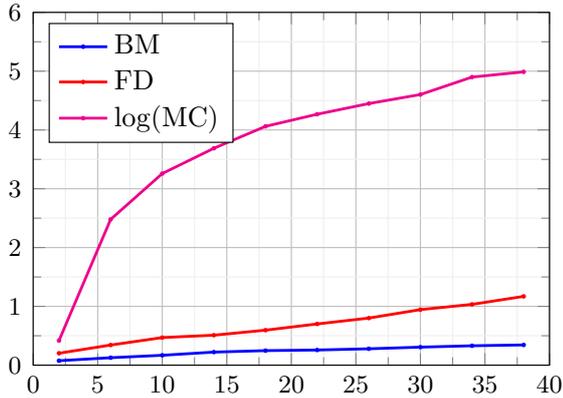
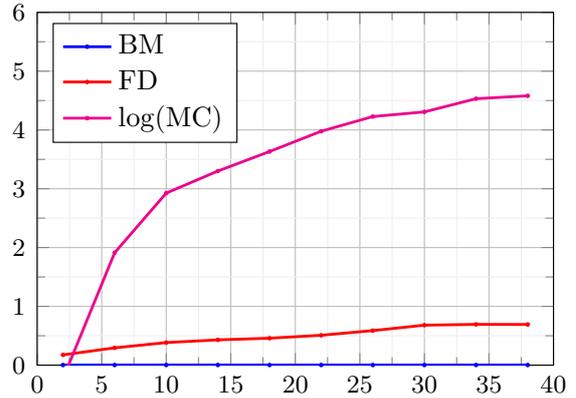
\begin{figure}[htbp]
\centering

\begin{subfigure}[t]{0.48\textwidth}
\centering
\resizebox{\textwidth}{!}{
 \pgfplotstableread{F1.txt}{\table}
 \begin{tikzpicture}
    \begin{axis}[
        xmin = 0, xmax = 40,
        ymin = 0, ymax = 6,
        xtick distance = 5,
        ytick distance = 1,
        grid = both,
        minor tick num = 1,
        major grid style = {lightgray},
        minor grid style = {lightgray!25},
        width = \textwidth,
        height = 0.75\textwidth,
        legend cell align = {left},
        legend pos = north west
    ]
        \addplot[blue, mark = *, mark size = 0.3pt, line width = 1pt] table [x = {x}, y = {a}] {\table};
        \addplot[red, mark = *, mark size = 0.3pt, line width = 1pt] table [x = {x}, y = {b}] {\table};
        \addplot[magenta, mark = *, mark size = 0.3pt, line width = 1pt] table [x = {x}, y = {c}] {\table};
        \addlegendentry{BM}
        \addlegendentry{FD}
        \addlegendentry{$\log ({\rm MC})$}
    \end{axis}
 \end{tikzpicture}
}
\caption{Trivariate setting: up to 2nd-order moments}
\label{fig:runtime-tri}
\end{subfigure}
\hfill
\begin{subfigure}[t]{0.48\textwidth}
\centering
\resizebox{\textwidth}{!}{
 \pgfplotstableread{F2.txt}{\table}
 \begin{tikzpicture}
    \begin{axis}[
        xmin = 0, xmax = 40,
        ymin = 0, ymax = 6,
        xtick distance = 5,
        ytick distance = 1,
        grid = both,
        minor tick num = 1,
        major grid style = {lightgray},
        minor grid style = {lightgray!25},
        width = \textwidth,
        height = 0.75\textwidth,
        legend cell align = {left},
        legend pos = north west
    ]
        \addplot[blue, mark = *, mark size = 0.3pt, line width = 1pt] table [x = {x}, y = {a}] {\table};
        \addplot[red, mark = *, mark size = 0.3pt, line width = 1pt] table [x = {x}, y = {b}] {\table};
        \addplot[magenta, mark = *, mark size = 0.3pt, line width = 1pt] table [x = {x}, y = {c}] {\table};
        \addlegendentry{BM}
        \addlegendentry{FD}
        \addlegendentry{$\log ({\rm MC})$}
    \end{axis}
 \end{tikzpicture}
}
\caption{Bivariate setting: up to 3rd-order moments}
\label{fig:runtime-bi}
\end{subfigure}

\caption{Run times of the BM, FD, and MC methods for computing moments. (a) Trivariate setting (up to second order). (b) Bivariate setting (up to third order).}
\label{fig:runtime-comparison}
\end{figure}

We also studied the effect of varying the time parameter $t$ on the run times.
Recall that the FD method uses the (conditional) joint transform, where the latter requires solving systems of ODEs.
Figure~~\ref{fig:runtime-bi} shows that the run times of the ODE-based method and the FD method scale effectively linearly with $t$, with the ODE method having the smaller slope.
The run time for MC increases superlinearly; we took its logarithm to be able to show it in the same plot. This superlinear behavior is an inherent consequence of the branching structure underlying the Hawkes process (relied upon in Ogata's algorithm).
We note that this qualitative behavior is observed for all choices of parameters that we considered.

\subsection{Bivariate}

In this subsection, we compute for the bivariate setting ($d=2$) the transient moments of $\bm{Q}(t) = (Q_1(t), Q_2(t))$ and $\bm{\lambda}(t) = (\lambda_1(t), \lambda_2(t))$, of orders $n=1,2,3$. Again we take $t=5$, but later assess the effect of the choice of $t$. 
As before the random marks are exponentially distributed, i.e., for $i,j\in\{1,2\}$, we set $B_{ij} \sim \text{Exp}(b_{ij})$ for some $b_{ij} > 0$, with independence between the $B_{ij}$. 
The parameters are
\begin{align*}
    \overline{\bm{\lambda}} = \begin{bmatrix}
    0.5 \\
    0.5
    \end{bmatrix},
    \quad   \ee\bm{B} = \begin{bmatrix}
    1.5 & 0.5 \\
    0.75 & 1.25 \\
    \end{bmatrix},
    \quad \bm{D}_\alpha = \begin{bmatrix}
    3 & 0 \\
    0 & 2 \\
    \end{bmatrix},
    \quad \bm{D}_\mu = \begin{bmatrix}
    1 & 0 \\
    0 & 2
    \end{bmatrix}.
\end{align*}
It can be verified that the stability condition of Eqn.~\eqref{eq: general stability condition}, which in this bivariate case reads
$
    (\alpha_1 - \ee[B_{11}])(\alpha_2 - \ee[B_{22}]) > \ee[B_{12}]\ee[B_{21}]$, is met.

We compute all the first, second and third order moments, i.e., all entries of the stacked vectors $\bm{\Psi}_t^{(1)}$, $\bm{\Psi}_t^{(2)}$, and $\bm{\Psi}_t^{(3)}$.
We use the main result of Section~\ref{sec: nested block matrices}, namely Proposition~\ref{prop: stack of stacked Psi Q lambda solution}, as our benchmark.
This result exploits the block-matrix structure, by which we can simultaneously compute moments of multiple orders, thus greatly increasing the computational performance.
As it turns out, the output hardly differs from what is obtained by the ODE-based approach of the previous subsection (i.e., a difference in the order of $10^{-8}$).
The difference in computational effort, however, is substantial: for this instance the block-matrix method is about $200$ times faster.

Table~\ref{tab: runtime all bivariate moments} shows that the BM method is much faster than FD and MC, especially for second and third order moments (with again the number of simulation runs being $m = 10^3$).
We also see that the absolute and relative errors of FD and MC significantly grow as the order of moments increase.
Particularly for the third order moments, the poor stability of the FD method significantly degrades the performance, as can be seen by the variability of the error when changing the precision parameter $h$.

\begin{table}[h]
{\small
    \centering
    \begin{tabular}{c c c c c c c c c c c c } 
         & BM & & \multicolumn{4}{c}{FD} & & \multicolumn{4}{c}{MC} \\
         \cline{2-2} \cline{4-7} \cline{9-12}
        $n$ & RT & & $h$ & RT & MAE & MRE &  & $m$ & RT  & MAE & MRE \\ \hline
        1 & 4.77$\,\cdot 10^{-4}$  & & $10^{-2}$ & 4.42$\,\cdot 10^{-2}$  & 2.55$\,\cdot 10^{-3}$ & 1.12$\,\cdot 10^{-3}$ & & $10^2$ & 5 & 7.16$\,\cdot 10^{-1}$ & 3.32$\,\cdot 10^{-1}$  \\
        & $\cdot$ &  & $10^{-3}$ & 3.31$\, \cdot 10^{-2}$ &  8.55$\,\cdot 10^{-5}$ &  4.61$\,\cdot10^{-5}$ &  & $10^3$ & 62 & 3.19$\,\cdot 10^{-1}$ & 1.49$\,\cdot 10^{-1}$ \\
        & $\cdot$ & & $10^{-4}$ & 2.97$\, \cdot 10^{-2}$ & 4.38$\,\cdot 10^{-4}$ & 2.64$\,\cdot 10^{-4}$ & & $10^4$ & 589 & 1.69$\,\cdot 10^{-2}$ & 7.30$\,\cdot 10^{-3}$ \\ \hline
        2 & $5.61 \, \cdot 10^{-4}$ & & $10^{-2}$ &  8.12$\,\cdot 10^{-2}$ & 2.70$\,\cdot 10^{-1}$ &  2.45$\,\cdot 10^{-2}$ & & $10^{2}$ & 6 & 2.98$\,\cdot 10^{1}$ & 2.76$\,\cdot 10^{0}$   \\
        & $\cdot$ & & $10^{-3}$ &  6.42$\,\cdot10^{-2}$ & 2.03$\,\cdot 10^{-3}$ & 4.18$\,\cdot 10^{-4}$ & & $10^{3}$ & 67 & 8.01$\,\cdot 10^{0}$ & 9.71$\,\cdot 10^{-1}$   \\
        & $\cdot$ & & $10^{-4}$ &  5.73$\,\cdot10^{-2}$ & 1.14$\,\cdot 10^{-2}$ &  1.38$\,\cdot10^{-3}$ & & $10^{4}$ & 631 & 5.36$\,\cdot 10^{0}$ & 6.92$\,\cdot 10^{-1}$  \\ \hline
        3 &  9.26$ \, \cdot 10^{-4}$ & & $10^{-2}$ &  2.49$\,\cdot 10^{-1}$ & 2.23$\,\cdot 10^{2}$ &  1.77$\,\cdot 10^{0}$ & & $10^{2}$ & 7 & 9.57$\,\cdot 10^{2}$ & 1.09 $\,\cdot 10^{1}$   \\
        & $\cdot$ & & $10^{-3}$ &  2.07$\,\cdot10^{-1}$ & 2.99$\,\cdot 10^{2}$ & 5.61$\,\cdot 10^{1}$ & & $10^{3}$ & 69 & 2.21$\,\cdot 10^{2}$ & 2.87 $\,\cdot 10^{0}$   \\
        & $\cdot$ & & $10^{-4}$ &  1.72$\,\cdot10^{-1}$ & 1.34$\,\cdot 10^{6}$ &  3.34$\,\cdot10^{4}$ & & $10^{4}$ & 639 & 3.98$\,\cdot 10^{1}$ & 4.47$\,\cdot 10^{-1}$
    \end{tabular}
    \caption{\label{tab: runtime all bivariate moments}\small \textit{Run times (RT) in seconds and errors (MAE, MRE) for first $(n=1)$, second $(n=2)$, and third $(n=3)$ order moments in the bivariate setting: comparison of the benchmark block-matrix method relative to FD and MC.
    }}
}
\end{table}

Figure~\ref{fig:runtime-tri} shows the effect of the time parameter $t$.
As before, the run time of FD scales linearly with $t$, and that of  MC superlinearly.
This should be contrasted with the attractive feature of the block-matrix method that its run time does not depend on $t$.

We now consider the block-matrix stationary moments. 
The first and second order stationary moments can be immediately obtained from the results of Appendix~\ref{sec: stationary moments d-dim}, by solving the associated Sylvester matrix equations.
Obvious alternatives when not knowing these stationary moments, amount to picking a `large' value of $t$ in the FD and MC methods.
These alternative methods have two intrinsic drawbacks: (1)~run times increase in $t$, and (2)~we do not know \textit{a priori} what value of $t$ guarantees that the error made is sufficiently small.
As we already saw that MC is typically outperformed by FD, we focus on FD only. 
A pragmatic way to select a sufficiently large value of $t$, is to compute the FD-based approximation of first and second order transient moments for successive \textit{integer} values of $t$, until the difference of the respective MREs is smaller than some given threshold precision level $\epsilon$. 
We compare the resulting approximation to our benchmark, i.e., the values obtained by solving the Sylvester matrix equations, so as to quantify the error made.

Table~\ref{tab: runtime transient - stationary bivariate} presents the results for the FD method with precision level $\epsilon = 0.01$.
We have performed the above procedure for different choices of parameters, where the parameter that is altered is given in the table.
Note that the benchmark method is exact and provides near-instant response. 
Observe that for specific sets of parameters there is a substantial effect on the value of $t$ (i.e., the value of the time parameter at which the procedure terminates), the run time, and the MRE, in particular when the parameters are close to the boundary of the stability condition in Assumption~\ref{ass: stability condition}, e.g., $\ee[B_{11}]=2.25$ or $\alpha_1=2.1$.



\begin{table}[hh]
{\small
    \centering
    \begin{tabular}{r @{\hspace{0.8\tabcolsep}} l c c c c } 
     & & & \multicolumn{3}{c}{FD} \\
    \cline{4-6} 
     \multicolumn{2}{r}{Parameter} & & $t$ & RT & MRE  \\ \hline
    \rule{0pt}{2.5ex} $\ol{\lambda}_1$ &$= 3$ & & 18 & 3.25$\,\cdot 10^{0}$ & 2.17$\,\cdot 10^{-2}$  \\
    $\ol{\lambda}_1$ &$= 10$ & & 18 & 3.38$\,\cdot 10^{0}$ & 1.94$\,\cdot 10^{-2}$  \\ \hline
    \rule{0pt}{2.5ex} $\ee B_{11}$ &$= 0.5$ & & 14 & 2.75$\,\cdot 10^{0}$ & 1.51$\,\cdot 10^{-2}$  \\
    $\ee B_{11}$ &$= 2.25$ & & 43 & 1.08$\,\cdot 10^{1}$ & 6.87$\,\cdot 10^{-2}$  \\ \hline
    \rule{0pt}{2.5ex} $\alpha_1$ &$= 2.1$ & & 86 & 2.89$\,\cdot 10^{1}$ & 1.58$\,\cdot 10^{-1}$  \\
    $\alpha_1$ &$= 5$ & & 15 & 3.77$\,\cdot 10^{0}$ & 1.55$\,\cdot 10^{-2}$  \\ \hline
    \rule{0pt}{2.5ex} $\mu_1$ &$= 0.5$ & & 19 & 3.97$\,\cdot 10^{0}$ & 2.11$\,\cdot 10^{-2}$  \\
    $\mu_1$ &$= 5$ & & 18 & 3.38$\,\cdot 10^{0}$ & 2.84$\,\cdot 10^{-2}$  \\
    \end{tabular}
    \caption{\label{tab: runtime transient - stationary bivariate}\small \textit{Run times (RT) in seconds until first and second order transient moments approximate stationary moments in the bivariate setting: FD method with precision level $\epsilon = 0.01$.
    }}
}
\end{table}

In Appendix \ref{appendix: numerics} we present a numerical evaluation of the objects featured in Section~\ref{sec: characterization}.

\section{Concluding Remarks} \label{sec: conclusion}

This paper has studied multivariate Hawkes-fed Markovian infinite-server queues, which can be alternatively interpreted as population processes.
Our objective was to devise accurate and efficient algorithms to compute transient and stationary moments. 
We succeeded in doing so, heavily relying on having access to the joint transform of the Hawkes intensity process and the population process.
When the multivariate Hawkes process is of general dimension $d$, this transform is expressed in terms of systems of ODEs, allowing for the computation of joint moments.
This includes joint moments where the components pertain to the same as well as to different points in time, thus also covering the evaluation of the processes' autocovariance functions. 
We then proceeded by deriving expressions for the first and second order, transient and stationary moments for the $d$-dimensional processes.
Next, in the $2$-dimensional setting we derived a recursive procedure, revealing a block-matrix structure for the computation of moments of any order.
Our numerical experiments show that our approach outperforms existing alternatives: it produces highly accurate results with computation times that are virtually negligible.

This paper considering Markovian multivariate Hawkes processes and associated population processes, we conclude this section by a brief discussion of potential extensions to more general classes of processes. 
In the first place,  where in this work we exclusively focus on exponential decay functions, one could work with more general non-increasing and integrable decay functions. When doing so, one leaves the Markovian setting, rendering the results of this paper not applicable. 
Instead, the process may be analyzed through the cluster representation, first described in \cite{HO74}, where the Hawkes process is described as a Poisson cluster process. 
This approach has been followed in \cite{KLM21} to study distributional properties in a multivariate setting, including the multivariate Hawkes process as well as the associated population process.
Another extension concerns the nonlinear case of e.g.\ \cite{BM96,Z13}, entailing that the Hawkes process is not even a Poisson cluster process and therefore requires entirely different analysis techniques.

{\small 
\begin{spacing}{1.05}

\end{spacing}
}

\appendix

\section{Proofs of Theorems \ref{thm: zeta characterization} and \ref{thm: joint transform characterization t, t+tau}} \label{appendix: proofs}

\begin{proof}[Proof of Theorem \ref{thm: zeta characterization}]
The proof is comprised of a number of steps. 
First, we use the Markov property on the distribution function of the joint process $(\bm{Q}(t),\bm{\lambda}(t))$, next we take partial derivatives to obtain an expression for the density. Then, we consecutively apply the Laplace and $\bm{z}$-transform to obtain a PDE.
Finally, we use the method of characteristics to obtain a system of ODEs.
We describe below the main steps; some lengthy technical computations have been put into Appendix \ref{appendix: computations}.

We note that the probabilities considered in this proof are conditional on the value of the processes at time $t_0$, i.e. $\bm{Q}(t_0) = \bm{Q}_0$ and $\bm{\lambda}(t_0) = \bm{\lambda}_0$.
To start off, for $t \in \rr_+$, $\bm{k} \in \nn_+^d$ and $\bm{\nu} \in \rr_+^d$, set
\begin{align}
\label{eq: probability functions}
    F(t,\bm{\nu},\bm{k}) = \pp( \bm{\lambda}(t) \leqslant \bm{\nu}, \bm{Q}(t) = \bm{k}), 
    \quad
    \frac{\partial F(t,\bm{\nu},\bm{k})}{\partial \bm{\nu}} =
    \left(
    \frac{\partial F(t,\bm{\nu},\bm{k})}{\partial \nu_1} ,\cdots,
    \frac{\partial F(t,\bm{\nu},\bm{k})}{\partial \nu_d}
    \right)^{\top}.
\end{align}
Also define 
\begin{align}
\label{eq: pdf Q and lambda}
    f(t,\bm{\nu},\bm{k}) = \frac{\partial^d F(t,\bm{\nu},\bm{k})}{\partial \nu_1 \cdots \partial \nu_d},
\end{align}
as the joint density of $(\bm{Q}(t),\bm{\lambda}(t))$.
For some $\delta> 0$, consider the probability
\begin{align}
\label{eq: conditional probability Q,L}
    F(t+\delta, \, \bm{\nu} - \bm{\alpha} \odot (\bm{\nu}- \overline{\bm{\lambda}})\delta \, , \bm{k}),
\end{align}
where the interpretation of the term $\bm{\nu} -\bm{\alpha} \odot (\bm{\nu}- \overline{\bm{\lambda}})$ is a decay factor, in the sense that for small $\delta$, no new arrival of a point in $(t,t+\delta]$ makes the intensity $\bm{\lambda}(\cdot)$ decay with rate $\bm{\alpha}$ back to the mean reversion level $\overline{\bm{\lambda}}$.
To compute this probability, we apply the Markov property to Eqn.~\eqref{eq: conditional probability Q,L}, which leaves us to consider the possibilities to get in the state of exactly $\bm{k}$ active points and intensity equal to $\bm{\nu} - \bm{\alpha} \odot (\bm{\nu}- \overline{\bm{\lambda}})$.
There are three distinct ways to get to this state at time $t + \delta$ from time $t$: we have exactly $\bm{k}$ active points with no arrivals or departures; we have $\bm{k} - \bm{e}_j$ active points and exactly one arrival in component $j$; or we have $\bm{k}+\bm{e}_j$ active points and one departure in component $j$.
This yields, up to $o(\delta)$ terms, that
\begin{align*}
    &F(t+\delta, \, \bm{\nu} - \bm{\alpha} \odot (\bm{\nu}- \overline{\bm{\lambda}})\delta, \, \bm{k}) \\
    &= \sum_{j=1}^d \int_0^{\nu_d} \cdots \int_0^{\nu_1} \delta y_j  
    f(t, \bm{y},\bm{k} - \bm{e}_j) \pp(\bm{B}_j \leqslant \bm{\nu - y}) \mathrm{d}y_1\cdots \mathrm{d}y_d \\
    & \quad +  \sum_{j=1}^d (k_j + 1)\delta\mu_j  F(t, \bm{\nu}, \bm{k + e}_j) \\
    & \quad +  \int_0^{\nu_d} \cdots \int_0^{\nu_1} (1 - \sum_{j=1}^d\delta \mu_j k_j - \sum_{j=1}^d \delta y_j )f(t,\bm{y},\bm{k})\mathrm{d}y_1 \cdots \mathrm{d} y_d + o(\delta).
\end{align*}
Subtracting $F(t,\bm{k},\bm{\lambda})$ on both sides, dividing by $\delta$, and taking $\delta \downarrow 0$ yields 
\begin{align*}
    &\frac{\partial F(t,\bm{\nu},\bm{k})}{\partial t} - \big(\bm{\alpha} \odot (\bm{\nu}- \overline{\bm{\lambda}}) \big)^\top \frac{\partial F(t,\bm{\nu},\bm{k})}{\partial \bm{\nu}} \\
    &= \sum_{j=1}^d \int_0^{\nu_d} \cdots \int_0^{\nu_1}y_j 
    f(t, \bm{y},\bm{k} - \bm{e}_j) \pp(\bm{B}_j \leqslant \bm{\nu- y}) \mathrm{d}y_1\cdots \mathrm{d}y_d \\
    & \quad + \sum_{j=1}^d (k_j + 1)\mu_j F(t, \bm{\nu}, \bm{k + e}_j)  - \sum_{j=1}^d\int_0^{\nu_d} \cdots \int_0^{\nu_1} (\mu_j k_j + y_j)  f(t,\bm{y},\bm{k})\mathrm{d}y_1 \cdots \mathrm{d}y_d,
\end{align*}
where the left-hand side follows from the definition of the directional derivative.
Next, we take the partial derivatives with respect to $\nu_1,\dots,\nu_d$, so as to rewrite above equation in terms of the probability density function $f(t,\bm{\nu},\bm{k})$. 
By the definitions of $F$ and $f$ given in Eqns.~\eqref{eq: probability functions} and \eqref{eq: pdf Q and lambda}, and using Leibniz' integral rule on the integral terms, we obtain
\begin{align}
\label{eq: PDE equation in terms of f}
    &\frac{\partial f(t,\bm{\nu},\bm{k})}{\partial t} - \sum_{j=1}^d \alpha_j \frac{\partial}{\partial \nu_j} \nu_j f(t,\bm{\nu},\bm{k})
    + \sum_{j=1}^d \alpha_j\overline{\lambda}_j \frac{\partial f(t,\bm{\nu},\bm{k})}{\partial \nu_j} \notag\\
    &= \sum_{j=1}^d \int_0^{\nu_d} \cdots \int_0^{\nu_1}y_j f(t, \bm{y},\bm{k} - \bm{e}_j)
    \frac{\partial^d }{\partial \nu_1 \cdots \partial \nu_d} \pp(\bm{B}_j \leqslant \bm{\nu - y}) \mathrm{d}y_1\cdots \mathrm{d}y_d \\
    & \quad + \sum_{j=1}^d \Big( f(t, \bm{\nu}, \bm{k + e}_j)(k_j + 1)\mu_j 
    - f(t,\bm{\nu},\bm{k}) (k_j \mu_j + \nu_j) \Big).\notag
\end{align}
Denote the $d$-dimensional Laplace transform with respect to $\bm{\nu}$ by
\begin{align*}
    \xi(t,\bm{s},\bm{k}) := \cL(f(t,\bm{\nu},\bm{k}))(\bm{s}) = \int_0^\infty \cdots \int_0^\infty e^{-\bm{s}^\top \bm{\nu}} f(t,\bm{\nu},\bm{k})\mathrm{d}\nu_1 \cdots \mathrm{d}\nu_d.
\end{align*}
Taking the Laplace transform of Eqn.~\eqref{eq: PDE equation in terms of f} yields
\begin{align} \label{eq: PDE equation Laplace transformed}
    &\frac{\partial \xi(t,\bm{s},\bm{k})}{\partial t} 
    + \sum_{j=1}^d \alpha_j s_j \frac{\partial \xi(t,\bm{s},\bm{k})}{\partial s_j} 
    +  \sum_{j=1}^d \alpha_j \overline{\lambda}_js_j \xi(t,\bm{s},\bm{k}) \\
    &=  \sum_{j=1}^d \Big(- \frac{\partial \xi(t,\bm{s},\bm{k -e}_j)}{\partial s_j} \beta_j(\bm{s})
    + (k_j + 1)\mu_j \xi(t, \bm{s}, \bm{k + e}_j) - k_j\mu_j \xi(t,\bm{s},\bm{k}) + \frac{\partial \xi(t,\bm{s},\bm{k})}{\partial s_j}\Big); \notag
\end{align}
see Appendix \ref{appendix:transform computations} for the term-by-term derivation. 
The computations boil down to applying integration by parts, convolution arguments and the properties of $F$ and $f$.
Rewriting the expression in such a way that all derivatives end up on one side,
\begin{align}
\label{eq: PDE equation in terms of xi}
\begin{split}
    &\frac{\partial \xi(t,\bm{s},\bm{k})}{\partial t} 
    + \sum_{j=1}^d (\alpha_j s_j - 1)\frac{\partial \xi(t,\bm{s},\bm{k})}{\partial s_j} 
    + \sum_{j=1}^d \frac{\partial \xi(t,\bm{s},\bm{k -e}_j)}{\partial s_j} \beta_j(\bm{s}) \\
    &= \sum_{j=1}^d \Big( (k_j + 1)\mu_j \xi(t, \bm{s}, \bm{k + e}_j)
    - k_j\mu_j \xi(t,\bm{s},\bm{k}) - \alpha_j  \overline{\lambda}_j s_j \xi(t,\bm{s},\bm{k}) \Big).
    \end{split}
\end{align}
Next, we rewrite equation (\ref{eq: PDE equation in terms of xi}) by taking the $\bm{z}$-transform, which gives us the joint transform introduced in Eqn.~\eqref{eq: def joint transform Q L condition time t0}, i.e.,
\begin{align*}
    \zeta_{t_0}(t,\bm{s},\bm{z})
    = \sum_{k_1=0}^\infty \cdots \sum_{k_d=0}^\infty z_{1}^{k_1}\cdots z_d^{k_d} \xi(t,\bm{s},\bm{k})
    = \ee_{t_0} \Big[ e^{-\bm{s}^\top \bm{\lambda}(t)} \prod_{i=1}^d z_i^{Q_i(t)} \Big],
\end{align*}
yielding
\begin{align}
\label{eq: PDE equation in terms of zeta}
\begin{split}
    \frac{\partial \zeta_{t_0}(t,\bm{s},\bm{z})}{\partial t} 
    &+ \sum_{j=1}^d \big(\alpha_js_j + z_j \beta_j(\bm{s}) - 1\big)
    \frac{\partial \zeta_{t_0}(t,\bm{s},\bm{z})}{\partial s_j} 
    + \sum_{j=1}^d \big(\mu_j(z_j -1)\big)
    \frac{\partial \zeta_{t_0}(t,\bm{s},\bm{z})}{\partial z_j} \\
    &= - \zeta_{t_0}(t,\bm{s},\bm{z}) \sum_{j=1}^d \alpha_j \overline{\lambda}_j s_j,
\end{split}
\end{align}
where we added the subscript $t_0$ to emphasize the dependence on this initial time value.
We refer to Appendix \ref{appendix:transform computations} for the term-by-term derivation.

By employing the method of characteristics, we can rewrite the PDE in Eqn.~\eqref{eq: PDE equation in terms of zeta} into a system of ODEs. 
To that end, consider a curve in $\rr^{2d}$ parameterized by $(\hat{\bm{s}}(u),\hat{\bm{z}}(u))$ as a function of $u$, where $t_0\leqslant u \leqslant t$, which terminates at the set of parameters $(\bm{s},\bm{z})$, i.e. $(\hat{\bm{s}}(t), \hat{\bm{z}}(t)) = (\bm{s},\bm{z})$.
Since we have a first-order PDE, we easily obtain the characteristic system of ODEs by
\begin{align}
\label{eq: characteristic system ODEs}
\begin{split}
    \frac{\mathrm{d} \hat{s}_j(u)}{\mathrm{d} u} &= \alpha_j \hat{s}_j(u) + \hat{z}_j(u) \beta_j(\hat{\bm{s}}(u)) - 1, \\
    \frac{\mathrm{d} \hat{z}_j(u)}{\mathrm{d} u}&= \mu_j(\hat{z}_j(u) -1), \\
\end{split}
\end{align}
for each $j\in[d]$.
For $\hat{z}_j(\cdot)$, the solution can be directly computed as
\begin{align*}
    \hat{z}_j(u) = 1 + C_j e^{u\mu_j},
\end{align*}
where $C_j$ is derived by the boundary condition $\hat{z}_j(t) = z_j$, yielding $C_j = (z_j -1)e^{-t\mu_j}$, and thus $\hat{z}_j(u) = 1 + (z_j - 1) e^{-\mu_j(t-u)}$.
Upon substituting the solution for $\hat{z}_j(u)$ in the equation of $\hat{s}_j(u)$ in \eqref{eq: characteristic system ODEs}, we obtain the ODE
\begin{align}
\label{eq: ODE s with terminal condition}
    -\frac{\mathrm{d} \hat{s}_j(u)}{\mathrm{d} u} + \alpha_j \hat{s}_j(u) + (1+ (z_j-1)e^{-\mu_j(t-u)}) \beta_j(\hat{\bm{s}}(u)) - 1 = 0,
\end{align}
with terminal condition $\hat{s}_j(t) = s_j$.
For later purposes, we rephrase the ODE in Eqn.~\eqref{eq: ODE s with terminal condition} into an ODE subject to an initial condition.
To that end, let $v = t_0 + t - u$ such that $t_0\leqslant v \leqslant t$ and the ODE for $\hat{s}_j(\cdot)$ becomes
\begin{align*}
    \frac{\mathrm{d} \hat{s}_j(t_0 +t-v)}{\mathrm{d} v} + \alpha_j \hat{s}_j(t_0 +t-v) + (1+ (z_j-1)e^{-\mu_j(v-t_0)} \beta_j(\hat{\bm{s}}(t_0 +t-v)) - 1 = 0.
\end{align*}
Upon defining $\tilde{s}_j(v) = \hat{s}_j(t_0 + t-v)$, we have that $\tilde{s}_j(\cdot)$ satisfies Eqn.~\eqref{eq: thm zeta_t0 joint transform eqn for hatZ and tildeS},
with initial condition $\tilde{s}_j(t_0) = \hat{s}_j(t) = s_j$.

We can now solve the characteristic equation of $\zeta_{t_0}(\cdot)$.
Since the original PDE in Eqn.~\eqref{eq: PDE equation in terms of zeta} is non-homogeneous, we know the solution $\zeta_{t_0}(t,\bm{s},\bm{z})$ is not constant along characteristics, but evolves according to the right-hand side of~\eqref{eq: PDE equation in terms of zeta}.
Therefore, if we set $\hat{\zeta}_{t_0}(u) := \zeta_{t_0}(u, \hat{\bm{s}}(u), \hat{\bm{z}}(u))$ to be the solution restricted to the characteristics, then $\hat{\zeta}_{t_0}(\cdot)$ satisfies
\begin{align*}
    \frac{\partial \hat{\zeta}_{t_0}(u)}{\partial u} = -\hat{\zeta}_{t_0}(u)\sum_{j=1}^d \alpha_j \overline{\lambda}_j \hat{s}_j(u),
\end{align*}
subject to the initial condition 
\begin{align*}
    \hat{\zeta}_{t_0}(t_0) = \zeta_{t_0}(t_0,\hat{\bm{z}}(t_0),\hat{\bm{s}}(t_0)) 
    = \prod_{j=1}^d \hat{z}_j(t_0)^{Q_{j,0}} \exp\big(-\hat{s}_j(t_0)\lambda_{j,0}\big).
\end{align*}
Solving this yields
\begin{align*}
    \hat{\zeta}_{t_0}(u) &= \prod_{j=1}^d \hat{z}_j(t_0)^{Q_{j,0}} \exp\Big( - \hat{s}_j(t_0)\lambda_{j,0} - \alpha_j\overline{\lambda}_j  \int_{t_0}^u \hat{s}_j(v)\mathrm{d}v\Big).
\end{align*}
Finally, the solution of the PDE is given at the endpoint of the characteristic $(t,\bm{s},\bm{z})$, implying that $\zeta_{t_0}(t,\bm{s},\bm{z}) = \hat{\zeta}_{t_0}(t)$.
Using the relation $\hat{s}_j(t_0) = \tilde{s}_j(t)$, we obtain
\begin{align*}
    \zeta_{t_0}(t,\bm{s},\bm{z}) 
    &= \prod_{j=1}^d \hat{z}_j(t_0)^{Q_{j,0}} \exp\Big( - \hat{s}_j(t_0)\lambda_{j,0}  
    - \alpha_j\overline{\lambda}_j \int_0^t \hat{s}_j(t_0 + t-u)\mathrm{d}u\Big) \\
    &=\prod_{j=1}^d \hat{z}_j(t_0)^{Q_{j,0}} \exp\Big( - \tilde{s}_j(t)\lambda_{j,0}
    - \alpha_j\overline{\lambda}_j \int_{t_0}^t \tilde{s}_j(u)\mathrm{d} u\Big),
\end{align*}
which finishes the proof.
\end{proof}

\begin{proof}[Proof of Theorem~\ref{thm: joint transform characterization t, t+tau}]
The proof follows by conditioning on $\bm{Q}(t)$ and $\bm{\lambda}(t)$, and then applying Theorem~\ref{thm: zeta characterization} and techniques from its proof.
By the tower property we have
\begin{align} \label{eq: proofstep: conditioned joint transform t, t+tau}
\begin{split}
    &\ee\Big[ \prod_{i=1}^d y_i^{Q_i(t)}e^{-r_i\lambda_i(t)} z_i^{Q_i(t+\tau)}e^{-s_i\lambda_i(t+\tau)}\Big] \\
    &= \ee\Big[ \prod_{i=1}^d y_i^{Q_i(t)}e^{-r_i\lambda_i(t)} \ee\Big[ \prod_{i=1}^d z_i^{Q_i(t+\tau)}e^{-s_i\lambda_i(t+\tau)} \, | \, \bm{Q}(t),\bm{\lambda}(t)\Big] \Big].
\end{split}
\end{align}
The inner expectation can be derived from Theorem~\ref{thm: zeta characterization} and is given by
\begin{align*}
    \ee\Big[ \prod_{i=1}^d z_i^{Q_i(t+\tau)}e^{-s_i\lambda_i(t+\tau)} \, | \, \bm{Q}(t),\bm{\lambda}(t)\Big] = \prod_{j=1}^d \hat{z}_j(t)^{Q_j(t)} e^{-\tilde{s}_j(t+\tau)\lambda_j(t)} \exp\Big(-\overline{\lambda}_j\alpha_j \int_t^{t+\tau} \tilde{s}_j(u)\ddiff u\Big),
\end{align*}
where $\hat{z}_j(\cdot)$ and $\tilde{s}_j(\cdot)$ satisfy Eqn.~\eqref{eq: thm joint transform at t and t+tau equations for hatZ and tildeS}.
Substituting this back into Eqn.~\eqref{eq: proofstep: conditioned joint transform t, t+tau} yields
\begin{align*}
    &\ee\Big[ \prod_{i=1}^d y_i^{Q_i(t)}e^{-r_i\lambda_i(t)}  \prod_{j=1}^d \hat{z}_j(t)^{Q_j(t)} e^{-\tilde{s}_j(t+\tau)\lambda_j(t)} \exp\Big(-\overline{\lambda}_j\alpha_j \int_t^{t+\tau} \tilde{s}_j(u)\ddiff u\Big) \Big] \\
    &= \ee\Big[ \prod_{j=1}^d (y_j\hat{z}_j(t))^{Q_j(t)}e^{-(r_j + \tilde{s}_j(t+\tau))\lambda_j(t)}\Big] 
    \prod_{j=1}^d \exp\Big(-\overline{\lambda}_j\alpha_j \int_t^{t+\tau} \tilde{s}_j(u)\ddiff u\Big) \\
    &=\zeta(t, \bm{y}\odot\hat{\bm{z}}(t), \bm{r} + \tilde{\bm{s}}(t+\tau))
    \prod_{j=1}^d \exp\Big(-\overline{\lambda}_j \alpha_j \int_t^{t+\tau} \tilde{s}_j(u)\ddiff u\Big).
\end{align*}
Applying Corollary~\ref{cor: zeta_0 characterization} to the $\zeta(\cdot)$ term on the right-hand side, specifically Eqn.~\eqref{eq: zeta characterization in terms of s(u)}, 
\begin{align*}
    \zeta(t, \bm{y}\odot\hat{\bm{z}}(t), \bm{r} + \tilde{\bm{s}}(t+\tau))
    = \prod_{j=1}^d
    \exp\Big(-\overline{\lambda}_j\tilde{r}_j(t) - \overline{\lambda}_j\alpha_j \int_0^t \tilde{r}_j(v)\ddiff v \Big),
\end{align*}
where $\tilde{r}_j(\cdot)$ satisfies, for each $j\in[d]$, the ODE
\begin{align*}
    \frac{\ddiff \tilde{r}_j(v)}{\ddiff v} + \alpha_j \tilde{r}_j(v) + \big(1 + (y_j\hat{z}_j(t)-1)e^{-\mu_j v}\big)\beta(\tilde{\bm{r}}(v)) - 1 &= 0.
\end{align*}
Since $\hat{z}_j(t) = 1 + (z_j - 1)e^{-\mu_j\tau}$, substituting this into the ODE for $\tilde{r}_j(\cdot)$ and rearranging terms finishes the proof.
\end{proof}

\section{Computations Related to Theorems 1 and 2} \label{appendix: computations}

\subsection{Transform computations} \label{appendix:transform computations}

In this section, we provide the details behind taking the Laplace and $\bm{z}$-transform of Eqns.~\eqref{eq: PDE equation in terms of f} and \eqref{eq: PDE equation Laplace transformed} respectively.
First we show the Laplace transform, denoted by $\cL(\cdot)$, of \eqref{eq: PDE equation in terms of f}, which we restate here for convenience
\begin{align*}
    &\frac{\partial f(t,\bm{\nu},\bm{k})}{\partial t} - \sum_{j=1}^d \alpha_j \frac{\partial}{\partial \nu_j} \nu_j f(t,\bm{\nu},\bm{k})
    + \sum_{j=1}^d \alpha_j\overline{\lambda}_j \frac{\partial f(t,\bm{\nu},\bm{k})}{\partial \nu_j} \notag\\
    &= \sum_{j=1}^d \int_0^{\nu_d} \cdots \int_0^{\nu_1}y_j f(t, \bm{y},\bm{k} - \bm{e}_j)
    \frac{\partial^d }{\partial \nu_1 \cdots \partial \nu_d} \pp(\bm{B}_j \leqslant \bm{\nu - y}) \mathrm{d}y_1\cdots \mathrm{d}y_d \\
    & \quad + \sum_{j=1}^d \Big( f(t, \bm{\nu}, \bm{k + e}_j)(k_j + 1)\mu_j 
    - f(t,\bm{\nu},\bm{k}) (k_j \mu_j + \nu_j) \Big),
\end{align*}
and we introduce the shorthand notation
\begin{align*}
    \xi(t,\bm{s},\bm{k}) &:= \cL(f(t,\bm{\nu},\bm{k}))(\bm{s}) \\&= \int_0^\infty \cdots \int_0^\infty e^{-\bm{s}^\top \bm{\nu}} f(t,\bm{\nu},\bm{k})\,\mathrm{d}\nu_1 \cdots \mathrm{d}\nu_d 
    \equiv \int_{\bm{0}}^{\bm{\infty}} e^{-\bm{s}^\top \bm{\nu}} f(t,\bm{\nu},\bm{k}) \,\ddiff \bm{\nu}.
\end{align*}
We consider the term-by-term derivation in the equation that yields the transformed version as given in Eqn.~\eqref{eq: PDE equation Laplace transformed}.
For the first term, it is clear that
\begin{align*}
    \cL\Big(\frac{\partial f(t,\bm{\nu},\bm{k})}{\partial t}\Big)(\bm{s}) = \frac{\partial \xi(t,\bm{k},\bm{s})}{\partial t}.
\end{align*}
For the second term, we need to show
\begin{align*}
    -\cL\Big( \sum_{j=1}^d \alpha_j \frac{\partial}{\partial \nu_j} \nu_j f(t,\bm{\nu},\bm{k})\Big)(\bm{s}) = \sum_{j=1}^d \alpha_j s_j \frac{\partial}{\partial s_j}\xi(t,\bm{k},\bm{s}).
\end{align*}
The argument of the Laplace transform is 
\begin{align*}
    \sum_{j=1}^d \alpha_j \frac{\partial}{\partial \nu_j} \nu_j f(t,\bm{\nu},\bm{k})
    =  \sum_{j=1}^d \alpha_j f(t,\bm{\nu},\bm{k})  + \sum_{j=1}^d \alpha_j \nu_j \frac{\partial f(t,\bm{\nu},\bm{k})}{\partial \nu_j}.
\end{align*}
We then use the linearity of the Laplace transform and apply integration by parts to obtain
\begin{align*}
    &\sum_{j=1}^d \alpha_j \cL\big(f(t,\bm{\nu},\bm{k})\big)(\bm{s}) + \sum_{j=1}^d \alpha_j \cL\Big(\nu_j \frac{\partial f(t,\bm{\nu},\bm{k})}{\partial \nu_j}\Big)(\bm{s}) \\
    &= \sum_{j=1}^d \alpha_j \int_{\bm{0}}^{\bm{\infty}} e^{-\bm{s}^\top \bm{\nu}} f(t,\bm{\nu},\bm{k})\ddiff \bm{\nu}
    + \sum_{j=1}^d \alpha_j \int_{\bm{0}}^{\bm{\infty}} e^{-\bm{s}^\top \bm{\nu}} \nu_j \frac{\partial f(t,\bm{\nu},\bm{k})}{\partial\nu_j} \ddiff \bm{\nu} \\
    &=  \sum_{j=1}^d \alpha_j \int_{\bm{0}}^{\bm{\infty}} e^{-\bm{s}^\top \bm{\nu}} f(t,\bm{\nu},\bm{k})\ddiff \bm{\nu}
    + \sum_{j=1}^d \alpha_j \big[ \nu_je^{-\bm{s}^\top \bm{\nu}} f(t,\bm{\nu},\bm{k})\big]^{\bm{\infty}}_{\bm{0}} \\
    &- \sum_{j=1}^d \alpha_j \int_{\bm{0}}^{\bm{\infty}} \big(1- \nu_j s_j\big)e^{-\bm{s}^\top \bm{\nu}} f(t,\bm{\nu},\bm{k}) \ddiff \bm{\nu} \\
    &= 0 + \sum_{j=1}^d \alpha_j s_j \int_{\bm{0}}^{\bm{\infty}} \nu_j  e^{-\bm{s}^\top \bm{\nu}} f(t,\bm{\nu},\bm{k})\ddiff \bm{\nu} \\
    &= \sum_{j=1}^d \alpha_j s_j \int_{\bm{0}}^{\bm{\infty}} - \frac{\partial }{\partial s_j} e^{-\bm{s}^\top \bm{\nu}} f(t,\bm{\nu},\bm{k})\ddiff \bm{\nu} = -\sum_{j=1}^d \alpha_j s_j \frac{\partial}{\partial s_j}\xi(t,\bm{k},\bm{s}).
\end{align*}
For the third term, we need to show
\begin{align*}
    \sum_{j=1}^d \alpha_j \overline{\lambda}_j \cL\Big( \frac{\partial f(t,\bm{\nu},\bm{k})}{\partial\nu_j}\Big)(\bm{s}) = \sum_{j=1}^d \alpha_j\overline{\lambda}_j s_j \xi(t,\bm{k},\bm{s}).
\end{align*}
Using integration by parts and that $f(t,\bm{k},\bm{0})=0$, we have
\begin{align*}
    \cL\Big( \frac{\partial f(t,\bm{\nu},\bm{k})}{\partial\nu_j}\Big)(\bm{s})
    &=  \int_{\bm{0}}^{\bm{\infty}}  e^{-\bm{s}^\top \bm{\nu}}  \frac{\partial f(t,\bm{\nu},\bm{k})}{\partial\nu_j} \ddiff \bm{\nu} \\
    &= \big[e^{-\bm{s}^\top \bm{\nu}} f(t,\bm{\nu},\bm{k})\big]^{\bm{\infty}}_{\bm{0}} + s_j s_j \int_{\bm{0}}^{\bm{\infty}} e^{-\bm{s}^\top \bm{\nu}} f(t,\bm{\nu},\bm{k})\ddiff \bm{\nu} = 0 + s_j \xi(t,\bm{k},\bm{s}).
\end{align*}
For the fourth term, let us denote the probability density function of $\bm{B}_j = (B_{1j},\dots,B_{dj})^\top$ by $h_j(\cdot)$.
We need to show
\begin{align*}
    &\sum_{j=1}^d \cL\Big(\int_0^{\nu_d} \cdots \int_0^{\nu_1}y_j f(t, \bm{y},\bm{k} - \bm{e}_j)
    \frac{\partial^d }{\partial \nu_1 \cdots \partial \nu_d} \pp(\bm{B}_j \leqslant \bm{\nu - y}) \mathrm{d}y_1\cdots \mathrm{d}y_d\Big) \\
    &= \sum_{j=1}^d \cL\Big( \int_{\bm{0}}^{\bm{\nu}} y_j h_j(\bm{\nu}-\bm{y}) f(t,\bm{y},\bm{k}- \bm{e}_j) \ddiff \bm{y}\Big) \\
    &= -\sum_{j=1}^d \beta_j(\bm{s}) \frac{\partial \xi(t,\bm{s},\bm{k}-\bm{e}_j)}{\partial s_j},
\end{align*}
with $\beta_j(\bm{s}) = \ee[e^{-\bm{s}^\top\bm{B}_j}]$. 
To show that this holds, we need the property that relates convolutions with integration, which states that
\begin{align*}
    \int_{\rr^d} (f * g)(\bm{x})\mathrm{d}\bm{x} 
    = \Big(\int_{\rr^d} f(\bm{x})\mathrm{d}\bm{x}\Big) \Big(\int_{\rr^d} g(\bm{x})\mathrm{d}\bm{x}\Big),
\end{align*}
for given integrable functions $f$ and $g$.
Using this, we have
\begin{align*}
    -\sum_{j=1}^d& \beta_j(\bm{s}) \frac{\partial \xi(t,\bm{s},\bm{k}-\bm{e}_j)}{\partial s_j} \\
    &= - \sum_{j=1}^d \beta_j(\bm{s}) \int_{\bm{0}}^{\bm{\infty}} \frac{\partial}{\partial s_j} e^{-\bm{s}^\top \bm{\nu}} f(t, \bm{\nu},\bm{k} -\bm{e}_j) \ddiff \bm{\nu} \\
    &= \sum_{j=1}^d \int_{\bm{0}}^{\bm{\infty}} e^{-\bm{s}^\top \bm{\nu}} h_j(\bm{\nu}) \ddiff \bm{\nu} \int_{\bm{0}}^{\bm{\infty}} \nu_j e^{-\bm{s}^\top \bm{\nu}} f(t, \bm{\nu},\bm{k} -\bm{e}_j) \ddiff \bm{\nu} \\
    &= \sum_{j=1}^d \int_{\bm{0}}^{\bm{\infty}} \int_{\bm{0}}^{\bm{\infty}} e^{-\bm{s}^\top (\bm{\nu} - \bm{y})} h_j(\bm{\nu} - \bm{y}) y_j e^{-\bm{s}^\top\bm{y}} f(t,\bm{y},\bm{k}-\bm{e}_j)\ddiff \bm{y} \ddiff \bm{\nu} \\
    &\overset{(\star)}{=} \sum_{j=1}^d \int_{\bm{0}}^{\bm{\infty}} e^{-\bm{s}^\top \bm{\nu}} \int_{\bm{0}}^{\bm{\nu}} h_j(\bm{\nu} - \bm{y}) y_j f(t,\bm{y},\bm{k}-\bm{e}_j)\ddiff \bm{y} \ddiff \bm{\nu} \\
    &= \sum_{j=1}^d \cL\Big(\int_0^{\nu_d} \cdots \int_0^{\nu_1}y_j f(t, \bm{y},\bm{k} - \bm{e}_j)
    \frac{\partial^d }{\partial \nu_1 \cdots \partial \nu_d} \pp(\bm{B}_j \leqslant \bm{\nu - y}) \mathrm{d}y_1\cdots \mathrm{d}y_d\Big),
\end{align*}
where $(\star)$ holds because the non negativity $\pp(\bm{B}_j \geqslant \bm{0}) = 1$ implies $h_j(\bm{\nu}-\bm{y}) = 0$ if $\bm{\nu} \geqslant \bm{y}$.
The fifth term follows immediately by linearity since
\begin{align*}
    \sum_{j=1}^d \cL\big(f(t, \bm{\nu}, \bm{k + e}_j)(k_j + 1)\mu_j\big)(\bm{s})
    = \sum_{j=1}^d (k_j + 1)\mu_j \xi(t, \bm{s},\bm{k} + \bm{e}_j).
\end{align*}
Finally, the sixth term follows from the elementary computation
\begin{align*}
    - \sum_{j=1}^d  \cL\big(f(t,\bm{\nu},\bm{k}) (k_j \mu_j + \nu_j) \big)(\bm{s})
    &= - \sum_{j=1}^d k_j \mu_j \cL\big(f(t,\bm{\nu},\bm{k}) (\bm{s}) 
    - \sum_{j=1}^d k_j\nu_j\cL\big(f(t,\bm{\nu},\bm{k})\big)(\bm{s}) \\
    &= - \sum_{j=1}^d k_j \mu_j \xi(t,\bm{s},\bm{k}) 
    + \sum_{j=1}^d \int_{\bm{0}}^{\bm{\infty}} \frac{\partial}{\partial s_j} e^{-\bm{s}^\top \bm{\nu}} f(t,\bm{\nu},\bm{k})\ddiff \bm{\nu} \\
    &=- \sum_{j=1}^d k_j \mu_j \xi(t,\bm{s},\bm{k})  + \sum_{j=1}^d \frac{\partial \xi(t,\bm{s},\bm{k}) }{\partial s_j}.
\end{align*}

We can now derive Eqn.~\eqref{eq: PDE equation in terms of xi}, where we use the shorthand notation
\begin{align*}
    \zeta(t,\bm{s},\bm{z}) = \cZ\big(\xi(t,\bm{s},\cdot)\big)(\bm{z}) 
    = \sum_{k_1=0}^\infty \cdots \sum_{k_d=0}^\infty z_{1}^{k_1}\cdots z_d^{k_d} \xi(t,\bm{s},\bm{k})
    \equiv \sum_{\bm{k}\in \nn_0^d} \bm{z}^{\bm{k}}\xi(t,\bm{s},\bm{k}),
\end{align*}
with $\nn_0^d = \{0,1,2,\dots\}^d$.
As before, we take the term-by-term zeta transformation and show that we obtain Eqn.~\eqref{eq: PDE equation in terms of zeta}.
The first and second terms are immediate by construction and linearity, since we have
\begin{align*}
    \cZ\Big(\frac{\partial \xi(t,\bm{s},\cdot)}{\partial t}\Big)(\bm{z}) 
    &= \frac{\partial \zeta(t,\bm{s},\bm{z})}{\partial t},\\
    \sum_{j=1}^d\cZ\Big( \alpha_j s_j \frac{\partial \xi(t,\bm{s},\cdot)}{\partial s_j}\Big)(\bm{z})
    &= \sum_{j=1}^d (\alpha_js_j -1 ) \frac{\partial \zeta(t,\bm{s},\bm{z})}{\partial s_j}.
\end{align*}
We proceed by analyzing the third term,
with the notation $\nn_j^d = \{\bm{n}\in\nn^d : n_j \geqslant 1\}$ for $j\in[d]$, with mild abuse of notation,
\begin{align*}
    \sum_{j=1}^d \cZ\Big( \frac{\partial \xi(t,\bm{s},\cdot-\bm{e}_j)}{\partial s_j} \beta_j(\bm{s})\Big)(\bm{z})
    &= \sum_{j=1}^d \sum_{\bm{k}\in\nn_j^d} \bm{z}^{\bm{k}} \beta_j(\bm{s}) \frac{\partial \xi(t,\bm{s},\bm{k -e}_j)}{\partial s_j} \\
    &=  \sum_{j=1}^d \beta_j(\bm{s}) z_j \sum_{\bm{k} \in\nn_0^d} \bm{z}^{\bm{k}} \frac{\partial \xi(t,\bm{s},\bm{k})}{\partial s_j} = \sum_{j=1}^d \beta_j(\bm{s}) z_j \frac{\partial \zeta(t,\bm{s},\bm{z})}{\partial s_j}.
\end{align*}
For the fourth and fifth term, using elementary computations,
\begin{align*}
    &\sum_{j=1}^d \cZ\Big( \mu_j(k_j + 1)\xi(t,\bm{s},\bm{k}+\bm{e}_j) - \mu_jk_j \xi(t,\bm{s},\bm{k})\Big)(\bm{z}) \\ 
    &= \sum_{j=1}^d \mu_j \sum_{\bm{k}\in\nn_0^d} (k_j +1) \bm{z}^{\bm{k}} \xi(t,\bm{s},\bm{k}+\bm{e}_j) - k_j \bm{z}^{\bm{k}}\xi(t,\bm{s},\bm{k}) \\
    &= \sum_{j=1}^d \mu_j \sum_{\bm{k}\in\nn_j^d} k_j \bm{z}^{\bm{k} - \bm{e}_j}\xi(t,\bm{s},\bm{k}) - k_j z_j \bm{z}^{\bm{k} - \bm{e}_j}\xi(t,\bm{s},\bm{k}) \\
    &= \sum_{j=1}^d \mu_j(1-z_j) \frac{\partial}{\partial z_j} \sum_{\bm{k}\in\nn_0^d} \bm{z}^{\bm{k}} \xi(t,\bm{s},\bm{k}) = - \sum_{j=1}^d \mu_j(z_j - 1)\frac{\partial \zeta(t,\bm{s},\bm{z})}{\partial z_j}.
\end{align*}
Finally, the sixth term follows immediately from the definition since
\begin{align*}
    \sum_{j=1}^d \alpha_j  \overline{\lambda}_j s_j \cZ\Big(\xi(t,\bm{s},\cdot)\Big)(\bm{z})
    =\zeta(t,\bm{s},\bm{z}) \sum_{j=1}^d \alpha_j \overline{\lambda}_j s_j.
\end{align*}

\subsection{Joint moments: Computations} \label{appendix: joint moment compuations}

In this section, we provide the details behind the derivation of the PDE to ODE as given in Eqns.~\eqref{eq: PDE diff wrt s} and \eqref{eq: PDE diff wrt s and z}.
Since we are taking partial derivatives with respect to multiple variables, systematic bookkeeping is crucial.
Some terms are straightforward to compute, so we focus on the ones that require careful attention.
We first show the result and then provide details about the (relatively) complicated terms.

Differentiating Eqn.~\eqref{eq: PDE zeta rewritten} $n_{\lambda_1}, \dots, n_{\lambda_d}$ times with respect to $s_1,\dots,s_d$, respectively, and then substituting $\bm{s}=\bm{0}$, yields
\begin{align} \label{eq: PDE diff wrt s}
\begin{split}
    &\frac{\mathrm{d}}{\mathrm{d}t}\ee\Big[ \prod_{i=1}^d \lambda_i(t)^{n_{\lambda_i}} z_i^{Q_i(t)} \Big]
    + \sum_{j=1}^d n_{\lambda_j}\alpha_j \ee\Big[ \prod_{i=1}^d \lambda_i(t)^{n_{\lambda_i}} z_i^{Q_i(t)} \Big] \\
    &\quad- \sum_{l=1}^d  n_{\lambda_l} \sum_{j=1}^d  \ee\big[B_{lj}\big] \ee\Big[z_j\lambda_j(t) \prod_{i=1}^d \lambda_i(t)^{n_{\lambda_i} - \bm{1}_{\{i=l\}}} z_i^{Q_i(t)}\Big] \\
    &\quad + \sum_{j=1}^d \mu_j(z_j - 1)\ee\Big[ Q_j(t)\prod_{i=1}^d \lambda_i(t)^{n_{\lambda_i}} z_i^{Q_i(t) -\bm{1}_{\{i=j\}}} \Big] \\
    &= \sum_{j=1}^d (z_j-1)  \ee\Big[ \lambda_j(t) \prod_{i=1}^d \lambda_i(t)^{n_{\lambda_i}} z_i^{Q_i(t)} \Big] + \sum_{j=1}^d \alpha_j \overline{\lambda}_jn_{\lambda_j} \ee\Big[ \prod_{i=1}^d \lambda_i(t)^{n_{\lambda_i}-\bm{1}_{\{i=j\}}} z_i^{Q_i(t)} \Big] \\
    & \quad +\sum_{j=1}^d 
    \sum_{m_1=0}^{n_{\lambda_1}} \cdots \sum_{m_d=0}^{n_{\lambda_d}} \bm{1}_{\{m\leqslant n_\lambda -2\}} 
    \prod_{k=1}^d {n_{\lambda_k}\choose m_k}\ee\Big[z_j \prod_{i=1}^d B_{ij}^{n_{\lambda_i} - m_i} \lambda_i(t)^{m_i + \bm{1}_{\{i=j\}}} z_i^{Q_i(t)}\Big],
\end{split}
\end{align}
where $m= \sum_{i=1}^dm_i$ and we collected the $\ee[B_{ij}]$ combinations of first order on the left-hand side and higher orders on the right-hand side.
Then take Eqn.~\eqref{eq: PDE diff wrt s} and differentiate $n_{Q_1},\dots,n_{Q_d}$ times with respect to $z_1,\dots,z_d$, respectively, and substitute $\bm{z}=\bm{1}$. After elementary calculus,
\begin{align} 
    &\frac{\mathrm{d}}{\mathrm{d}t}\psi_t( \bm{n_{\lambda}},\bm{n_Q})
    + \sum_{j=1}^d \big(n_{\lambda_j} (\alpha_j -\ee\big[B_{jj}\big]) + n_{Q_j} \mu_j\big) \psi_t( \bm{n_{\lambda}}, \bm{n_Q})\notag \\
    &= \sum_{j=1}^d \sum_{\substack{i=1 \\ i\neq j}}^d  n_{\lambda_i}   \ee\big[B_{ij}\big] \psi_t( \bm{n_{\lambda}} - \bm{e}_i + \bm{e}_j, \bm{n_Q}) + \sum_{j=1}^d n_{Q_j} \psi_t( \bm{n_{\lambda}}+\bm{e_j}, \bm{n_Q} - \bm{e}_j)\\
    &\quad  + \sum_{j=1}^d \alpha_j \overline{\lambda}_jn_{\lambda_j}  \psi_t(  \bm{n_{\lambda}}-\bm{e}_j,\bm{n_Q}) + \sum_{i=1}^d \sum_{j=1}^d n_{\lambda_i} n_{Q_j} \ee\big[B_{ij}\big]  \psi_t( \bm{n_{\lambda}} - \bm{e}_i + \bm{e}_j,\bm{n_Q}-\bm{e}_j) \notag\\
    & \quad + \sum_{j=1}^d\sum_{m_1=0}^{n_{\lambda_1}} \cdots \sum_{m_d=0}^{n_{\lambda_d}} \bm{1}_{\{m\leqslant n_\lambda -2\}}   \prod_{k=1}^d {n_{\lambda_k}\choose m_k}\Big\{ n_{Q_j} \prod_{i=1}^d\ee\big[B_{ij}^{n_{\lambda_i} - m_i}\big]
    \psi_t( \bm{m}+\bm{e}_j,\bm{n_Q}-\bm{e}_j)\notag \\
    &\quad \quad +\prod_{i=1}^d \ee\big[B_{ij}^{n_{\lambda_i} - m_i}\big] \psi_t(  \bm{m}+\bm{e}_j,\bm{n_Q})\Big\}\notag.
\end{align}

To obtain the ODE in Eqn.~\eqref{eq: PDE diff wrt s}, the starting point is Eqn.~\eqref{eq: PDE zeta rewritten}.
Differentiate $n_{\lambda_1},\dots,n_{\lambda_d}$ times with respect to $s_1,\dots,s_d$ respectively, and then substitute $\bm{s}=\bm{0}$. 
The terms that are not immediate to compute are those where we need to apply the product rule repeatedly.
Consider the computation of
\begin{align*}
    \frac{\partial^{n_{\lambda_1}} \cdots \partial^{n_{\lambda_d}}}{\partial s_1^{n_{\lambda_1}} \cdots \partial s_d^{n_{\lambda_d}}}
    \sum_{j=1}^d \alpha_js_j\ee\big[\lambda_j(t)e^{-\bm{s}^\top\bm{\lambda}(t)}\prod_{n=1}^dz_n^{Q_n(t)}\big].
\end{align*}
We first focus on differentiation with respect to the first component, yielding
\begin{align*}
    &\frac{\partial^{n_{\lambda_1}}}{\partial s_1^{n_{\lambda_1}}}\sum_{j=1}^d \alpha_js_j\ee\big[\lambda_j(t)e^{-\bm{s}^\top\bm{\lambda}(t)}\prod_{n=1}^dz_n^{Q_n(t)}\big] \\
    &= \frac{\partial^{n_{\lambda_1}-1}}{\partial s_1^{n_{\lambda_1}-1}} \alpha_1 \ee\big[\lambda_1(t)e^{-\bm{s}^\top\bm{\lambda}(t)}\prod_{n=1}^dz_n^{Q_n(t)}\big]
    - \frac{\partial^{n_{\lambda_1}-1}}{\partial s_1^{n_{\lambda_1}-1}} \sum_{j=1}^d \alpha_js_j\ee\big[\lambda_1(t)\lambda_j(t)e^{-\bm{s}^\top\bm{\lambda}(t)}\prod_{n=1}^dz_n^{Q_n(t)}\big] \\
    &= - 2\frac{\partial^{n_{\lambda_1}-2}}{\partial s_1^{n_{\lambda_1}-2}}\alpha_1 \ee\big[\lambda_1(t)^2e^{-\bm{s}^\top\bm{\lambda}(t)}\prod_{n=1}^dz_n^{Q_n(t)}\big]
    + \frac{\partial^{n_{\lambda_1}-2}}{\partial s_1^{n_{\lambda_1}-2}} \sum_{j=1}^d \alpha_js_j\ee\big[\lambda_1(t)^2\lambda_j(t)e^{-\bm{s}^\top\bm{\lambda}(t)}\prod_{n=1}^dz_n^{Q_n(t)}\big] \\
    &\ \,\vdots \\
    &= n_{\lambda_1} (-1)^{n_{\lambda_1}-1}\alpha_1 \ee\big[\lambda_1(t)^{n_{\lambda_1}}e^{-\bm{s}^\top\bm{\lambda}(t)}\prod_{n=1}^dz_n^{Q_n(t)}\big] 
    +\sum_{j=1}^d \alpha_js_j\ee\big[\lambda_1(t)^{n_{\lambda_1}}\lambda_j(t)e^{-\bm{s}^\top\bm{\lambda}(t)}\prod_{n=1}^dz_n^{Q_n(t)}\big].
\end{align*}
Note that all the terms in the latter sum vanish when we substitute $\bm{s}=\bm{0}$.
An analogous expression holds for the other components.
If we now combine the differentiation with respect to all components and substitute $\bm{s}=\bm{0}$, we have
\begin{align*}
    &\frac{\partial^{n_{\lambda_1}} \cdots \partial^{n_{\lambda_d}}}{\partial s_1^{n_{\lambda_1}} \cdots \partial s_d^{n_{\lambda_d}}}
    \sum_{j=1}^d \alpha_js_j\ee\big[\lambda_j(t)e^{-\bm{s}^\top\bm{\lambda}(t)}\prod_{n=1}^dz_n^{Q_n(t)}\big]
    =\sum_{j=1}^d n_{\lambda_j}\alpha_j \ee\Big[ \prod_{i=1}^d \lambda_i(t)^{n_{\lambda_i}} z_i^{Q_i(t)} \Big].
\end{align*}

Another, more complicated, term we need to compute is
\begin{align*}
    \frac{\partial^{n_{\lambda_1}} \cdots \partial^{n_{\lambda_d}}}{\partial s_1^{n_{\lambda_1}} \cdots \partial s_d^{n_{\lambda_d}}}\sum_{j=1}^d z_j\beta_j(\bm{s})\ee\big[\lambda_j(t)e^{-\bm{s}^\top\bm{\lambda}(t)}\prod_{n=1}^dz_n^{Q_n(t)}\big],
\end{align*}
with $\beta_j(\bm{s}) = \ee\big[e^{-\bm{s}^\top \bm{B}_j}\big]$.
It is clear that taking higher order derivatives means that we have to successively apply the product rule.
Moreover, since we are taking partial derivatives with respect to multiple components, we obtain a large number of cross terms.
Let us focus on the first component, which yields
\begin{align*}
    &\frac{\partial^{n_{\lambda_1}}}{\partial s_1^{n_{\lambda_1}}}\sum_{j=1}^d z_j\beta_j(\bm{s})\ee\big[\lambda_j(t)e^{-\bm{s}^\top\bm{\lambda}(t)}\prod_{n=1}^dz_n^{Q_n(t)}\big] \\
    &= (-1)^1\sum_{j=1}^d z_j \frac{\partial^{n_{\lambda_1}-1}}{\partial s_1^{n_{\lambda_1}-1}} \ee\big[B_{1j}e^{-\bm{s}^\top\bm{B}_j}\big] \ee\big[\lambda_j(t)e^{-\bm{s}^\top\bm{\lambda}(t)}\prod_{n=1}^dz_n^{Q_n(t)}\big] \\ 
    &\quad + (-1)^1\sum_{j=1}^d z_j\frac{\partial^{n_{\lambda_1}-1}}{\partial s_1^{n_{\lambda_1}-1}}\ee\big[e^{-\bm{s}^\top \bm{B}_j}\big] \ee\big[\lambda_1(t)\lambda_j(t)e^{-\bm{s}^\top\bm{\lambda}(t)}\prod_{n=1}^dz_n^{Q_n(t)}\big] \\
    &= (-1)^2\sum_{j=1}^d z_j \frac{\partial^{n_{\lambda_1}-2}}{\partial s_1^{n_{\lambda_1}-2}} \ee\big[B_{1j}^2e^{-\bm{s}^\top\bm{B}_j}\big] \ee\big[\lambda_j(t)e^{-\bm{s}^\top\bm{\lambda}(t)}\prod_{n=1}^dz_n^{Q_n(t)}\big] \\ 
    &\quad +2(-1)^2\sum_{j=1}^d z_j \frac{\partial^{n_{\lambda_1}-2}}{\partial s_1^{n_{\lambda_1}-2}} \ee\big[B_{1j}e^{-\bm{s}^\top\bm{B}_j}\big] \ee\big[\lambda_1(t)\lambda_j(t)e^{-\bm{s}^\top\bm{\lambda}(t)}\prod_{n=1}^dz_n^{Q_n(t)}\big]\\
    &\quad \quad +(-1)^2\sum_{j=1}^d z_j \frac{\partial^{n_{\lambda_1}-2}}{\partial s_1^{n_{\lambda_1}-2}} \ee\big[e^{-\bm{s}^\top\bm{B}_j}\big] \ee\big[\lambda_1(t)^2\lambda_j(t)e^{-\bm{s}^\top\bm{\lambda}(t)}\prod_{n=1}^dz_n^{Q_n(t)}\big] \\ 
    &\ \, \vdots \\
    &= n_{\lambda_1}(-1)^{n_{\lambda_1}}\sum_{j=1}^d z_j \ee\big[B_{1j}e^{-\bm{s}^\top\bm{B}_j}\big] \ee\big[\lambda_1(t)^{n_{\lambda_1}-1}\lambda_j(t)e^{-\bm{s}^\top\bm{\lambda}(t)}\prod_{n=1}^dz_n^{Q_n(t)}\big]\\
    &\quad + (-1)^{n_{\lambda_1}} \sum_{j=1}^d z_j \bm{1}_{\{n_{\lambda_1}\geqslant 2\}} \sum_{m_1=0}^{n_{\lambda_1}-2} {n_{\lambda_1}\choose m_1} \ee\big[B_{1j}^{n_{\lambda_1}-m_1}\big]\ee\big[\lambda_1(t)^{m}\lambda_j(t)e^{-\bm{s}^\top\bm{\lambda}(t)}\prod_{n=1}^dz_n^{Q_n(t)}\big]\\
    &\quad\quad + (-1)^{n_{\lambda_1}} \sum_{j=1}^d z_j\ee\big[\lambda_1(t)^{n_{\lambda_1}}\lambda_j(t)e^{-\bm{s}^\top\bm{\lambda}(t)}\prod_{n=1}^dz_n^{Q_n(t)}\big],
\end{align*}
since the number of terms is doubled in every step of the derivation.
The computation for the other components is entirely analogous.
Upon taking the joint derivative and substituting $\bm{s}=\bm{0}$, we obtain
\begin{align*}
    &\frac{\partial^{n_{\lambda_1}} \cdots \partial^{n_{\lambda_d}}}{\partial s_1^{n_{\lambda_1}} \cdots \partial s_d^{n_{\lambda_d}}}\sum_{j=1}^d z_j\beta_j(\bm{s})\ee\big[\lambda_j(t)e^{-\bm{s}^\top\bm{\lambda}(t)}\prod_{n=1}^dz_n^{Q_n(t)}\big] \\
    &=\sum_{l=1}^d  n_{\lambda_l} \sum_{j=1}^d  \ee\big[B_{lj}\big] \ee\Big[z_j\lambda_j(t) \prod_{i=1}^d \lambda_i(t)^{n_{\lambda_i} - \bm{1}_{\{i=l\}}} z_i^{Q_i(t)}\Big]
    + \sum_{j=1}^d z_j \ee\Big[ \lambda_j(t) \prod_{i=1}^d \lambda_i(t)^{n_{\lambda_i}} z_i^{Q_i(t)} \Big] \\
    &\quad +\sum_{j=1}^d 
    \sum_{m_1=0}^{n_{\lambda_1}} \cdots \sum_{m_d=0}^{n_{\lambda_d}} \bm{1}_{\{m\leqslant n_\lambda -2\}} 
    \prod_{k=1}^d {n_{\lambda_k}\choose m_k}\ee\Big[z_j \prod_{i=1}^d B_{ij}^{n_{\lambda_i} - m_i} \lambda_i(t)^{m_i + \bm{1}_{\{i=j\}}} z_i^{Q_i(t)}\Big],
\end{align*}
where $m = m_1+\cdots m_d$.

We now focus on the terms to obtain the ODE in Eqn.~\eqref{eq: PDE diff wrt s and z}.
The starting point is Eqn.~\eqref{eq: PDE diff wrt s}, which we differentiate $n_{Q_1},\dots,n_{Q_d}$ times with respect to $z_1,\dots,z_d$ respectively and substitute $\bm{z}=\bm{1}$.
There are multiple terms in \eqref{eq: PDE diff wrt s} that require the product rule when differentiating.
We consider one such term and take the appropriate derivative, i.e.,
\begin{align*}
    \frac{\partial^{n_{Q_1}} \cdots \partial^{n_{Q_d}}}{\partial z_1^{n_{Q_1}} \cdots \partial z_d^{n_{Q_d}}}
    \sum_{j=1}^d \mu_j(z_j - 1)\ee\Big[ Q_j(t)\prod_{i=1}^d \lambda_i(t)^{n_{\lambda_i}} z_i^{Q_i(t) -\bm{1}_{\{i=j\}}} \Big].
\end{align*}
Again, we focus on differentiation with respect to the first component, which yields
\begin{align*}
    &\frac{\partial^{n_{Q_1}}}{\partial z_1^{n_{Q_1}}} \sum_{j=1}^d \mu_j(z_j - 1)\ee\Big[ Q_j(t)\prod_{i=1}^d \lambda_i(t)^{n_{\lambda_i}} z_i^{Q_i(t) -\bm{1}_{\{i=j\}}} \Big] \\
    &=\frac{\partial^{n_{Q_1}-1}}{\partial z_1^{n_{Q_1}-1}} \mu_1 \ee\Big[ Q_1(t)\prod_{i=1}^d \lambda_i(t)^{n_{\lambda_i}} z_i^{Q_i(t) -\bm{1}_{\{i=1\}}} \Big] \\
    &\quad + \frac{\partial^{n_{Q_1}-1}}{\partial z_1^{n_{Q_1}-1}} \mu_1(z_1-1) \ee\Big[Q_1(t)(Q_1(t)-1)\prod_{i=1}^d \lambda_i(t)^{n_{\lambda_i}} z_i^{Q_i(t) -2\bm{1}_{\{i=1\}}} \Big] \\
    &\quad\quad + \sum_{j=2}^d \frac{\partial^{n_{Q_1}-1}}{\partial z_1^{n_{Q_1}-1}}\mu_j(z_j - 1)\ee\Big[Q_j(t)Q_1(t)\prod_{i=1}^d \lambda_i(t)^{n_{\lambda_i}} z_i^{Q_i(t) -\bm{1}_{\{i=j\}} - \bm{1}_{\{i=1\}}} \Big] \\
    &=2\frac{\partial^{n_{Q_1}-2}}{\partial z_1^{n_{Q_1}-2}} \mu_1 \ee\Big[ Q_1(t)(Q_1(t)-1)\prod_{i=1}^d \lambda_i(t)^{n_{\lambda_i}} z_i^{Q_i(t) -2\bm{1}_{\{i=1\}}} \Big] \\
    &\quad + \frac{\partial^{n_{Q_1}-2}}{\partial z_1^{n_{Q_1}-2}} \mu_1(z_1-1) \ee\Big[Q_1(t)(Q_1(t)-1)(Q_1(t)-2)\prod_{i=1}^d \lambda_i(t)^{n_{\lambda_i}} z_i^{Q_i(t) -3\bm{1}_{\{i=1\}}} \Big] \\
    &\quad \quad \sum_{j=2}^d \frac{\partial^{n_{Q_1}-2}}{\partial z_1^{n_{Q_1}-2}}\mu_j(z_j - 1)\ee\Big[Q_j(t)Q_1(t)(Q_1(t)-1)\prod_{i=1}^d \lambda_i(t)^{n_{\lambda_i}} z_i^{Q_i(t) -\bm{1}_{\{i=j\}} - 2\bm{1}_{\{i=1\}}} \Big] \\
    &\ \, \vdots \\
    &= n_{Q_1} \mu_1 \ee\Big[ Q_1(t)^{[n_{Q_1}]}\prod_{i=1}^d \lambda_i(t)^{n_{\lambda_i}} z_i^{Q_i(t) -n_{Q_1}\bm{1}_{\{i=1\}}} \Big],
\end{align*}
where we substituted $\bm{z}=\bm{1}$ in the last step, canceling out all the terms that contain the factor $(z_j-1)$. 
We can compute the derivatives with respect to other components in a similar manner, which results in
\begin{align*}
    &\frac{\partial^{n_{Q_1}} \cdots \partial^{n_{Q_d}}}{\partial z_1^{n_{Q_1}} \cdots \partial z_d^{n_{Q_d}}}
    \sum_{j=1}^d \mu_j(z_j - 1)\ee\Big[ Q_j(t)\prod_{i=1}^d \lambda_i(t)^{n_{\lambda_i}} z_i^{Q_i(t) -\bm{1}_{\{i=j\}}} \Big] \\
    &= \sum_{j=1}^d n_{Q_j} \ee\Big[ \prod_{i=1}^d \lambda_i^{n_{\lambda_i}} Q_i(t)^{[n_{Q_i}]}\Big]
    = \sum_{j=1}^d n_{Q_j} \psi_t(\bm{n_Q},\bm{n_\lambda}).
\end{align*}

\section{First and Second Order Transient and Stationary Moments}\label{appendix: trans_stat_moments}

In this appendix we present an illustration concerning moments of order $n\in\{1,2\}$.
We introduce the relevant objects along the way, starting with the matrix $\ee[\bm{B}] = (\ee[B_{ij}])_{i,j\in[d]}$
and the diagonal matrices
\begin{align*}
    \bm{D}_{\alpha} := \text{diag}(\alpha_1,\alpha_2,\dots,\alpha_d),
    \quad 
    \bm{D}_{\mu} := \text{diag}(\mu_1,\mu_2,\dots,\mu_d).
\end{align*}

\subsection{Transient moments} \label{sec: transient moments d-dim}
We focus on the transient moments $\psi_t(\bm{n_Q}, \bm{n_\lambda})$, where now $\bm{n_Q} = (n_{Q_1},\dots,n_{Q_d})$ and $\bm{n_\lambda} = (n_{\lambda_1}, \dots, n_{\lambda_d})$.
We separately consider the cases $n=1$ and $n=2$.
To describe the joint moments of equal order in vector/matrix-form,
it turns out that for $n=1$ we need a stacked \textit{vector}, and for $n=2$ a stacked \textit{matrix}, introduced in detail below.

For $n=1$ we define the stacked vector
\begin{align} \label{eq: def transient order 1 d dim stacked vector}
    \bm{\Sigma}_t^{(1)} := \Big( \ee[\bm{\lambda}(t)], \ee[\bm{Q}(t)]\Big)^\top.
\end{align}
For each entry of the vector, we use Eqn.~\eqref{eq: PDE diff wrt s and z} to obtain the vector-valued ODEs
\begin{align}\label{eq: transient ODE 1st order d-dim}
\begin{split}
    \frac{\ddiff}{\ddiff t} \ee[\bm{\lambda}(t)]
    &= \big(\ee[\bm{B}]-\bm{D}_\alpha\big)
    \ee[\bm{\lambda}(t)]+ \bm{L}^{(0,1)},\\
    \frac{\ddiff}{\ddiff t} \ee[\bm{Q}(t)]  
    &= -\bm{D}_\mu \ee[\bm{Q}(t)] + \ee[\bm{\lambda}(t)],
\end{split}
\end{align}
where $\bm{L}^{(0,1)} = \big( \alpha_1\overline{\lambda}_{1}, \alpha_2\overline{ \lambda}_{2},\dots, \alpha_d \overline{\lambda}_{d}\big)^\top$.
It is directly verified that \eqref{eq: transient ODE 1st order d-dim} is solved by
\begin{align} \label{eq: solution transient order 1, d-dim}
\begin{split}
    \ee[\bm{\lambda}(t)]&= e^{t(\ee[\bm{B}]-\bm{D}_\alpha)}\bm{\overline{\lambda}} + \int_0^t e^{(t-s)(\ee[\bm{B}]-\bm{D}_\alpha)} \ddiff s\, (\bm{\alpha} \odot \bm{\overline{\lambda}}) \\
    &= e^{t(\ee[\bm{B}]-\bm{D}_\alpha)}\bm{\overline{\lambda}} + (\ee[\bm{B}]-\bm{D}_\alpha)^{-1}\big( e^{t(\ee[\bm{B}]-\bm{D}_\alpha)} -\bm{I}\big) (\bm{\alpha} \odot \bm{\overline{\lambda}}), \\
    \ee[\bm{Q}(t)]
    &=  \int_0^t e^{-(t-s)\bm{D}_\mu} \ee[\bm{\lambda}(s)] \ddiff s.
\end{split}
\end{align}

We proceed with $n=2$.
As indicated at the start of this subsection, in this case we should work with a stacked matrix.
To this end, define
\begin{align*} 
    \ee\big[\bm{Q}(t)^{[2]}\big] &:= \ee\big[ \bm{Q}(t)\bm{Q}(t)^\top \big] 
    - \text{diag}(\ee[\bm{Q}(t)])
    \\&=
    \ee
    \begin{bmatrix}
    Q_1(t)^{[2]} & Q_1(t)Q_2(t) & \cdots & Q_1(t)Q_d(t) \\
    Q_2(t)Q_1(t) & Q_2(t)^{[2]} & \cdots & Q_2(t)Q_d(t) \\
    \vdots & \vdots & \ddots & \vdots \\
    Q_d(t)Q_1(t) & Q_d(t)Q_2(t) &\cdots & Q_d(t)^{[2]}
    \end{bmatrix},
\end{align*}
and we also consider the objects $\ee\big[ \bm{\lambda}(t)\bm{Q}(t)^\top\big]$ and $\ee\big[ \bm{\lambda}(t)\bm{\lambda}(t)^\top\big]$, which are all $d\times d$-matrices.
In addition, we define the stacked matrix $\bm{\Sigma}_t^{(2)}$ given by
\begin{align}\label{eq: def 2nd order stacked matrix d-dim}
    \bm{\Sigma}_t^{(2)} := 
    \ee\big[ \bm{\lambda}(t)\bm{\lambda}(t)^\top\big] \oplus
    \ee\big[ \bm{\lambda}(t)\bm{Q}(t)^\top\big] \oplus 
    \ee\big[\bm{Q}(t)^{[2]}\big],
\end{align}
where $\oplus$ indicates the direct sum, so that $\bm{\Sigma}_t^{(2)}$ is a $3d\times3d$-matrix.
For each entry of a submatrix, we derive its associated ODE from Eqn.~\eqref{eq: PDE diff wrt s and z} which we combine into matrix-valued ODEs.
We thus find the matrix-valued ODEs
\begin{align} \label{eq: transient ODE 2nd order d-dim}
    \frac{\ddiff}{\ddiff t} \ee\big[ \bm{\lambda}(t)\bm{\lambda}(t)^\top\big]
    &= \big(\ee[\bm{B}] -\bm{D}_{\alpha}\big) \ee\big[ \bm{\lambda}(t)\bm{\lambda}(t)^\top\big] 
    +\ee\big[ \bm{\lambda}(t)\bm{\lambda}(t)^\top\big] \big(\ee[\bm{B}] -\bm{D}_{\alpha}\big)^\top  \notag \\
    &\quad + \ee\big[\bm{B} \,\textrm{diag}\big(  \ee[\bm{\lambda}(t)]\big)\bm{B}^\top\big]
    + \bm{D}_\alpha\big( \bm{\bar{\lambda}}  \ee[\bm{\lambda}(t)]^\top\big) + \big(  \ee[\bm{\lambda}(t)] \bm{\bar{\lambda}}^\top\big)\bm{D}_\alpha, \notag \\
     \frac{\ddiff}{\ddiff t} \ee\big[ \bm{\lambda}(t)\bm{Q}(t)^\top\big]
    &= \big(\ee[\bm{B}]- \bm{D}_{\alpha}\big) \ee\big[ \bm{\lambda}(t)\bm{Q}(t)^\top\big] - \ee\big[ \bm{\lambda}(t)\bm{Q}(t)^\top\big] \bm{D}_\mu  +\ee\big[ \bm{\lambda}(t)\bm{\lambda}(t)^\top\big] \\
    &\quad + (\bm{\alpha} \odot \bm{\bar{\lambda}})  \ee[\bm{Q}(t)]^\top
    + \ee[\bm{B}]\,\textrm{diag}\big( \ee[\bm{\lambda}(t)]\big),\notag \\
    \frac{\ddiff}{\ddiff t} \ee\big[\bm{Q}(t)^{[2]}\big]
    &= -\bm{D}_\mu \ee\big[\bm{Q}(t)^{[2]}\big] - \ee\big[\bm{Q}(t)^{[2]}\big]\bm{D}_\mu + \ee\big[ \bm{\lambda}(t)\bm{Q}(t)^\top\big] +\big(\ee\big[ \bm{\lambda}(t)\bm{Q}(t)^\top\big]\big)^\top.\notag
\end{align}

We end our account of the transient moments with a series of brief remarks.
The ODEs for $\bm{\Sigma}_t^{(1)}$ and $\bm{\Sigma}_t^{(2)}$ are related to those derived in Lemmas 1 and 3 of \cite{DFZ15}.
Concretely, the solution for the first moment $\ee[\bm{\lambda}(t)]$ agrees with Eqn.~(8) in \cite{DFZ15}.
Furthermore, by taking the limit $\bm{\mu} \downarrow \bm{0}$ in our expression for $\bm{Q}(t)$, we obtain 
\begin{align}
    \ee\big[\bm{N}(t)\big] 
    =\:& \big( \ee[\bm{B}]-\bm{D}_\alpha\big)^{-1}\big( e^{t\,(\ee[\bm{B}]-\bm{D}_\alpha)} - I\big)\bm{\overline{\lambda}}\:+ \notag \\
    & \big( \ee[\bm{B}]-\bm{D}_\alpha\big)^{-2} \big( e^{t\,(\ee[\bm{B}]-\bm{D}_\alpha)} - I\big) (\bm{\alpha} \odot \bm{\overline{\lambda}})  + t\,\big(\ee[\bm{B}]-\bm{D}_\alpha\big)^{-1} (\bm{\alpha} \odot \bm{\overline{\lambda}}), 
\end{align}
which agrees with the result in Eqn.~(10) in \cite{DFZ15}.
Regarding the second order moments, upon taking $\bm{\mu}\downarrow 0$ in Eqn.~\eqref{eq: transient ODE 2nd order d-dim} we recover the expressions in Lemma~3 of \cite{DFZ15} (where it is noted that an elementary conversions needs to be performed, as we work with $\ee\big[ \bm{N}(t)^{[2]}\big]$ and~\cite{DFZ15} with $\ee\big[ \bm{N}(t)\bm{N}(t)^\top\big]$).

\subsection{Stationary moments} \label{sec: stationary moments d-dim}

We continue by considering the joint stationary moments of order at most $2$.
We adopt the notation used in Subsection \ref{sec: transient moments d-dim}.

For order $n=1$, define the stationary version of Eqn.~\eqref{eq: def transient order 1 d dim stacked vector}:
\begin{align}
    \bm{\Sigma}^{(1)} := \lim_{t\to\infty} \bm{\Sigma}^{(1)}_t = \Big( \ee[\bm{\lambda}], \ee[\bm{Q}] \Big)^\top.
\end{align}
We derive from Eqn.~\eqref{eq: PDE diff wrt s and z in steady-state} that the elements of this stacked vector satisfy
\begin{align}
    \ee[\bm{\lambda}] = -\big(\ee[\bm{B}] - \bm{D}_{\alpha}\big)^{-1}\bm{L}^{(0,1)},\qquad
    \ee[\bm{Q}] = \bm{D}_{\mu}^{-1}\ee[\bm{\lambda}].
\end{align}

For order $n=2$, we define the stacked matrix
\begin{align}
    \bm{\Sigma}^{(2)} := \lim_{t\to\infty}  \bm{\Sigma}_t^{(2)}
    = \ee[\bm{\lambda}\bm{\lambda}^\top] \oplus \ee[\bm{\lambda}\bm{Q}^\top] \oplus  \ee[\bm{Q}^{[2]}],
\end{align}
From the procedure followed for the transient moments, in combination with Eqn.~\eqref{eq: PDE diff wrt s and z in steady-state}, we conclude that, with $\ee[\bm{\lambda}]$ and $\ee[\bm{Q}]$ given above,
\begin{align} 
    0 &= \big(\ee[\bm{B}] -\bm{D}_{\alpha}\big) \ee[\bm{\lambda}\bm{\lambda}^\top] + \ee[\bm{\lambda}\bm{\lambda}^\top] \big(\ee[\bm{B}] -\bm{D}_{\alpha}\big)^\top + \ee\big[\bm{B} \text{diag}\big( \ee[\bm{\lambda}]\big)\bm{B}^\top\big] \notag \\
    & \quad + \bm{D}_\alpha\big( \bar{\bm{\lambda}} \ee[\bm{\lambda}]^\top\big) + \big( \ee[\bm{\lambda}]\bm{\overline{\lambda}}^\top\big)\bm{D}_\alpha,\notag  \\
    0 &= \big(\ee[\bm{B}]- \bm{D}_{\alpha}\big)\ee[\bm{\lambda}\bm{Q}^\top] - \ee[\bm{\lambda}\bm{Q}^\top]\bm{D}_{\mu} + \ee[\bm{\lambda}\bm{\lambda}^\top] 
    +  (\bm{\alpha} \odot \bar{\bm{\lambda}}) \ee[\bm{Q}]^\top + \ee[\bm{B}] \text{diag}\big(\ee[\bm{\lambda}]\big)^\top,  \notag \\
    0 &= -\bm{D}_\mu \ee[\bm{Q}^{[2]}] - \ee[\bm{Q}^{[2]}]\bm{D}_\mu + \ee[\bm{\lambda}\bm{Q}^\top] +\big(\ee[\bm{\lambda}\bm{Q}^\top]\big)^\top.
\label{eq: stationary moment order 2 d-dim}
\end{align}
The matrix-valued equations in \eqref{eq: stationary moment order 2 d-dim} are all \textit{Sylvester equations}, i.e., equations of the form
\begin{align}
    \bm{A}\bm{X} + \bm{X}\bm{B} = \bm{C},
\end{align}
for known matrices $\bm{A}$, $\bm{B}$, $\bm{C}$, with the matrix $\bm{X}$ being unknown. 
It is a known result that a unique solution for $\bm{X}$ exists if and only $\bm{A}$ and $-\bm{B}$ do not share any eigenvalue.

We conclude this subsection with two results on higher order stationary moments.
By applying Eqn.~\eqref{eq: PDE diff wrt s and z in steady-state}, we can obtain expressions for the moments
\begin{align}
    \psi(\bm{0},\bm{n_\lambda}) = \ee\big[\prod_{i=1}^d \lambda_i^{n_{\lambda_i}}\big], \quad \psi(\bm{n_Q},\bm{0}) = \ee\big[\prod_{i=1}^d Q_i^{[n_{Q_i}]}\big],
\end{align}
by straightforward substitution.
Indeed, for fixed $n_{\lambda} \in \nn$, we substitute  $n_{Q_j} \equiv 0$ for all $j\in[d]$ in Eqn.~\eqref{eq: PDE diff wrt s and z in steady-state}.
Rearranging terms, we obtain
\begin{align}
    &\psi(\bm{0},\bm{n_\lambda})
    = \Big( \sum_{j=1}^d n_{\lambda_j} (\alpha_j -\ee[B_{jj}])\Big)^{-1} \sum_{j=1}^d \Big\{\sum_{\substack{i=1 \\ i\neq j}}^d  n_{\lambda_i}   \ee\big[B_{ij}\big] \psi(\bm{0},\bm{n_{\lambda}} - \bm{e}_i + \bm{e}_j) \\
    & + \alpha_j \overline{\lambda}_jn_{\lambda_j}  \psi(\bm{0}, \bm{n_{\lambda}}-\bm{e}_j) 
    + \sum_{m_1=0}^{n_{\lambda_1}} \cdots \sum_{m_d=0}^{n_{\lambda_d}} \bm{1}_{\{m\leqslant n_\lambda -2\}}   \prod_{k=1}^d {n_{\lambda_k}\choose m_k}\prod_{i=1}^d \ee\big[B_{ij}^{n_{\lambda_i} - m_i}\big] \psi(\bm{0}, \bm{m}+\bm{e}_j)\Big\}. \notag
\end{align}
Observe that in order to obtain a final closed-form expression for $\psi(\bm{0},\bm{n_\lambda})$, we need to solve a linear system of equations of equal order moments, i.e., the $\psi(\bm{0},\bm{n_\lambda}-\bm{e}_i + \bm{e}_j)$ terms.

A similar result holds for the joint moments of $\bm{Q}$: for fixed $n_Q \in \nn$, we substitute $n_{\lambda_j} \equiv 0$ for all $j\in[d]$ in Eqn.~\eqref{eq: PDE diff wrt s and z in steady-state}, yielding
\begin{align}
    \psi(\bm{n_Q},\bm{0}) &= \Big(\sum_{j=1}^d n_{Q_j}\mu_j\Big)^{-1}\sum_{j=1}^d n_{Q_j} \psi(\bm{n_Q} - \bm{e}_j, \bm{e_j}).
\end{align}

\section{Explicit Examples for Bivariate Setting} \label{appendix: explicit bivariate}

In this section, we provide more explicit details for the moments in the bivariate setting, i.e, the case $d=2$.
We provide examples by writing out the recursive procedure outlined in Section~\ref{sec: recursive}.
The main objective is to derive near-explicit results for both the transient moments $\psi_t((n_{Q_1},n_{Q_2}),(n_{\lambda_1},n_{\lambda_2}))$ and stationary moments $\psi((n_{Q_1},n_{Q_2}),(n_{\lambda_1},n_{\lambda_2}))$, where the focus is on moments of order $1$ and~$2$.
In both cases, we apply the recursive procedures described in Section~\ref{sec: recursive}.

\subsection{Recursive procedure} \label{appendix: bivariate recursive procedure}

We illustrate the stacked vector $\bm{\Psi}_t^{(n)}$ for orders $n=1$ and $n=2$, and derive the ODEs associated with the recursive procedure.

\begin{example_text}[first order, bivariate]\label{example: order 1 moments} \em
For $n=1$, we have $\mathfrak{D}(2,1)=4$, and
\begin{align*}
    \bm{\Psi}^{(1)}_t 
    &= \big(\bm{\Psi}^{(0,1)}_t, \bm{\Psi}^{(1,0)}_t \big)^\top,
\end{align*}
where
\begin{align*}
    \bm{\Psi}^{(0,1)}_t &= \big(  \ee\big[ \lambda_1(t)\big], \ee\big[ \lambda_2(t)\big]\big)^\top, \quad
    \bm{\Psi}^{(1,0)}_t = \big(\ee\big[ Q_1(t)\big], \ee\big[ Q_2(t)\big]\big)^\top.
\end{align*}
By Step 0 of Algorithm \ref{alg: recursive procedure ODE blocks}, we obtain the ODE
\begin{align} \label{eq: ODE lambda transient order 1, 2-dim}
    \frac{\ddiff}{\ddiff t} \bm{\Psi}^{(0,1)}_t =
    \begin{bmatrix}
    -\overline{\alpha}_1 & \ee[B_{12}] \\
    \ee[B_{21}] & -\overline{\alpha}_2
    \end{bmatrix}
    \bm{\Psi}^{(0,1)}_t 
    +
    \begin{bmatrix}
    \alpha_1\overline{\lambda}_1 \\
    \alpha_2\overline{\lambda}_2
    \end{bmatrix},
\end{align}
whose solution gives us an expression for $\bm{\Psi}^{(0,1)}_t = (\ee[\lambda_1(t)], \ee[\lambda_2(t)])^\top$.
We need this $\bm{\Psi}^{(0,1)}_t$ in Step~1, which states
\begin{align} \label{eq: ODE Q transient order 1, 2-dim}
    \frac{\ddiff}{\ddiff t} \bm{\Psi}^{(1,0)}_t =
    \begin{bmatrix}
    -\mu_1 & 0 \\
    0 & -\mu_2
    \end{bmatrix}
    \bm{\Psi}^{(1,0)}_t
    +
    \begin{bmatrix}
    1 & 0 \\
    0 & 1
    \end{bmatrix}
    \bm{\Psi}^{(0,1)}_t,
\end{align}
whose solution yields an expression for $\bm{\Psi}^{(1,0)}_t = (\ee[Q_1(t)], \ee[Q_2(t)])^\top$.
\end{example_text}

\begin{example_text}[second order, bivariate]\label{example: order 2 moments} \em
For $n=2$, we have $\mathfrak{D}(2,2)=10$, and
\begin{align*}
    \bm{\Psi}^{(2)}_t = \big( \bm{\Psi}^{(0,2)}_t, \bm{\Psi}^{(1,1)}_t, \bm{\Psi}^{(2,0)}_t \big)^\top,
\end{align*}
where
\begin{align*}
    \bm{\Psi}^{(0,2)}_t &= \big( \ee\big[ \lambda_1(t)^2\big], \ee\big[ \lambda_1(t)\lambda_2(t)\big], \ee\big[ \lambda_2(t)^2\big]\big)^\top, \\
    \bm{\Psi}^{(1,1)}_t &= \big( \ee\big[ Q_1(t)\lambda_1(t)\big], \ee\big[Q_1(t)\lambda_2(t)\big], \ee\big[ Q_2(t)\lambda_1(t)\big], \ee\big[Q_2(t)\lambda_2(t)\big]\big)^\top, \\
    \bm{\Psi}^{(2,0)}_t &= \big( \ee\big[ Q_1(t)^{[2]}\big], \ee\big[ Q_1(t)Q_2(t)\big], \ee\big[ Q_2(t)^{[2]}\big]\big)^\top.
\end{align*}
For order $n=2$, our objective is to compute $\bm{\Psi}^{(0,2)}_t, \bm{\Psi}^{(1,1)}_t$, and $\bm{\Psi}^{(2,0)}_t$.
Step 0 of Algorithm \ref{alg: recursive procedure ODE blocks} yields
\begin{align} \label{eq: ODE lambda transient order 2, 2-dim}
\begin{split}
    \frac{\ddiff}{\ddiff t} \bm{\Psi}^{(0,2)}_t &=
    \begin{bmatrix}
    -2\overline{\alpha}_1 & 2\ee[B_{12}] & 0 \\
    \ee[B_{21}] & -\overline{\alpha}_1 - \overline{\alpha}_2 & \ee[B_{12}] \\
    0 & 2\ee[B_{21}] & -2\overline{\alpha}_2
    \end{bmatrix}
    \bm{\Psi}^{(0,2)}_t \\
    &\quad +
    \begin{bmatrix}
    2\alpha_1\overline{\lambda}_1 {+}\ee[B_{11}^2] & \ee[B_{12}^2] \\
    \ee[B_{11}]\ee[B_{21}]+\alpha_2\overline{\lambda}_2 & \ee[B_{22}]\ee[B_{12}]+\alpha_1\overline{\lambda}_1 \\
    \ee[B_{21}^2] & 2\alpha_2\overline{\lambda}_2 {+} \ee[B_{22}^2]
    \end{bmatrix}
    \bm{\Psi}^{(0,1)}_t,
\end{split}
\end{align}
which depends on the lower-order vector $\bm{\Psi}^{(0,1)}_t$ (which was found in Example 1).
For Step~1, note that $\bm{\Psi}^{(1,1)}_t$ is a $4$-dimensional vector, which satisfies
\begin{align} \label{eq: ODE Q + lambda transient order 2, 2-dim}
    \frac{\ddiff}{\ddiff t} \bm{\Psi}^{(1,1)}_t &=
    \begin{bmatrix}
    -\overline{\alpha}_1 -\mu_1 & \ee[B_{12}] \\
    \ee[B_{21}] & -\overline{\alpha}_2 -\mu_1
    \end{bmatrix}
    \oplus
    \begin{bmatrix}
    -\overline{\alpha}_1 -\mu_2 & \ee[B_{12}] \\
    \ee[B_{21}] & -\overline{\alpha}_2 -\mu_2
    \end{bmatrix}
    \bm{\Psi}^{(1,1)}_t \\
    &\quad +
    \begin{bmatrix}
    1 & 0 & 0 \\
    0 & 1 & 0 \\
    0 & 1 & 0 \\
    0 & 0 & 1
    \end{bmatrix}
    \bm{\Psi}^{(0,2)}_t
    +
    \begin{bmatrix}
    \alpha_{1}\overline{\lambda}_1 & 0 & \ee[B_{11}] & 0 \\
    \alpha_{2}\overline{\lambda}_2 & 0 & \ee[B_{21}] & 0 \\
    0 & \alpha_{1}\overline{\lambda}_1 & 0 & \ee[B_{12}] \\
    0 & \alpha_{2}\overline{\lambda}_2 & 0 & \ee[B_{22}]
    \end{bmatrix}
    \bm{\Psi}^{(1)}_t, \notag
\end{align}
where we see the dependence on the lower-order stacked vector $\bm{\Psi}^{(1)}_t$  (which was found in Example 1).
Regarding the final step, i.e., Step 2,
\begin{align}\label{eq: ODE Q transient order 2, 2-dim}
    \frac{\ddiff}{\ddiff t} \bm{\Psi}^{(2,0)}_t &=
    \begin{bmatrix}
    -2\mu_1 & 0 & 0 \\
    0 & -\mu_1 - \mu_2 & 0 \\
    0 & 0 & -2\mu_2
    \end{bmatrix}
    \bm{\Psi}^{(2,0)}_t
    +
    \begin{bmatrix}
    2 & 0 & 0 & 0 \\
    0 & 1 & 1 & 0 \\
    0 & 0 & 0 & 2
    \end{bmatrix}
    \bm{\Psi}^{(1,1)}_t.
\end{align}

\end{example_text}

\begin{example_text}[third order, bivariate]\label{example: order 3 moments} \em
For $n=3$, we have $\mathfrak{D}(2,3)=20$, and
\begin{align*}
    \bm{\Psi}^{(3)}_t = \big(\bm{\Psi}^{(0,3)}_t, \bm{\Psi}^{(1,2)}_t, \bm{\Psi}^{(2,1)}_t, \bm{\Psi}^{(3,0)}_t\big)^\top,
\end{align*}
where
\begin{align*}
    \bm{\Psi}^{(0,3)}_t &= \big( \ee\big[ \lambda_1(t)^3\big], \ee\big[ \lambda_1(t)^2\lambda_2(t)\big],\ee\big[ \lambda_1(t)\lambda_2(t)^2\big], \ee\big[ \lambda_2(t)^3\big]\big)^\top, \\
    \bm{\Psi}^{(1,2)}_t &= \big( \ee\big[ Q_1(t)\lambda_1(t)^2\big], \ee\big[Q_1(t)\lambda_1\lambda_2(t)\big], \ee\big[ Q_1(t)\lambda_2(t)^2\big], \\
    &\qquad \ee\big[ Q_2(t)\lambda_1(t)^2\big], \ee\big[ Q_2(t)\lambda_1(t)\lambda_2\big], \ee\big[Q_2(t)\lambda_2(t)^2\big]\big)^\top, \\
    \bm{\Psi}^{(2,1)}_t &= \big( \ee\big[ Q_1(t)^{[2]}\lambda_1(t)\big], \ee\big[Q_1(t)^{[2]}\lambda_2(t)\big], \ee\big[ Q_1(t)Q_2(t)\lambda_1(t)\big], 
    \ee\big[ Q_1(t)Q_2(t)\lambda_2(t)\big], \\
    &\qquad \ee\big[ Q_2(t)^{[2]}\lambda_1(t)\big], \ee\big[Q_2(t)^{[2]}\lambda_2(t)\big]\big)^\top, \\
    \bm{\Psi}^{(3,0)}_t &= \big( \ee\big[ Q_1(t)^{[3]}\big], \ee\big[ Q_1(t)^{[2]}Q_2(t)\big],\ee\big[ Q_1(t)Q_2(t)^{[2]}\big], \ee\big[ Q_2(t)^{[3]}\big]\big)^\top.
\end{align*}

We could explicitly write down the ODEs of these vectors using Algorithm~\ref{alg: recursive procedure ODE blocks}, but the exposition would be rather tedious with large matrices. 
\end{example_text}

\subsection{Transient moments} \label{appendix: transient moments 2-dim}
The goal of this subsection is to find near-explicit expressions for $\bm{\Psi}_t^{(1)}$ and $\bm{\Psi}_t^{(2)}$ by further solving the associated ODEs.
It is clear that we can obtain the solution in terms of a matrix exponential, which can be made more explicit in terms of its eigenvalues, namely
\begin{align} \label{eq: matrix exponential lagrange formula}
    e^{t\bm{M}^{(k,n-k)}} = \sum_{\ell=1}^{\overline{k}} e^{t \eta_\ell^{(k)}} 
    \prod_{\substack{m=1 \\ m \neq \ell}}^{\overline{k}} \frac{\bm{M}^{(k,n-k)} - \eta_m^{(k)} \bm{I}}{\eta_\ell^{(k)} - \eta_m^{(k)}},
\end{align}
with $\overline{k}$ denoting the dimension of $\bm{\Psi}_t^{(k,n-k)}$ and $\eta_1^{(k)},\dots,\eta_{\overline{k}}^{(k)}$ the eigenvalues of $\bm{M}^{(k,n-k)}$, and $\bm{I}$ the identity matrix.

We first consider the transient moments of order $n=1$, evaluating the entries of the stacked vector $\bm{\Psi}_t^{(1)}$, in particular solving the ODE of $\bm{\Psi}_t^{(0,1)}$ as given in Eqn.~\eqref{eq: ODE lambda transient order 1, 2-dim}.
By Proposition \ref{prop: solution vector ODE with matrix exponential}, the solution requires us to find the eigenvalues of the matrix
\begin{align*}
    \bm{M}^{(0,1)} = \begin{bmatrix}
    -\overline{\alpha}_1 & \ee[B_{12}] \\
    \ee[B_{21}] & -\overline{\alpha}_2
    \end{bmatrix},
\end{align*}
so as to compute the matrix exponential $e^{t\bm{M}^{(0,1)}}$.
With $\eta \equiv \eta_1,\eta_2$ denoting the two eigenvalues, we straightforwardly obtain
\begin{align} \label{eq: eigenvalue transient order 1, 2-dim}
    \eta = \frac{1}{2}\big( -\overline{\alpha}_1 -\overline{\alpha}_2 \pm \sqrt{\overline{\alpha}_1^2 - 2\overline{\alpha}_1\overline{\alpha}_2 + \overline{\alpha}_2^2 +4 \ee[B_{12}]\ee[B_{21}]}\big) \equiv \frac{1}{2}\big( -\overline{\alpha}_1 -\overline{\alpha}_2 \pm \sqrt{D_1}\big),
\end{align}
where $\overline{\alpha}_i = \alpha_i - \ee[B_{ii}]$ for $i=1,2$.
We let $\eta_1$ and $\eta_2$ denote the plus- and minus-variant of $\eta$ respectively.
Note that $D_1\geqslant 0$ since it involves a square and $B_{ij}$ are non-negative random variables.
Using Eqn.~\eqref{eq: matrix exponential lagrange formula} and performing some elementary computations, 
\begin{align*}
    e^{t\bm{M}^{(0,1)}} 
    &= \frac{1}{\eta_1 - \eta_2}\Big( e^{t\eta_1}\big( \bm{M}^{(0,1)} - \eta_2 \bm{I}\big) - e^{t\eta_2} \big( \bm{M}^{(0,1)} - \eta_1\bm{I}\big)\Big) \\
    &= \frac{1}{\sqrt{D_1}} \Big( 
    \big\{e^{t\eta_1} - e^{t\eta_2}\big\}
    \begin{bmatrix}
    \frac{1}{2}(\overline{\alpha}_2 - \overline{\alpha}_1) & \ee[B_{12}] \\
    \ee[B_{21}] & \frac{1}{2}(\overline{\alpha}_1 - \overline{\alpha}_2)
    \end{bmatrix}
    + \big\{e^{t\eta_1} + e^{t\eta_2}\big\}
    \begin{bmatrix}
    \frac{1}{2}\sqrt{D_1} & 0 \\
    0 & \frac{1}{2} \sqrt{D_1}
    \end{bmatrix}\Big) \\
    &= \frac{1}{2}\big\{e^{t\eta_1} + e^{t\eta_2}\big\}\bm{I}
    + \frac{1}{\sqrt{D_1}}
    \big\{e^{t\eta_1} - e^{t\eta_2}\big\}
    \begin{bmatrix}
    \frac{1}{2}(\overline{\alpha}_2 - \overline{\alpha}_1) & \ee[B_{12}] \\
    \ee[B_{21}] & \frac{1}{2}(\overline{\alpha}_1 - \overline{\alpha}_2)
    \end{bmatrix};
\end{align*}
we use curly brackets to  distinguish scalar terms from the vectors and matrices.

Next, we consider $\bm{\Psi}^{(1,0)}_t$ for which we need to find the eigenvalues of the matrix $\bm{M}^{(1,0)} = \text{diag}(-\mu_1,-\mu_2)$; cf.\ the ODE in Eqn.~\eqref{eq: ODE Q transient order 1, 2-dim}.
Since this is a diagonal matrix, these are simply $-\mu_1$ and $-\mu_2$, such that the matrix exponential is just
\begin{align*}
    e^{t\bm{M}^{(1,0)}} = \begin{bmatrix}
    e^{-t\mu_1} & 0 \\
    0 & e^{-t\mu_2}
    \end{bmatrix}.
\end{align*}
Before we get to the solution, we define a number of functions needed for the solution of $\bm{\Psi}^{(1,0)}_t$, namely
\begin{align*}
    \bm{u}_1(t)
    &:= \begin{bmatrix}
    (\mu_1+\eta_1)^{-1}(e^{t\eta_1}-e^{-t\mu_1}) + (\mu_1+\eta_2)^{-1}( e^{t\eta_2}-e^{-t\mu_1}) \\
    (\mu_2+\eta_1)^{-1}(e^{t\eta_1}-e^{-t\mu_2}) + (\mu_2+\eta_2)^{-1}( e^{t\eta_2}-e^{-t\mu_2})
    \end{bmatrix}, \\
    \bm{u}_2(t)
    &:= \begin{bmatrix}
    (\mu_1+\eta_1)^{-1}(e^{t\eta_1}-e^{-t\mu_1}) - (\mu_1+\eta_2)^{-1}( e^{t\eta_2}-e^{-t\mu_1}) \\
    (\mu_2+\eta_1)^{-1}(e^{t\eta_1}-e^{-t\mu_2}) - (\mu_2+\eta_2)^{-1}( e^{t\eta_2}-e^{-t\mu_2})
    \end{bmatrix}, \\
    \bm{u}_3(t)
    &:= \begin{bmatrix}
    (\eta_1(\mu_1+\eta_1))^{-1}(e^{t\eta_1}-e^{-t\mu_1}) + (\eta_2(\mu_1+\eta_2))^{-1}( e^{t\eta_2}-e^{-t\mu_1}) \\
    (\eta_1(\mu_2+\eta_1))^{-1}(e^{t\eta_1}-e^{-t\mu_2}) + (\eta_2(\mu_2+\eta_2))^{-1}( e^{t\eta_2}-e^{-t\mu_2})
    \end{bmatrix}, \\
    \bm{u}_4(t)
    &:= \begin{bmatrix}
    (\eta_1(\mu_1+\eta_1))^{-1}(e^{t\eta_1}-e^{-t\mu_1}) - (\eta_2(\mu_1+\eta_2))^{-1}( e^{t\eta_2}-e^{-t\mu_1}) \\
    (\eta_1(\mu_2+\eta_1))^{-1}(e^{t\eta_1}-e^{-t\mu_2}) - (\eta_2(\mu_2+\eta_2))^{-1}( e^{t\eta_2}-e^{-t\mu_2})
    \end{bmatrix}.
\end{align*}

We can now give the first moments explicitly through the application of Proposition \ref{prop: solution vector ODE with matrix exponential}.
Given the initial conditions $\bm{\Psi}_0^{(0,1)} =( \overline{\lambda}_1 , \overline{\lambda}_2)^\top$ and $\bm{\Psi}_0^{(1,0)} = ( 0 , 0 )^\top$, the solutions to the ODEs in \eqref{eq: ODE lambda transient order 1, 2-dim} and \eqref{eq: ODE Q transient order 1, 2-dim} are given by
\begin{align}\label{eq: solution transient moment lambda order 1, 2-dim}
\begin{split}
    \bm{\Psi}_t^{(0,1)} 
    &= e^{t\bm{M}^{(0,1)}}
    \begin{bmatrix}
    \,\overline{\lambda}_1 \\
    \,\overline{\lambda}_2
    \end{bmatrix}
    + \int_0^t e^{(t-s)\bm{M}^{(0,1)}}
    \begin{bmatrix}
    \alpha_1\overline{\lambda}_1 \\
    \alpha_2\overline{\lambda}_2
    \end{bmatrix} \ddiff s \\
    &= \frac{1}{\overline{\alpha}_1\overline{\alpha}_2 - \ee[B_{12}]\ee[B_{21}]}
    \begin{bmatrix}
    \alpha_1\overline{\lambda}_1 \overline{\alpha}_2 + \alpha_2\overline{\lambda}_2\ee[B_{12}] \\
    \alpha_2\overline{\lambda}_2 \overline{\alpha}_1 + \alpha_1\overline{\lambda}_1\ee[B_{21}] \\
    \end{bmatrix} \\
    &\quad + \frac{1}{2}\big\{e^{t\eta_1}+ e^{t\eta_2}\big\}\begin{bmatrix}
    \overline{\lambda}_1 \\
    \overline{\lambda}_2
    \end{bmatrix} 
    + \frac{1}{\sqrt{D_1}}\big\{e^{t\eta_1} - e^{t\eta_2} \big\} \begin{bmatrix}
    \frac{1}{2}\overline{\lambda}_1(\overline{\alpha}_2 - \overline{\alpha}_1) + \overline{\lambda}_2 \ee[B_{12}] \\
    \frac{1}{2}\overline{\lambda}_2(\overline{\alpha}_1 - \overline{\alpha}_2) + \overline{\lambda}_1 \ee[B_{21}]
    \end{bmatrix} \\
    &\quad + \frac{1}{2} \big\{\eta_1^{-1}e^{t\eta_1}+ \eta_2^{-1}e^{t\eta_2}\big\}
    \begin{bmatrix}
    \alpha_1\overline{\lambda}_1 \\
    \alpha_2\overline{\lambda}_2
    \end{bmatrix} \\
    &\quad + \frac{1}{\sqrt{D_1}}\big\{\eta_1^{-1}e^{t\eta_1} - \eta_2^{-1}e^{t\eta_2} \big\}
    \begin{bmatrix}\frac{1}{2}\alpha_1\overline{\lambda}_1(\overline{\alpha}_2 - \overline{\alpha}_1) + \alpha_2\overline{\lambda}_2 \ee[B_{12}] \\
    \frac{1}{2}\alpha_2\overline{\lambda}_2(\overline{\alpha}_1 - \overline{\alpha}_2) + \alpha_1\overline{\lambda}_1 \ee[B_{21}]
    \end{bmatrix},
\end{split}
\end{align}
and
\begin{align} \label{eq: solution transient moment Q order 1, 2-dim}
\begin{split}
    \bm{\Psi}^{(1,0)}_t 
    &=\int_0^t e^{(t-s)\bm{M}^{(1,0)}} \bm{\Psi}_s^{(0,1)} \ddiff s \\
    &= \frac{1}{\overline{\alpha}_1\overline{\alpha}_2 - \ee[B_{12}]\ee[B_{21}]}
    \begin{bmatrix}
    \mu_1^{-1}(1-e^{-t\mu_1}) \\ 
    \mu_1^{-2}(1-e^{-t\mu_2})
    \end{bmatrix}
    \odot
    \begin{bmatrix}
    \alpha_1\overline{\lambda}_1 \overline{\alpha}_2 + \alpha_2\overline{\lambda}_2\ee[B_{12}] \\
    \alpha_2\overline{\lambda}_2 \overline{\alpha}_1 + \alpha_1\overline{\lambda}_1\ee[B_{21}] \\
    \end{bmatrix} \\
    &\quad + \frac{1}{2}
    \bm{u}_1(t)
    \odot
    \begin{bmatrix}
    \,\overline{\lambda}_1 \\
    \,\overline{\lambda}_2
    \end{bmatrix}
    + \frac{1}{\sqrt{D_1}} 
    \bm{u}_2(t)
    \odot
    \begin{bmatrix}
    \frac{1}{2}\overline{\lambda}_1(\overline{\alpha}_2 - \overline{\alpha}_1) + \overline{\lambda}_2 \ee[B_{12}] \\
    \frac{1}{2}\overline{\lambda}_2(\overline{\alpha}_1 - \overline{\alpha}_2) + \overline{\lambda}_1 \ee[B_{21}]
    \end{bmatrix} \\
    &\quad +\frac{1}{2} \bm{u}_3(t) 
    \odot 
    \begin{bmatrix}
    \alpha_1\overline{\lambda}_1 \\
    \alpha_2\overline{\lambda}_2
    \end{bmatrix}
    + \frac{1}{\sqrt{D_1}} \bm{u}_4(t)
    \odot
    \begin{bmatrix}\frac{1}{2}\alpha_1\overline{\lambda}_1(\overline{\alpha}_2 - \overline{\alpha}_1) + \alpha_2\overline{\lambda}_2 \ee[B_{12}] \\
    \frac{1}{2}\alpha_2\overline{\lambda}_2(\overline{\alpha}_1 - \overline{\alpha}_2) + \alpha_1\overline{\lambda}_1 \ee[B_{21}]
    \end{bmatrix}.
\end{split}
\end{align}

Observe that in order for the solution in \eqref{eq: solution transient moment lambda order 1, 2-dim} to remain stable and to obtain finite moments, we need that both eigenvalues are strictly smaller than 0.
By some elementary algebra, it is seen that we should have that
\begin{align}
    \overline{\alpha}_1 \overline{\alpha}_2 > \ee[B_{12}]\ee[B_{21}].
\end{align}
Note that this is the explicit version of the stability condition $\rho(\bm{H})<1$ for the bivariate setting; see Assumption~\ref{ass: stability condition}.
Also note that if one is interested in the Hawkes process $\bm{N}(t) = (N_1(t),N_2(t))^\top$ rather than the population process $\bm{Q}(t) = (Q_1(t),Q_2(t))^\top$, one needs to take $\mu_1 = \mu_2 \equiv 0$.
The corresponding moments $\ee[\bm{N}(t)] = (\ee[N_1(t)],\ee[N_2(t)])^\top$ can be derived from Eqn.~\eqref{eq: solution transient moment Q order 1, 2-dim} by taking the limit $(\mu_1,\mu_2) \downarrow (0,0)$ and the use of L'Hopital's rule.

We now turn to order $2$ and compute elements of the stacked vector $\bm{\Psi}_t^{(2)}$.
As before, we start by considering $\bm{\Psi}_t^{(0,2)}$, the vector containing the (mixed) moments corresponding to $\bm{\lambda}(t)$.
By Proposition \ref{prop: solution vector ODE with matrix exponential}, and recalling that $\bm{\Psi}_t^{(0,2)}$ satisfies the ODE in Eqn.~\eqref{eq: ODE lambda transient order 2, 2-dim}, we need to find the eigenvalues of 
\begin{align*}
    \bm{M}^{(0,2)} =
    \begin{bmatrix}
    -2\overline{\alpha}_1 & 2\ee[B_{12}] & 0 \\
    \ee[B_{21}] & -\overline{\alpha}_1 - \overline{\alpha}_2 & \ee[B_{12}] \\
    0 & 2\ee[B_{21}] & -2\overline{\alpha}_2
    \end{bmatrix}.
\end{align*}
Let $\kappa \equiv \kappa_1,\kappa_2,\kappa_3$ denote the eigenvalues of $\bm{M}^{(0,2)}$.
We compute
\begin{align*}
    &\begin{vmatrix}
    -2\overline{\alpha}_1 - \kappa & 2\ee[B_{12}] & 0 \\
    \ee[B_{21}] & -\overline{\alpha}_1 - \overline{\alpha}_2 - \kappa & \ee[B_{12}] \\
    0 & 2\ee[B_{21}] & -2\overline{\alpha}_2 -\kappa
    \end{vmatrix} 
    = 0 \\
    \iff &(-2\overline{\alpha}_1 - \kappa)
    \begin{vmatrix}
    -\overline{\alpha}_1 - \overline{\alpha}_2 - \kappa & \ee[B_{12}] \\
    2\ee[B_{21}] & -2\overline{\alpha}_2 -\kappa
    \end{vmatrix}
    - 2\ee[B_{12}]
    \begin{vmatrix}
    \ee[B_{21}] & \ee[B_{12}] \\
    0 & -2\overline{\alpha}_2 -\kappa
    \end{vmatrix}
    =0 \\
    \iff & (-2\overline{\alpha}_1 - \kappa)\big( \kappa^2 + 3\overline{\alpha}_2\kappa + 2 \overline{\alpha}_2^2 + 2 \overline{\alpha}_1\overline{\alpha}_2 + \overline{\alpha}_1\kappa - 2\ee[B_{12}]\ee[B_{21}]\big) \\
    & \quad + 2\ee[B_{12}](\ee[B_{21}]\kappa + 2\overline{\alpha}_2\ee[B_{21}] = 0\\
    \iff     
    &\kappa^3 + \kappa^23(\overline{\alpha}_1 + \overline{\alpha}_2 ) + \kappa\big(2(\overline{\alpha}_1     + \overline{\alpha}_2)^2 - 4\ee[B_{12}]\ee[B_{21}]\big) -4(\overline{\alpha}_1 +     \overline{\alpha}_2)\ee[B_{12}]\ee[B_{21}] = 0 \\
    \iff &\kappa^3 + \kappa^2 b + \kappa c + d = 0,
\end{align*}
with $b,c,d$ defined as the constants of the square, linear and constant term respectively.
To apply the formula for the solutions to this cubic equation, we compute $p$ and $q$, given by
\begin{align*}
    p = \frac{1}{3}(3c - b^2) &= 2(\overline{\alpha}_1     + \overline{\alpha}_2)^2 - 4\ee[B_{12}]\ee[B_{21}] - 3((\overline{\alpha}_1 + \overline{\alpha}_2)^2 \\
    &= -(\overline{\alpha}_1 + \overline{\alpha}_2)^2 - 4\ee[B_{12}]\ee[B_{21}], 
\end{align*}
and
\begin{align*}
    q &= \frac{1}{27}\big\{2b^3 - 9bc +27d\big\} \\
    &= \frac{1}{27} \Big\{ 2\big(3(\overline{\alpha}_1 + \overline{\alpha}_2)\big)^3 - 27 (\overline{\alpha}_1 + \overline{\alpha}_2)\big(2(\overline{\alpha}_1 + \overline{\alpha}_2)^2 - 4\ee[B_{12}]\ee[B_{21}]\big) \\
    &\quad \quad \quad -4 \cdot 27(\overline{\alpha}_1 + \overline{\alpha}_2)\ee[B_{12}]\ee[B_{21}]\Big\}\\
    &= 2(\overline{\alpha}_1 + \overline{\alpha}_2)^3 - 2 (\overline{\alpha}_1 + \overline{\alpha}_2)^3 + 4(\overline{\alpha}_1 + \overline{\alpha}_2)\ee[B_{12}]\ee[B_{21}]
    -4(\overline{\alpha}_1  +\overline{\alpha}_2)\ee[B_{12}]\ee[B_{21}]\\
    &= 0.
\end{align*}
It is well-known that the cubic equation has three real roots if $4p^3 + 27q^2 < 0$.
Since $q=0$, the condition becomes $4p^3<0$, which holds since $p< 0$ because of the square term and $\ee[B_{ij}]\geqslant 0$.
Hence, the eigenvalues $\kappa_m$, with $m=1,2,3$, are given by the trigonometric solution 
\begin{align*}
    \kappa_k &= - \frac{b}{3} + 2\sqrt{\frac{-p}{3}}\cos(\theta_m) 
    = -(\overline{\alpha}_1 + \overline{\alpha}_2) +  2 \sqrt{(\overline{\alpha}_1+\overline{\alpha}_2)^2 + 4\ee[B_{12}]\ee[B_{21}]} \cos(\theta_m), \\
    \theta_m &= \frac{1}{3} \arccos\big( \frac{3q}{2p}\sqrt{\frac{3}{-p}}\big)- \frac{2\pi}{3}(m-1)
    =\frac{\pi}{6} - \frac{2\pi}{3}(m-1).
\end{align*}
This yields the eigenvalues
\begin{align*}
    \kappa_1 &= -(\overline{\alpha}_1 + \overline{\alpha}_2) \\
    \kappa_2 &= -(\overline{\alpha}_1 + \overline{\alpha}_2) + \sqrt{3}\sqrt{(\overline{\alpha}_1+\overline{\alpha}_2)^2 + 4\ee[B_{12}]\ee[B_{21}]} = -(\overline{\alpha}_1 + \overline{\alpha}_2) + \sqrt{D_2}\\
    \kappa_3 &= -(\overline{\alpha}_1 + \overline{\alpha}_2) - \sqrt{3}\sqrt{(\overline{\alpha}_1+\overline{\alpha}_2)^2 + 4\ee[B_{12}]\ee[B_{21}]} = -(\overline{\alpha}_1 + \overline{\alpha}_2) - \sqrt{D_2},
\end{align*}
with $D_2 = 3((\overline{\alpha}_1+\overline{\alpha}_2)^2 + 4\ee[B_{12}]\ee[B_{21}])$.
We apply these eigenvalues in the computation of the matrix exponential, as described in Eqn.~\eqref{eq: matrix exponential lagrange formula}, to obtain
\begin{align*}
    &e^{t\bm{M}^{(0,2)}} 
    = e^{\kappa_1 t} \frac{1}{\kappa_1 - \kappa_2}\frac{1}{\kappa_1 - \kappa_3}(\bm{M}^{(0,2)} - \kappa_2\bm{I})(\bm{M}^{(0,2)} - \kappa_3\bm{I}) \\
    &\quad \quad + e^{\kappa_2 t} \frac{1}{\kappa_2 - \kappa_1}\frac{1}{\kappa_2 - \kappa_3}(\bm{M}^{(0,2)} - \kappa_1\bm{I})(\bm{M}^{(0,2)} - \kappa_3\bm{I}) \\ 
    &\quad \quad + e^{\kappa_3 t} \frac{1}{\kappa_3 - \kappa_1}\frac{1}{\kappa_3 - \kappa_2}(\bm{M}^{(0,2)} - \kappa_1\bm{I})(\bm{M}^{(0,2)} - \kappa_2\bm{I}) \\ 
    &= \frac{e^{\kappa_1 t}}{D_2}
    \begin{bmatrix}
    c_{\kappa_1} & -2\ee[B_{12}](\overline{\alpha}_2 - \overline{\alpha}_1) & -2\ee[B_{12}]^2 \\
    -\ee[B_{21}](\overline{\alpha}_2 - \overline{\alpha}_1) & 3(\overline{\alpha}_1 + \overline{\alpha}_2)^2 +8\ee[B_{12}]\ee[B_{21}] & -\ee[B_{12}](\overline{\alpha}_1 - \overline{\alpha}_2) \\
    -2\ee[B_{21}]^2 & -2\ee[B_{21}](\overline{\alpha}_1 - \overline{\alpha}_2) & c_{\kappa_1}
    \end{bmatrix} \\
    &\quad + \frac{e^{\kappa_2 t}}{2D_2}
    \begin{bmatrix}
    c_- & -2\ee[B_{12}](\overline{\alpha}_1 - \overline{\alpha}_2 - \sqrt{D_2}) & 2\ee[B_{12}]^2 \\
    -\ee[B_{21}](\overline{\alpha}_1 - \overline{\alpha}_2 - \sqrt{D_2}) & 4\ee[B_{12}]\ee[B_{21}] & \ee[B_{12}](\overline{\alpha}_1 - \overline{\alpha}_2 + \sqrt{D_2}) \\
    2\ee[B_{21}]^2 & 2\ee[B_{21}](\overline{\alpha}_1 - \overline{\alpha}_2 + \sqrt{D_2}) & c_+
    \end{bmatrix} \\
    &\quad + \frac{e^{\kappa_3 t}}{2D_2}
    \begin{bmatrix}
    c_+ & -2\ee[B_{12}](\overline{\alpha}_1 - \overline{\alpha}_2 + \sqrt{D_2}) & 2\ee[B_{12}]^2 \\
    -\ee[B_{21}](\overline{\alpha}_1 - \overline{\alpha}_2 - \sqrt{D_2}) & 4\ee[B_{12}]\ee[B_{21}] & \ee[B_{12}](\overline{\alpha}_1 - \overline{\alpha}_2 - \sqrt{D_2}) \\
    2\ee[B_{21}]^2 & 2\ee[B_{21}](\overline{\alpha}_1 - \overline{\alpha}_2 - \sqrt{D_2}) & c_-
    \end{bmatrix},
\end{align*}
where 
\begin{align*}
    c_{\kappa_1} &= 2\overline{\alpha}_1^2 + 8\overline{\alpha}_1\overline{\alpha}_2 +2\overline{\alpha}_2^2 + 10\ee[B_{12}]\ee[B_{21}], \\
    c_- &= 2\ee[B_{12}]\ee[B_{21}] + (\overline{\alpha}_1 - \overline{\alpha}_2)(\overline{\alpha}_1 - \overline{\alpha}_2 - \sqrt{D_2}), \\
    c_+ &= 2\ee[B_{12}]\ee[B_{21}] + (\overline{\alpha}_1 - \overline{\alpha}_2)(\overline{\alpha}_1 - \overline{\alpha}_2 + \sqrt{D_2}).
\end{align*}

We now derive the matrix exponential corresponding to $\bm{\Psi}_t^{(1,1)}$, appearing in the ODE given in Eqn.~\eqref{eq: ODE Q + lambda transient order 2, 2-dim}.
Observe that \eqref{eq: ODE Q + lambda transient order 2, 2-dim} reveals that to compute $\bm{\Psi}_t^{(1,1)}$, we need to know the non-homogeneous part of the equation, i.e. $\bm{\Psi}_t^{(0,2)}$, as well as the lower order $(n=1)$ stacked vector $\bm{\Psi}_t^{(1)}$.
Further notice that the $4\times 4$ matrix $\bm{M}^{(1,1)}$ is the direct sum of two $2\times 2$ matrices, which implies that we can split the $4$-dimensional ODE into two $2$-dimensional ODEs.
We introduce the relevant objects by setting
\begin{align*}
    \bm{\Psi}_t^{(1,1)}=  \big(\bm{\Psi}_{t,Q_1}^{(1,1)}, \bm{\Psi}_{t,Q_2}^{(1,1)}\big)^\top,
\end{align*}
where
\begin{align*}
    \bm{\Psi}_{t,Q_1}^{(1,1)} 
    &= \big( \ee\big[ Q_1(t)\lambda_1(t)\big], \ee\big[Q_1(t)\lambda_2(t)\big]\big)^\top, \\
    \bm{\Psi}_{t,Q_2}^{(1,1)} 
    &=\big( \ee\big[ Q_2(t)\lambda_1(t)\big], \ee\big[Q_2(t)\lambda_2(t)\big]\big)^\top.
\end{align*}
Focus on the solution of $\bm{\Psi}_{t,Q_1}^{(1,1)}$, where we note that the solution for $\bm{\Psi}_{t,Q_2}^{(1,1)}$ can be obtained in an analogous manner.
Observe that one can derive from Eqn.~\eqref{eq: ODE Q + lambda transient order 2, 2-dim} that $\bm{\Psi}_{t,Q_1}^{(1,1)}$ satisfies the ODE
\begin{align*}
    \frac{\ddiff}{\ddiff t}\bm{\Psi}_{t,Q_1}^{(1,1)}
    &= 
    \begin{bmatrix}
    -\overline{\alpha}_1 -\mu_1 & \ee[B_{12}] \\
    \ee[B_{21}] & -\overline{\alpha}_2 -\mu_1
    \end{bmatrix}
    \bm{\Psi}_{t,Q_1}^{(1,1)}
    +
    \begin{bmatrix}
    1 & 0 & 0 \\
    0 & 1 & 0
    \end{bmatrix}
    \bm{\Psi}^{(0,2)}_t 
    +
    \begin{bmatrix}
    \alpha_{1}\overline{\lambda}_1 & 0 & \ee[B_{11}] & 0 \\
    \alpha_{2}\overline{\lambda}_2 & 0 & \ee[B_{21}] & 0 
    \end{bmatrix}
    \bm{\Psi}^{(1)}_t.
\end{align*}
This means that we need the matrix exponential of 
\begin{align*}
    \bm{M}_{Q_1}^{(1,1)} =\begin{bmatrix}
    -\overline{\alpha}_1 -\mu_1 & \ee[B_{12}] \\
    \ee[B_{21}] & -\overline{\alpha}_2 -\mu_1
    \end{bmatrix},
\end{align*}
which requires us to find the corresponding two eigenvalues, denoted by $\gamma^{(Q_1)}_1$ and $\gamma^{(Q_1)}_2$.
These eigenvalues are similarly derived as in Eqn.~\eqref{eq: eigenvalue transient order 1, 2-dim}, in this case given by
\begin{align*}
    \gamma_1^{(Q_1)} 
    = -\mu_1 + \eta_1, \quad 
    \gamma_2^{(Q_1)} 
    = -\mu_1 + \eta_2,
\end{align*}
with $\eta_1 = \frac{1}{2}\big( -\overline{\alpha}_1 - \overline{\alpha}_2  + \sqrt{D_1}\big)$ and $\eta_2 = \frac{1}{2}\big( -\overline{\alpha}_1 - \overline{\alpha}_2  - \sqrt{D_1}\big)$.
Substituting this in the Lagrange interpolation formula of Eqn.~\eqref{eq: matrix exponential lagrange formula}, we obtain
\begin{align*}
    e^{t\bm{M}_{Q_1}^{(1,1)}} 
    &= \frac{1}{\gamma^{(Q_1)}_1 - \gamma^{(Q_1)}_2}\Big\{e^{t\gamma^{(Q_1)}_1}\big(\bm{M}_{Q_1}^{(1,1)} - \gamma^{(Q_1)}_2\bm{I}_2\big) 
    -e^{t\gamma^{(Q_1)}_2}\big( \bm{M}_{Q_1}^{(1,1)} - \gamma^{(Q_1)}_1\bm{I}_2\big)\Big\} \\
    &= \frac{1}{2}\big\{e^{t\gamma^{(Q_1)}_1} + e^{t\gamma^{(Q_1)}_2}\big\} \bm{I} 
    + \frac{1}{\sqrt{D_1}}\big\{e^{t\gamma^{(Q_1)}_1} - e^{t\gamma^{(Q_1)}_2}\big\} 
    \begin{bmatrix}
    \frac{1}{2}(\overline{\alpha}_2 - \overline{\alpha}_1) & \ee[B_{12}] \\
    \ee[B_{21}] & \frac{1}{2}(\overline{\alpha}_1 -  \overline{\alpha}_2)
    \end{bmatrix}.
\end{align*}

In a very similar manner, the matrix exponential needed to evaluate $\bm{\Psi}_{t,Q_2}^{(1,1)}$ can be obtained from the ODE in Eqn.~\eqref{eq: ODE Q + lambda transient order 2, 2-dim}, and is given by
\begin{align*}
    e^{t\bm{M}_{Q_2}^{(1,1)}} 
    &= \frac{1}{\gamma^{(Q_2)}_1 - \gamma^{(Q_2)}_2}\Big\{e^{t\gamma^{(Q_2)}_1}\big(\bm{M}_{Q_2}^{(1,1)} - \gamma^{(Q_2)}_2\bm{I}_2\big) 
    -e^{t\gamma^{(Q_2)}_2}\big( \bm{M}_{Q_2}^{(1,1)} - \gamma^{(Q_2)}_1\bm{I}_2\big)\Big\} \\
    &= \frac{1}{2}\big\{e^{t\gamma^{(Q_2)}_1} + e^{t\gamma^{(Q_2)}_2}\big\} \bm{I} 
    + \frac{1}{\sqrt{D_1}}\big\{e^{t\gamma^{(Q_2)}_1} - e^{t\gamma^{(Q_2)}_2}\big\} 
    \begin{bmatrix}
    \frac{1}{2}(\overline{\alpha}_2 - \overline{\alpha}_1) & \ee[B_{12}] \\
    \ee[B_{21}] & \frac{1}{2}(\overline{\alpha}_1 -  \overline{\alpha}_2)
    \end{bmatrix},
\end{align*}
with $\gamma_1^{(Q_2)} = -\mu_2 + \eta_1$ and $\gamma_2^{(Q_2)} = -\mu_2 + \eta_2$.

Finally, for the solution of $\bm{\Psi}_t^{(2,0)}$, the vector containing the mixed factorial moments of $\bm{Q}(t)$, the matrix exponential of $\bm{M}^{(2,0)} = \text{diag}(-2\mu_1,-\mu_1-\mu_2,-2\mu_2)$ is simply
\begin{align*}
    e^{t\bm{M}^{(2,0)}} =
    \begin{bmatrix}
    e^{-2t\mu_1} & 0 & 0 \\
    0 & e^{-t(\mu_1 + \mu_2)} & 0 \\
    0 & 0 & e^{-2t\mu_2}
    \end{bmatrix}.
\end{align*}

Applying Proposition \ref{prop: solution vector ODE with matrix exponential} to $ \bm{\Psi}_t^{(0,2)}$, $ \bm{\Psi}_t^{(1,1)}$ and $ \bm{\Psi}_t^{(2,0)}$, we obtain the following result.
Given the initial conditions $\bm{\Psi}_0^{(0,2)} = ( \overline{\lambda}_1^2 , \overline{\lambda}_1\overline{\lambda}_2 , \overline{\lambda}_2^2 )^\top$, $\bm{\Psi}_{0,Q_1}^{(1,1)} = ( 0, 0 )^\top$, $\bm{\Psi}_{0,Q_2}^{(1,1)} = ( 0, 0 )^\top$ and $\bm{\Psi}_0^{(2,0)} = ( 0 , 0 , 0)^\top$, the solutions to the ODEs in Eqns.~\eqref{eq: ODE lambda transient order 2, 2-dim}, \eqref{eq: ODE Q + lambda transient order 2, 2-dim} and \eqref{eq: ODE Q transient order 2, 2-dim}, are given by, respectively,
\begin{align} \label{eq: solution transient moment lambda order 2, 2-dim}
\begin{split}
    \bm{\Psi}_t^{(0,2)}
    &= e^{t\bm{M}^{(0,2)}} \bm{\Psi}_0^{(0,2)} \\
    & \quad + \int_0^t e^{(t-s)\bm{M}^{(0,2)}}     \begin{bmatrix}
    2\alpha_1\overline{\lambda}_1 {+}\ee[B_{11}^2] & \ee[B_{12}^2] \\
    \ee[B_{11}]\ee[B_{21}] + \alpha_2\overline{\lambda}_2 & \ee[B_{22}]\ee[B_{12}] + \alpha_1\overline{\lambda}_1 \\
    \ee[B_{21}^2] & 2\alpha_2\overline{\lambda}_2 {+} \ee[B_{22}^2]
    \end{bmatrix}
    \bm{\Psi}^{(0,1)}_s \ddiff s,
\end{split}
\end{align}
and $\bm{\Psi}_t^{(1,1)} = ( \bm{\Psi}_{t,Q_1}^{(1,1)}, \bm{\Psi}_{t,Q_2}^{(1,1)})^\top$, with
\begin{align}
\begin{split}
    \bm{\Psi}_{t,Q_1}^{(1,1)} 
    &= \int_0^t  e^{(t-s)\bm{M}_{Q_1}^{(1,1)}}\Big\{\begin{bmatrix}
    1 & 0 & 0 \\
    0 & 1 & 0
    \end{bmatrix}
    \bm{\Psi}^{(0,2)}_s
    +
    \begin{bmatrix}
    \alpha_{1}\overline{\lambda}_1 & 0 & \ee[B_{11}] & 0 \\
    \alpha_{2}\overline{\lambda}_2 & 0 & \ee[B_{21}] & 0 
    \end{bmatrix}
    \bm{\Psi}^{(1)}_s\Big\}\ddiff s \\
    \bm{\Psi}_{t,Q_2}^{(1,1)} 
    &= \int_0^t  e^{(t-s)\bm{M}_{Q_2}^{(1,1)}}\Big\{\begin{bmatrix}
    0 & 1 & 0 \\
    0 & 0 & 1
    \end{bmatrix}
    \bm{\Psi}^{(0,2)}_s
    +
    \begin{bmatrix}
    0 & \alpha_{1}\overline{\lambda}_1 & 0 & \ee[B_{12}] \\
    0 & \alpha_{2}\overline{\lambda}_2 & 0 & \ee[B_{22}]
    \end{bmatrix}
    \bm{\Psi}^{(1)}_s\Big\}\ddiff s 
\end{split}
\end{align}
and 
\begin{align}
\begin{split}
    \bm{\Psi}_t^{(2,0)} 
    &= \int_0^t e^{(t-s)\bm{M}^{(2,0)}} \begin{bmatrix}
    2 & 0 & 0 & 0 \\
    0 & 1 & 1 & 0 \\
    0 & 0 & 0 & 2
    \end{bmatrix}
    \bm{\Psi}^{(1,1)}_s \ddiff s \\
    &= \int_0^t
    \begin{bmatrix}
    2 e^{-2(t-s)\mu_1}\ee[Q_1(s)\lambda_1(s)] \\
    e^{-(t-s)(\mu_1 + \mu_2)}\big(\ee[Q_1(s)\lambda_2(s)] + \ee[Q_2(s)\lambda_1(s)]\big) \\
    2  e^{-2(t-s)\mu_2}\ee[Q_2(s)\lambda_2(s)]
    \end{bmatrix} \ddiff s.
\end{split}
\end{align}

We remark that explicit evaluation of the expressions for the second-order moments is tedious, but in principle possible.
For instance, it is evident that substituting the matrix exponential of $\bm{M}^{(0,2)}$ and the lower-order solution $\bm{\Psi}_s^{(0,1)}$ will yield rather involved expressions.
For higher order moments, these expressions become even more complex.
However, the near-explicit solution in terms of matrix exponentials and lower order terms is useful for practical purposes.
Due to the availability of fast and robust algorithms for the matrix exponential and the solution of ODEs, it is relatively straightforward to numerically compute higher-order moments.

\subsection{Stationary moments} \label{appendix: stationary moments 2-dim}

This subsection deals with the application of Algorithm \ref{alg: recursive procedure stationary blocks}  to evaluate the stationary moments $\bm{\Psi}^{(1)}$ and $\bm{\Psi}^{(2)}$.

We start by computing the stationary moments of order $1$.
By Step 1 of Algorithm \ref{alg: recursive procedure stationary blocks}, $\bm{\Psi}^{(0,1)}$ satisfies the linear equation
\begin{align} 
    0 &= 
    \begin{bmatrix}
    -\overline{\alpha}_1 & \ee[B_{12}] \\
    \ee[B_{21}] & -\overline{\alpha}_2
    \end{bmatrix}
    \bm{\Psi}^{(0,1)} 
    +
    \begin{bmatrix}
    \alpha_1\overline{\lambda}_1 \\
    \alpha_2\overline{\lambda}_2
    \end{bmatrix}, \notag 
\end{align}
yielding
\begin{align} \label{eq: solution stationary moment lambda order 1, 2-dim}
    \bm{\Psi}^{(0,1)} &= \frac{1}{\overline{\alpha}_1\overline{\alpha}_2 - \ee[B_{12}]\ee[B_{21}]}
    \begin{bmatrix}
    \alpha_1\overline{\lambda}_1\overline{\alpha}_2 + \alpha_2\overline{\lambda}_2\ee[B_{12}] \\
    \alpha_2\overline{\lambda}_2\overline{\alpha}_1 + \alpha_1\overline{\lambda}_1\ee[B_{21}]
    \end{bmatrix}.
\end{align}
Then, regarding Step 2, we find after some calculus
\begin{align} \label{eq: solution stationary moment Q order 1, 2-dim}
    \bm{\Psi}^{(1,0)} &= \frac{1}{\overline{\alpha}_1\overline{\alpha}_2 - \ee[B_{12}]\ee[B_{21}]}
    \begin{bmatrix}
    \mu_1^{-1}\big(\alpha_1\overline{\lambda}_1\overline{\alpha}_2 + \alpha_2\overline{\lambda}_2\ee[B_{12}]\big) \\
    \mu_2^{-1}\big(\alpha_2\overline{\lambda}_2\overline{\alpha}_1 + \alpha_1\overline{\lambda}_1\ee[B_{21}]\big).
    \end{bmatrix}
\end{align}
We have thus found the stacked vector $\bm{\Psi}^{(1)}$.
These expressions could also have been derived by sending $t\to\infty$  in the expressions of the transient moments $\bm{\Psi}_t^{(0,1)}$ and $\bm{\Psi}_t^{(1,0)}$, respectively.

To identify the second order stationary moments, we again go over the steps of Algorithm~\ref{alg: recursive procedure stationary blocks}.
Step~0 yields
\begin{align} \label{eq: solution stationary moment lambda order 2, 2-dim}
\begin{split}
    \bm{\Psi}^{(0,2)} &=
    \begin{bmatrix}
    2\overline{\alpha}_1 & -2\ee[B_{12}] & 0 \\
    -\ee[B_{21}] & \overline{\alpha}_1 + \overline{\alpha}_2 & -\ee[B_{12}] \\
    0 & -2\ee[B_{21}] & 2\overline{\alpha}_2
    \end{bmatrix}^{-1} \\
    &\quad \times
    \begin{bmatrix}
    \ee[\lambda_1](2\alpha_1\overline{\lambda}_1 {+}\ee[B_{11}^2]) + \ee[\lambda_2]\ee[B_{12}^2] \\
    \ee[\lambda_1](\ee[B_{11}]\ee[B_{21}] + \alpha_2\overline{\lambda}_2) + \ee[\lambda_2]( \ee[B_{22}]\ee[B_{12}] +\alpha_1\overline{\lambda}_1) \\
    \ee[\lambda_1]\ee[B_{21}^2] + \ee[\lambda_2]( 2\alpha_2\overline{\lambda}_2 {+} \ee[B_{22}^2])
    \end{bmatrix}
\end{split}
\end{align}
where the inverse may be explicitly computed in specific cases.
For Step 1, we have, after some elementary matrix computations, that
\begin{align} \label{eq: solution stationary moment Q + lambda order 2, 2-dim}
    \bm{\Psi}^{(1,1)}&=
    \begin{bmatrix}
    \overline{\alpha}_1 +\mu_1 & -\ee[B_{12}] & 0 & 0 \\
    -\ee[B_{21}] & \overline{\alpha}_2 +\mu_1 & 0 & 0 \\
    0 & 0 & \overline{\alpha}_1 +\mu_2 & -\ee[B_{12}] \\
    0 & 0 & -\ee[B_{21}] & \overline{\alpha}_2 +\mu_2
    \end{bmatrix}^{-1}
    \begin{bmatrix}
    \ee[\lambda_1^2] + \alpha_1\overline{\lambda}_1\ee[Q_1] + \ee[B_{11}]\ee[\lambda_1] \\
    \ee[\lambda_1\lambda_2] + \alpha_2\overline{\lambda}_2\ee[Q_1] + \ee[B_{21}]\ee[\lambda_1] \\
    \ee[\lambda_1\lambda_2] + \alpha_1\overline{\lambda}_1\ee[Q_2] + \ee[B_{12}]\ee[\lambda_2] \\
    \ee[\lambda_2^2] + \alpha_2\overline{\lambda}_2\ee[Q_2] + \ee[B_{22}]\ee[\lambda_2]
    \end{bmatrix}.
\end{align}
Finally, for Step 2 we have
\begin{align} \label{eq: solution stationary moment Q order 2, 2-dim}
    \bm{\Psi}^{(2,0)} 
    =
    \begin{bmatrix}
    1/(2\mu_1) & 0 & 0 \\
    0 & 1/(\mu_1 + \mu_2) & 0 \\
    0 & 0 & 1/(2\mu_2)
    \end{bmatrix}
    \begin{bmatrix}
    2 & 0 & 0 & 0 \\
    0 & 1 & 1 & 0 \\
    0 & 0 & 0 & 2
    \end{bmatrix}
    \bm{\Psi}^{(1,1)}.
\end{align}

In line with earlier observations, the quasi-explicit results for the transient and stationary moments become involved for larger values of the order $n$.

\subsection{Higher order moments} \label{appendix: higher order moments 2-dim}

In this section, we provide some more detailed explicit matrices discussed Section~\ref{sec: nested block matrices}, for the bivariate $d=2$ setting for moments of order $n=3$.
For $n=3$ we have $\mathfrak{D}(3,2) = 20$, and we consider the stacked vector (size $34$)
\begin{align*}
    \big( \bm{\Psi}^{(3)}_t, \bm{\Psi}^{(2)}_t, \bm{\Psi}^{(1)}_t\big)^\top,
\end{align*}
which satisfies the ODE
\begin{align}
    \frac{\ddiff}{\ddiff t} 
    \begin{bmatrix}
    \bm{\Psi}^{(3)}_t \\
    \bm{\Psi}^{(2)}_t \\
    \bm{\Psi}^{(1)}_t
    \end{bmatrix}
    = \bm{A}_3^{35\times 35} \begin{bmatrix}
    \bm{\Psi}^{(3)}_t \\
    \bm{\Psi}^{(2)}_t \\
    \bm{\Psi}^{(1)}_t
    \end{bmatrix}
    +
    \begin{bmatrix}
    \bm{b}^{2\times 1} \\
    \bm{0}^{32 \times 1}
    \end{bmatrix}.
\end{align}
The matrix $\bm{F}_3^{34\times 34}$ is given by
\begin{align}
    \bm{F}_3^{34\times 34} = \begin{bmatrix}
    \bm{F}_2^{14\times 14} & \bm{0}^{14 \times 20} \\
    \bm{G}_3^{20 \times 14} & \bm{H}_3^{14 \times 14}
    \end{bmatrix},
\end{align}
where 
\begin{align*}
    \bm{H}_3^{14\times 14} = \begin{bmatrix}
    \bm{M}^{(3,0)} & \bm{0}^{4\times 6} & \bm{0}^{4\times 6} & \bm{0}^{4\times 4} \\
    \bm{K}^{(2,1)} & \bm{M}^{(2,1)} & \bm{0}^{6 \times 6} & \bm{0}^{6 \times 4} \\
    \bm{0}^{6 \times 4} & \bm{K}^{(1,2)} & \bm{M}^{(1,2)} & \bm{0}^{6 \times 4} \\
    \bm{0}^{4 \times 4} & \bm{0}^{4 \times 6} & \bm{K}^{(0,3)} & \bm{M}^{(0,3)}
    \end{bmatrix},
    \quad \bm{G}_3^{20\times 14}
    = \begin{bmatrix}
    \bm{L}^{(0,3)} \\
    \bm{L}^{(1,2)} \\
    \bm{L}^{(2,1)} \\
    \bm{0}^{4 \times 20}
    \end{bmatrix}.
\end{align*}
The elements of $\bm{G}_3^{20\times 20}$ require some more notation to describe, by introducing a number of sub-matrices.
First, we have $\bm{L}^{(0,3)} = \big[ \bm{L}_{\lambda^1}^{(0,3)} \:\: \bm{0}^{4\times 2} \:\: \bm{L}_{\lambda^2}^{(0,3)} \:\: \bm{0}^{4 \times 7} \big]$, with
\begin{align*}
    \bm{L}_{\lambda^1}^{(0,3)} &= \begin{bmatrix}
    \ee[B_{11}^3] & \ee[B_{12}^3] \\
    \ee[B_{11}^2]\ee[B_{21}] & \ee[B_{12}^2]\ee[B_{22}] \\
    \ee[B_{11}]\ee[B_{21}^2] & \ee[B_{12}]\ee[B_{22}^2] \\
    \ee[B_{21}^3] & \ee[B_{22}^3]
    \end{bmatrix}, \\
    \bm{L}_{\lambda^2}^{(0,3)} &= \begin{bmatrix}
    3\ee[B_{11}^2] + 3\alpha_1\overline{\lambda}_1 & 3\ee[B_{12}^2] & 0 \\
    2\ee[B_{11}]\ee[B_{21}] + \alpha_{2}\overline{\lambda}_2 & 2\ee[B_{12}]\ee[B_{22}] + \ee[B_{11}^2] + 2\alpha_1\overline{\lambda}_1 & \ee[B_{12}^2]  \\
    \ee[B_{21}^2] & 2\ee[B_{21}]\ee[B_{11}] + \ee[B_{22}^2] + 2\alpha_2\overline{\lambda}_2 & 2\ee[B_{12}]\ee[B_{22}] + \alpha_1\overline{\lambda}_1 \\
    0 & 3\ee[B_{21}^2] & 3\ee[B_{22}^2] + 3\alpha_2\overline{\lambda}_2
    \end{bmatrix}.
\end{align*}
Second, we have $\bm{L}^{(1,2)} = \big[\bm{L}_{\lambda^1}^{(1,2)} \:\: \bm{0}^{6 \times 2} \:\: \bm{L}_{\lambda^2}^{(1,2)} \:\: \bm{L}_{Q^1\lambda^1}^{(1,2)} \:\: \bm{0}^{6 \times 3} \big]$, where
\begin{align*}
    &\bm{L}_{\lambda^1}^{(1,2)} = \begin{bmatrix}
    \ee[B_{11}^2] & 0 \\
    \ee[B_{11}]\ee[B_{21}] & 0 \\
    \ee[B_{21}^2] & 0 \\
    0 & \ee[B_{12}^2] \\
    0 & \ee[B_{12}]\ee[B_{22}] \\
    0 & \ee[B_{22}^2]
    \end{bmatrix}, \quad 
    \bm{L}_{\lambda^2}^{(1,2)} = \begin{bmatrix}
    2\ee[B_{11}] & 0 & 0 \\
    \ee[B_{21}] & \ee[B_{11}] & 0 \\
    0 & 2\ee[B_{21} & 0 \\
    0 & 2\ee[B_{12}] & 0 \\
    0 & \ee[B_{22}] & \ee[B_{12}] \\
    0 & 0 & 2\ee[B_{22}]
    \end{bmatrix}, \quad \\
    \bm{L}_{Q^1L^1}^{(1,2)}
    &= \begin{bmatrix}
    \ee[B_{11}^2] + 2\alpha_1\overline{\lambda}_1 & \ee[B_{12}^2] & 0 & 0 \\
    \ee[B_{11}]\ee[B_{21}] + \alpha_2\overline{\lambda}_2 & \ee[B_{12}]\ee[B_{22}] + \alpha_1\overline{\lambda}_1 & 0 & 0 \\
    \ee[B_{21}^2] & \ee[B_{22}^2] + 2\alpha_2\overline{\lambda}_2 & 0 & 0 \\
    0 & 0 & \ee[B_{11}^2] + 2\alpha_1\overline{\lambda}_1 & \ee[B_{12}^2] \\
    0 & 0 & \ee[B_{11}]\ee[B_{21}] + \alpha_2\overline{\lambda}_2 & \ee[B_{12}]\ee[B_{22}] +
    \alpha_1\overline{\lambda}_1 \\
    0 & 0 & \ee[B_{21}^2] & \ee[B_{22}^2] + 2\alpha_2\overline{\lambda}_2  \\
    \end{bmatrix}.
\end{align*}
Finally, $\bm{L}^{(2,1)} = \big[ \bm{0}^{6\times 7} \:\: \bm{L}_{Q^1L^1}^{(2,1)} \:\: \bm{L}_{Q^2}^{(2,1)} \big]$, where
\begin{align*}
    \bm{L}_{Q^1L^1}^{(2,1)} &= \begin{bmatrix}
    2\ee[B_{11}] & 0 & 0 & 0 \\
    2\ee[B_{21}] & 0 & 0 & 0 \\
    0 & \ee[B_{12}] & \ee[B_{11}] & 0 \\
    0 & \ee[B_{22}] & \ee[B_{21}] & 0 \\
    0 & 0 & 0 & 2\ee[B_{12}] \\
    0 & 0 & 0 & 2\ee[B_{22}]
    \end{bmatrix},
    \quad \bm{L}_{Q^2}^{(2,1)} = \begin{bmatrix}
    \alpha_1\overline{\lambda}_1 & 0 & 0 \\
    \alpha_2\overline{\lambda}_2 & 0 & 0 \\
    0 & \alpha_1\overline{\lambda}_1 & 0 \\
    0 & \alpha_2\overline{\lambda}_2 & 0 \\
    0 & 0 & \alpha_1\overline{\lambda}_1 \\ 
    0 & 0 & \alpha_2\overline{\lambda}_2
    \end{bmatrix}.
\end{align*}

\section{The Nearly Unstable Behavior: Proofs} \label{appendix: nearly unstable}

\begin{proof}[Proof of Lemma \ref{lemma: near unstable}]
From the PDE in Eqn.~\eqref{eq: PDE zeta rewritten} and the assumptions imposed,
\begin{align*}
    \sum_{i=1}^d \big(\alpha s_i + \beta_i(\overline{s}) - 1\big) \frac{\ddiff}{\ddiff s_i} \cT\{\bm{\lambda}\}(\bm{s}) = - \alpha \overline{\lambda} \sum_{i=1}^d  s_i  \cT\{\bm{\lambda}\}(\bm{s}),
\end{align*}
upon substituting $z_1 = \cdots = z_d = 1$.
Further observe that for any $i,j\in[d]$, we have
\begin{align*}
    \frac{\ddiff}{\ddiff s_i} \cT\{\bm{\lambda}\}(\bm{s}) 
    = \ee\big[-\lambda_i e^{-\bm{s}^\top \bm{\lambda}}\big] 
    =  \ee\big[-\lambda_j e^{-\bm{s}^\top \bm{\lambda}}\big] 
    =\frac{\ddiff}{\ddiff s_j} \cT\{\bm{\lambda}\}(\bm{s}),
\end{align*}
since $\lambda_i \overset{d}{=} \lambda_j$, again by our assumptions on the parameters.
Hence, we obtain the ODE
\begin{align}
    \frac{\ddiff }{\ddiff s_1} \cT\{\bm{\lambda}\}(\bm{s})
    = \frac{ -\alpha \overline{\lambda} \overline{s}}{\alpha\overline{s} + \sum_{i=1}^d \beta_i(\overline{s}) -d} \cT\{\bm{\lambda}\}(\bm{s}) =: -f(s_1,\dots,s_d)) \cT\{\bm{\lambda}\}(\bm{s}),
\end{align}
where we note that the choice of the index in the left-hand side is arbitrary.
The solution to this ODE may be expressed as
\begin{align} \label{eq: log solution LT lambda steady-state}
    \log\big( \cT\{\bm{\lambda}\}(\bm{s})\big) = - \int_0^{s_1} f(u, s_2,\dots,s_d) \ddiff u + K, \quad K = \log(\cT\{\bm{\lambda}\}(0, s_2,\dots,s_d)).
\end{align}
Since $\lambda_i \overset{d}{=} \lambda_j$, the Laplace transforms of the marginals satisfy, for any $s\in\rr_+$, 
\begin{align*}
    \cT\{\bm{\lambda}\}(s,0,\dots,0) 
    = \cT\{\bm{\lambda}\}(0,s,0,\dots,0) = \cdots = \cT\{\bm{\lambda}\}(0,0,\dots,s).
\end{align*}
We are then able to derive the joint Laplace transform of, say, $(\lambda_1,\lambda_2)^\top$ from Eqn.~\eqref{eq: log solution LT lambda steady-state}:
\begin{align} \label{eq: procedure for solution LT lambda steady-state}
\begin{split}
    &\cT\{\bm{\lambda}\}(s_1,s_2,\dots,0) \\
    &= \exp\Big(- \int_0^{s_1} f(u, s_2,\dots,0) \ddiff u \Big) \cT\{\bm{\lambda}\}(0,s_2,\dots,0) \\
    &=\exp\Big(-\alpha\overline{\lambda} \int_0^{s_1} \frac{u + s_2}{\alpha(u+s_2) + \sum_{i=1}^d \beta_i(u+s_2) - d}\ddiff u\Big)\\
    &\quad \times \exp\Big(-\alpha \overline{\lambda} \int_0^{s_2} \frac{u}{\alpha u + \sum_{i=1}^d \beta_i(u) - d} \ddiff u \Big) \\
    &= \exp\Big(-\alpha\overline{\lambda} \big( \int_{s_2}^{s_1+s_2} \frac{v}{\alpha v + \sum_{i=1}^d \beta_i(v) - d}\ddiff v +\int_0^{s_2} \frac{u}{\alpha u + \sum_{i=1}^d \beta_i(u) - d} \ddiff u\big) \Big) \\
    &= \exp\Big(-\alpha\overline{\lambda} \int_{0}^{s_1+s_2} \frac{v}{\alpha v + \sum_{i=1}^d \beta_i(v) - d}\ddiff v\Big).
\end{split}
\end{align}
By symmetry, we can do this for any pair $(\lambda_i,\lambda_j)^\top$, with $i,j\in[d]$.
Iterating the derivation in Eqn.~\eqref{eq: procedure for solution LT lambda steady-state}, we obtain the full solution
\begin{align*}
    \cT\{\bm{\lambda}\}(\bm{s}) 
    &= \exp\Big(- \int_0^{s_1} f(u, s_2,\dots,s_d) \ddiff u \Big) \cT\{\bm{\lambda}\}(0,s_2,\dots,s_d) \\
    &= \exp\Big(-\alpha \overline{\lambda} \int_0^{s_1+\dots+s_d} \frac{u}{\alpha u + \sum_{i=1}^d \beta_i(u) - d} \ddiff u \Big),
\end{align*}
as desired.
\end{proof}

\begin{proof}[Proof of Theorem \ref{thm: heavy traffic lambda thm}]
The proof follows from the expression for $\cT\{\bm{\lambda}\}(\bm{s})$ in Eqn.~\eqref{eq: lemma statement solution LT lambda steady-state}, applying a Taylor expansion to the $\beta_i(\cdot)$, and computing the limit.
Since the second moments of $B_i$ exist, we have $\beta_i(u) = 1- u\ee[B_i] + \frac{1}{2}u^2\ee[B_i^2] + o(u^2)$ as $u\downarrow 0$.
Substituting $\bm{s}(1-\theta)$ as the argument in Eqn.~\eqref{eq: lemma statement solution LT lambda steady-state} and the Taylor expansion of $\beta_i(\cdot)$, we obtain as $\theta\uparrow 1$, that
\begin{align*}
    \cT\{\bm{\lambda}\}(\bm{s}(1-\theta))
    &= \exp\Big(-\alpha\overline{\lambda} \int_0^{\overline{s}} \frac{u(1-\theta)}{\alpha u(1-\theta) + \sum_{i=1}^d \beta_i(u(1-\theta)) - d} (1-\theta)\ddiff u\Big) \\
    &= \exp\Big(-\overline{\lambda} \int_0^{\overline{s}} \frac{\alpha(1-\theta)}{\alpha - \sum_{i=1}^d\ee[B_i] + \frac{u}{2}(1-\theta)\sum_{i=1}^d\ee[B_i^2] + o(1-\theta)}\ddiff u\Big) \\
    &=  \exp\Big(-\overline{\lambda} \int_0^{\overline{s}} \frac{1}{1 + \frac{u}{2\alpha}\sum_{i=1}^d\ee[B_i^2] + o(1)}\ddiff u\Big).
\end{align*}
Finally, by definition of $\sigma$ and an elementary computation, we have
\begin{align*}
    \lim_{\theta \uparrow 1} \cT\{\bm{\lambda}\}(\bm{s}(1-\theta))
    = \lim_{\theta \uparrow 1} \exp\Big(-\overline{\lambda} \int_0^{\overline{s}} \frac{1}{1 + u\sigma^{-1} + o(1)}\ddiff u\Big)
    =\Big( \frac{\sigma}{\sigma + \overline{s}}\Big)^{\sigma\overline{\lambda}},
\end{align*} as claimed.
\end{proof}

\section{Experiments with Transient Moments}\label{appendix: numerics}
This appendix demonstrates the numerical evaluation of the time-dependent moments featured in Section~\ref{sec: characterization}.
From the joint transform characterization in Theorem~\ref{thm: joint transform characterization t, t+tau}, we compute the mixed moments, for any $t\geqslant 0$, $\tau>0$ and any combination $i,j=1,2$, 
\begin{align} \label{eq: Q ttau and L ttau expression}
    \ee[Q_i(t)\,Q_j(t+\tau)], \quad 
    \ee[\lambda_i(t)\,\lambda_j(t+\tau)],
\end{align}
as before 
using finite differences. It is noted that along the same lines objects of the type $\ee[Q_i(t)\,\lambda_j(t+\tau)]$ can be evaluated, and in addition various types of auto-covariances and auto-correlations (cf.\ \cite{DFZ15} for the auto-covariance in the market micro-structure setting).
The joint transform characterization allows for efficient and fast computation of these cross-moments, also for large $t>0$, which makes it practical in these settings.

To assess the efficiency and precision, we have conducted a numerical experiment with the same parameters as earlier in this subsection.
To analyze the effect of the $\tau$ parameter in~\eqref{eq: Q ttau and L ttau expression}, in Figure~\ref{fig: F3} we fix $t = 1.5$ and plot the quantities of interest as functions of $\tau$.
The solid lines are the moments computed by applying FD to the joint transform, and the dotted lines represent the results from the MC method (based on $10^4$ runs).
We see that MC performs increasingly poorly as $\tau$ increases, in particular for the population processes $Q_i(\cdot)$, which is due to the fact that there are more events (i.e., arrivals and departures) for larger $\tau$, and hence more variation.
Furthermore, the different shapes in the plots indicate that the effect of $\tau$ on the specific cross-moment depends on the chosen parameters.

\begin{figure}
\centering
\resizebox{15.9cm}{5.5cm}{
\pgfplotstableread{F3b.txt}{\table}
    \begin{tikzpicture}
        \begin{axis}[
            xmin = 0, xmax = 10,
            ymin = 1.5, ymax = 4.9,
            xtick distance = 2,
            ytick distance = 1,
            grid = both,
            minor tick num = 1,
            major grid style = {lightgray},
            minor grid style = {lightgray!25},
            width = \textwidth,
            height = 0.75\textwidth,
            legend cell align = {left},
            legend pos = north west
        ]
        
            \addplot[blue, mark = *, mark size = 0.3pt, line width = 1pt] table [x = {x}, y = {Q1Q1_tau}] {\table};
            \addlegendentry{\normalsize ${\mathbb E}[Q_1(t)Q_1(t+\tau)]$}
            \addplot[blue, dashed, mark = *, mark size = 0.3pt, line width = 1pt] table [x = {x}, y = {Q1Q1_tau_sim}] {\table};
            \addlegendentry{\normalsize ${\mathbb E}[Q_1(t)Q_1(t+\tau)]$, sim.}
            \addplot[magenta, mark = *, mark size = 0.3pt, line width = 1pt] table [x = {x}, y = {Q1Q2_tau}] {\table};
            \addlegendentry{\normalsize ${\mathbb E}[Q_1(t)Q_2(t+\tau)]$}
            \addplot[magenta, dashed, mark = *, mark size = 0.3pt, line width = 1pt] table [x = {x}, y = {Q1Q2_tau_sim}] {\table};
            \addlegendentry{\normalsize ${\mathbb E}[Q_1(t)Q_2(t+\tau)]$, sim.}
        \end{axis}
\end{tikzpicture}

 \pgfplotstableread{F3a.txt}{\table}
    \begin{tikzpicture}
        \begin{axis}[
            xmin = 0, xmax = 10,
            ymin = 1.5, ymax = 3.5,
            xtick distance = 2,
            ytick distance = 1,
            grid = both,
            minor tick num = 1,
            major grid style = {lightgray},
            minor grid style = {lightgray!25},
            width = \textwidth,
            height = 0.75\textwidth,
            legend cell align = {left},
            legend pos = north west
        ]
        
            \addplot[blue, mark = *, mark size = 0.3pt, line width = 1pt] table [x = {x}, y = {Q2Q2_tau}] {\table};
            \addlegendentry{\normalsize ${\mathbb E}[Q_2(t)Q_2(t+\tau)]$}
            \addplot[blue, dashed, mark = *, mark size = 0.3pt, line width = 1pt] table [x = {x}, y = {Q2Q2_tau_sim}] {\table};
            \addlegendentry{\normalsize ${\mathbb E}[Q_2(t)Q_2(t+\tau)]$, sim.}
            \addplot[magenta, mark = *, mark size = 0.3pt, line width = 1pt] table [x = {x}, y = {Q2Q1_tau}] {\table};
            \addlegendentry{\normalsize ${\mathbb E}[Q_2(t)Q_1(t+\tau)]$}
            \addplot[magenta, dashed, mark = *, mark size = 0.3pt, line width = 1pt] table [x = {x}, y = {Q2Q1_tau_sim}] {\table};
            \addlegendentry{\normalsize ${\mathbb E}[Q_2(t)Q_1(t+\tau)]$, sim.}
        \end{axis}
\end{tikzpicture}

}

\centering
\resizebox{15.9cm}{5.5cm}{
\pgfplotstableread{F3d.txt}{\table}
    \begin{tikzpicture}
        \begin{axis}[
            xmin = 0, xmax = 10,
            ymin = 3, ymax = 8,
            xtick distance = 2,
            ytick distance = 1,
            grid = both,
            minor tick num = 1,
            major grid style = {lightgray},
            minor grid style = {lightgray!25},
            width = \textwidth,
            height = 0.75\textwidth,
            legend cell align = {left},
            legend pos = north west
        ]
        
            \addplot[blue, mark = *, mark size = 0.3pt, line width = 1pt] table [x = {x}, y = {L1L1_tau}] {\table};
            \addlegendentry{\normalsize ${\mathbb E}[L_1(t)L_1(t+\tau)]$}
            \addplot[blue, dashed, mark = *, mark size = 0.3pt, line width = 1pt] table [x = {x}, y = {L1L1_tau_sim}] {\table};
            \addlegendentry{\normalsize ${\mathbb E}[L_1(t)L_1(t+\tau)]$, sim.}
            \addplot[magenta, mark = *, mark size = 0.3pt, line width = 1pt] table [x = {x}, y = {L1L2_tau}] {\table};
            \addlegendentry{\normalsize ${\mathbb E}[L_1(t)L_2(t+\tau)]$}
            \addplot[magenta, dashed, mark = *, mark size = 0.3pt, line width = 1pt] table [x = {x}, y = {L1L2_tau_sim}] {\table};
            \addlegendentry{\normalsize ${\mathbb E}[L_1(t)L_2(t+\tau)]$, sim.}
        \end{axis}
\end{tikzpicture}

 \pgfplotstableread{F3c.txt}{\table}
    \begin{tikzpicture}
        \begin{axis}[
            xmin = 0, xmax = 10,
            ymin = 4, ymax = 14,
            xtick distance = 2,
            ytick distance = 2,
            grid = both,
            minor tick num = 1,
            major grid style = {lightgray},
            minor grid style = {lightgray!25},
            width = \textwidth,
            height = 0.75\textwidth,
            legend cell align = {left},
            legend pos = north west
        ]
        
            \addplot[blue, mark = *, mark size = 0.3pt, line width = 1pt] table [x = {x}, y = {L2L2_tau}] {\table};
            \addlegendentry{\normalsize ${\mathbb E}[\lambda_2(t)\lambda_2(t+\tau)]$}
            \addplot[blue, dashed, mark = *, mark size = 0.3pt, line width = 1pt] table [x = {x}, y = {L2L2_tau_sim}] {\table};
            \addlegendentry{\normalsize ${\mathbb E}[\lambda_2(t)\lambda_2(t+\tau)]$, sim.}
            \addplot[magenta, mark = *, mark size = 0.3pt, line width = 1pt] table [x = {x}, y = {L2L1_tau}] {\table};
            \addlegendentry{\normalsize ${\mathbb E}[\lambda_2(t)\lambda_1(t+\tau)]$}
            \addplot[magenta, dashed, mark = *, mark size = 0.3pt, line width = 1pt] table [x = {x}, y = {L2L1_tau_sim}] {\table};
            \addlegendentry{\normalsize ${\mathbb E}[\lambda_2(t)\lambda_1(t+\tau)]$, sim.}
        \end{axis}
\end{tikzpicture}

}

\caption{Computation of cross-moments for $t=1.5$ and $\tau \in[0,10]$ using the joint transform characterization (solid lines) compared to Monte Carlo simulated averages (dashed lines).}
\label{fig: F3}
\end{figure}
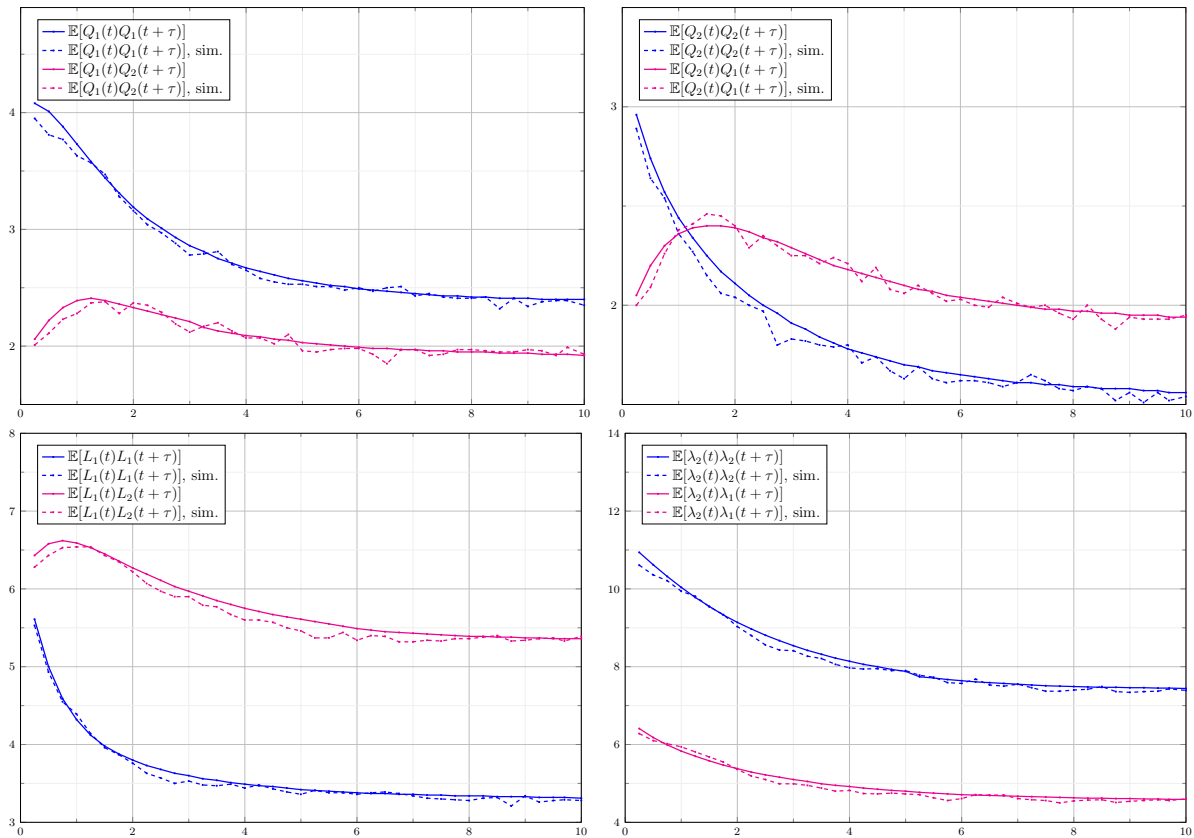

We conclude our numerical section by an experiment that quantifies the effect of the initial values.
In Theorem~\ref{thm: zeta characterization}, we characterized the joint transform with the processes being initialized at $\bm{Q}(t_0) = (Q_1(t_0),Q_2(t_0)) = (q_{1,0},q_{2,0}) \in \nn^2$ and $\bm{\lambda}(t_0) = (\lambda_1(t_0),\lambda_2(t_0)) = (\lambda_{1,0}, \lambda_{2,0}) \in \rr_+^2$ for some $t_0>0$.
By applying FD, we can compute the moments of our interest for any initial value, for instance
\begin{align}
    \ee[Q_i(t) \, | \,Q_i(t_0) = q_{i,0}], \quad \ee[\lambda_i(t) \, | \,  \lambda_i(t_0) = \lambda_{i,0}],
\end{align}
with $i=1,2$, where $q_{i,0} \in \nn$ and $\lambda_{i,0}\in\rr_+$.
In our experiment we focus on the moments of $Q_1(\cdot)$ and $\lambda_1(\cdot)$, studying the effect of three different choices of $q_{i,0}$ and $\lambda_{i,0}$ on the first moments and variances as a function of $t$.
Note that a different value of $Q_i(t_0)= q_{i,0}$ will not influence $\lambda_j(\cdot)$, since the population processes do not directly affect the intensity processes, but due to mutual excitation, the values $\lambda_1(t_0) = \lambda_{1,0}$ and $\lambda_2(t_0) = \lambda_{2,0}$ do matter.
When computing $\ee[Q_1(t) | Q_1(t_0) = q_{1,0}]$, we leave $\lambda_i(t_0) = \lambda_i(0) = \overline{\lambda}_i$ for $i=1,2$, and we only change $q_{1,0}$.
Similarly, when computing $\ee[\lambda_1(t) | \lambda_1(t_0) = \lambda_{1,0}]$, we leave $\lambda_2(t_0) = \lambda_2(0) = \overline{\lambda}_2$ and only change $\lambda_{1,0}$.

Figure~\ref{fig: F4} shows the expectations and variances, where we introduced the compact notation $\ee[Q_1(t) \,|\, q_{1,0}] = \ee[Q_1(t)\, | \,Q_1(t_0)= q_{1,0}]$ and $\ee[\lambda_1(t) |\lambda_{1,0}] = \ee[\lambda_1(t) | \lambda_1(t_0) = \lambda_{1,0}]$, and similar notation for the variances.
For the moments of $Q_1(t)$, we observe a vertical shift of the plots, which is expected since the arrived individuals depart independently and according to the same distribution.
For the moments of $\lambda_1(t)$, we see that the effect of the $\lambda_{1,0}$-value is substantial. 
For both the mean and the variance there is convergence to their respective steady-state values.

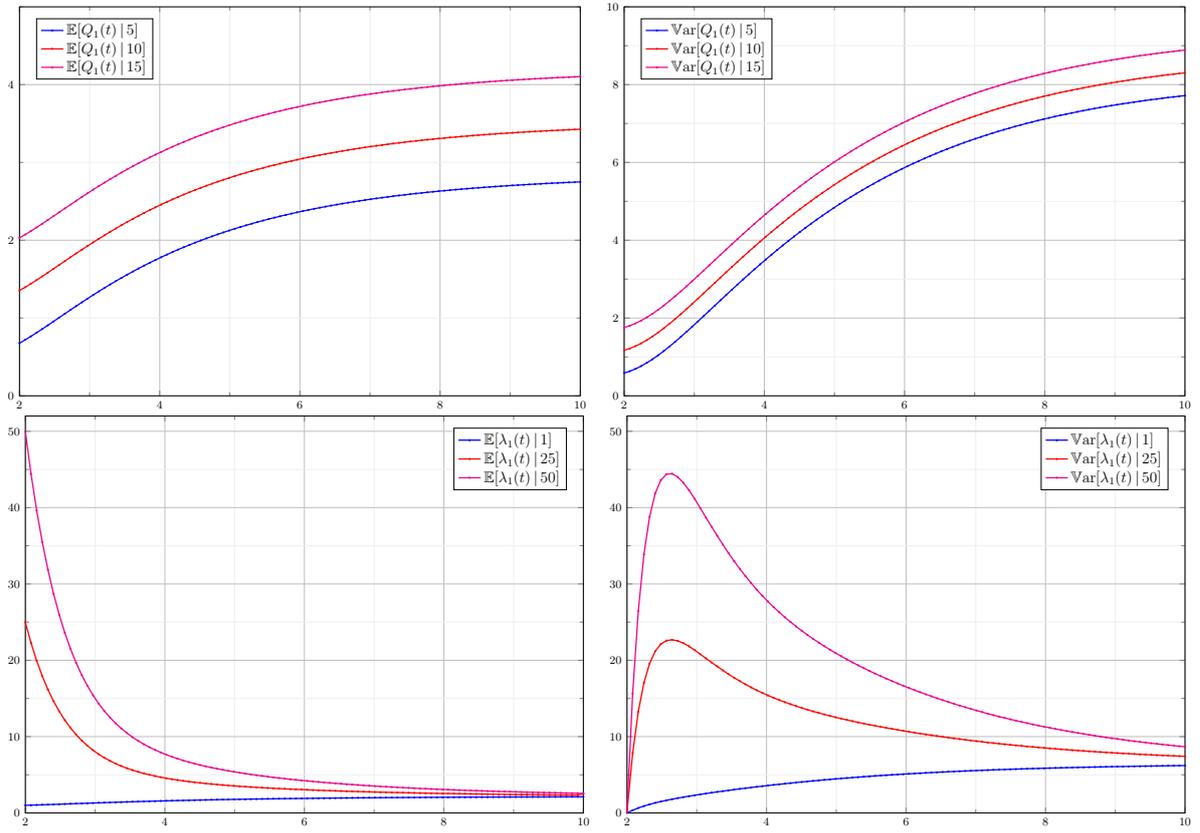
\begin{figure}
\centering
\resizebox{15.9cm}{5.5cm}{
\pgfplotstableread{F4a.txt}{\table}
    \begin{tikzpicture}
        \begin{axis}[
            xmin = 2, xmax = 10,
            ymin = 0, ymax = 5,
            xtick distance = 2,
            ytick distance = 2,
            grid = both,
            minor tick num = 1,
            major grid style = {lightgray},
            minor grid style = {lightgray!25},
            width = \textwidth,
            height = 0.75\textwidth,
            legend cell align = {left},
            legend pos = north west
        ]
        
            \addplot[blue, mark = *, mark size = 0.3pt, line width = 1pt] table [x = {x}, y = {a}] {\table};
            \addlegendentry{\normalsize ${\mathbb E}[Q_1(t) \,|\, 5]$}
            \addplot[red, mark = *, mark size = 0.3pt, line width = 1pt] table [x = {x}, y = {b}] {\table};
            \addlegendentry{\normalsize ${\mathbb E}[Q_1(t) \,|\, 10]$}
            \addplot[magenta, mark = *, mark size = 0.3pt, line width = 1pt] table [x = {x}, y = {c}] {\table};
            \addlegendentry{\normalsize ${\mathbb E}[Q_1(t) \,|\, 15]$}
        \end{axis}
\end{tikzpicture}

\pgfplotstableread{F4b.txt}{\table}
    \begin{tikzpicture}
        \begin{axis}[
            xmin = 2, xmax = 10,
            ymin = 0, ymax = 10,
            xtick distance = 2,
            ytick distance = 2,
            grid = both,
            minor tick num = 1,
            major grid style = {lightgray},
            minor grid style = {lightgray!25},
            width = \textwidth,
            height = 0.75\textwidth,
            legend cell align = {left},
            legend pos = north west
        ]
        
            \addplot[blue, mark = *, mark size = 0.3pt, line width = 1pt] table [x = {x}, y = {a}] {\table};
            \addlegendentry{\normalsize ${\mathbb V}{\rm ar}[Q_1(t) \,|\, 5]$}
            \addplot[red, mark = *, mark size = 0.3pt, line width = 1pt] table [x = {x}, y = {b}] {\table};
            \addlegendentry{\normalsize ${\mathbb V}{\rm ar}[Q_1(t) \,|\, 10]$}
            \addplot[magenta, mark = *, mark size = 0.3pt, line width = 1pt] table [x = {x}, y = {c}] {\table};
            \addlegendentry{\normalsize ${\mathbb V}{\rm ar}[Q_1(t) \,|\, 15]$}
        \end{axis}
\end{tikzpicture}

}

\centering
\resizebox{15.9cm}{5.5cm}{

\pgfplotstableread{F4c.txt}{\table}
    \begin{tikzpicture}
        \begin{axis}[
            xmin = 2, xmax = 10,
            ymin = 0, ymax = 52,
            xtick distance = 2,
            ytick distance = 10,
            grid = both,
            minor tick num = 1,
            major grid style = {lightgray},
            minor grid style = {lightgray!25},
            width = \textwidth,
            height = 0.75\textwidth,
            legend cell align = {left},
            legend pos = north east
        ]
        
            \addplot[blue, mark = *, mark size = 0.3pt, line width = 1pt] table [x = {x}, y = {a}] {\table};
            \addlegendentry{\normalsize ${\mathbb E}[\lambda_1(t) \,|\, 1]$}
            \addplot[red, mark = *, mark size = 0.3pt, line width = 1pt] table [x = {x}, y = {b}] {\table};
            \addlegendentry{\normalsize ${\mathbb E}[\lambda_1(t) \,|\, 25]$}
            \addplot[magenta, mark = *, mark size = 0.3pt, line width = 1pt] table [x = {x}, y = {c}] {\table};
            \addlegendentry{\normalsize ${\mathbb E}[\lambda_1(t) \,|\, 50]$}
        \end{axis}
\end{tikzpicture}

\pgfplotstableread{F4d.txt}{\table}
    \begin{tikzpicture}
        \begin{axis}[
            xmin = 2, xmax = 10,
            ymin = 0, ymax = 52,
            xtick distance = 2,
            ytick distance = 10,
            grid = both,
            minor tick num = 1,
            major grid style = {lightgray},
            minor grid style = {lightgray!25},
            width = \textwidth,
            height = 0.75\textwidth,
            legend cell align = {left},
            legend pos = north east
        ]
        
            \addplot[blue, mark = *, mark size = 0.3pt, line width = 1pt] table [x = {x}, y = {a}] {\table};
            \addlegendentry{\normalsize ${\mathbb V}{\rm ar}[\lambda_1(t) \,|\, 1]$}
            \addplot[red, mark = *, mark size = 0.3pt, line width = 1pt] table [x = {x}, y = {b}] {\table};
            \addlegendentry{\normalsize ${\mathbb V}{\rm ar}[\lambda_1(t) \,|\, 25]$}
            \addplot[magenta, mark = *, mark size = 0.3pt, line width = 1pt] table [x = {x}, y = {c}] {\table};
            \addlegendentry{\normalsize ${\mathbb V}{\rm ar}[\lambda_1(t) \,|\, 50]$}
        \end{axis}
\end{tikzpicture}

}

\caption{Computation of expected values and variances of $Q_1(t)$ and $\lambda_1(t)$, with different initial values at $t_0=2$, with $t \in [t_0,10]$, using the joint transform characterization.}
\label{fig: F4}
\end{figure}

\end{document}